\documentclass[12pt, oneside, reqno]{amsart}

\usepackage[T1]{fontenc}
\usepackage[french]{babel}
\usepackage{tikz-cd}
\usepackage{graphicx}
\usepackage{framed}
\usepackage[normalem]{ulem}
\usepackage{amsthm}
\usepackage{amssymb}
\usepackage{amsfonts}
\usepackage{enumerate}
\usepackage[utf8]{inputenc}
\usepackage{lmodern}
\usepackage[left=3cm,right=3cm,top=3cm,bottom=3cm]{geometry}
\usepackage{mathrsfs}
\usepackage{enumitem}
\usepackage{mathtools}
\usepackage{mathrsfs}
\usepackage{tikz}
\usepackage{verbatim}
\usepackage{amsmath}
\usepackage{hyperref}
\hypersetup{
    colorlinks=true,
    urlcolor=magenta,
    linkcolor=magenta,
    citecolor=magenta,
    breaklinks=true
}
\usepackage{stmaryrd}
\usepackage{blkarray, bigstrut} 
\usepackage{nicematrix}

\numberwithin{equation}{section}

\newtheorem{theorem}{Théoreme}[section]
\newtheorem{definition}[theorem]{Définition}
\newtheorem{lemma}[theorem]{Lemme}
\newtheorem{coro}[theorem]{Corollaire}
\newtheorem{prop}[theorem]{Proposition}
\newtheorem{Conjecture}[theorem]{Conjecture}
\newtheorem{IntroTheorem}[theorem]{Théorème}

\theoremstyle{definition}
\newtheorem{example}[theorem]{Exemple}

\newtheorem{remark}[theorem]{Remarque}

\newcommand{\Gr}{\mathrm{Gr}} 
\newcommand{\OK}{\mathcal{O}} 
\newcommand{\Z}{\mathbb{Z}} 
 
\newcommand{\Q}{\mathbb{Q}} 

\newcommand{\spec}{\mathrm{Spec}} 
\newcommand{\Split}{M^{\mathrm{PR}}}
\newcommand{\Naif}{M^{\mathrm{PEL}}}
\newcommand{\Splitt}{\overline{M}^{\mathrm{PR}}}
\newcommand{\Naiff}{\overline{M}^{\mathrm{PEL}}}
\newcommand{\kt}{[\![u]\!]}
\newcommand{\kT}{(\!(u)\!)}
\newcommand{\Ltilde}{\Lambda_0}
\newcommand{\SSplit}{\mathcal{S}h^{\mathrm{PR}}}
\newcommand{\NNaif}{\mathcal{S}h^{\mathrm{PEL}}}

\newcommand{\NNNaif}{\mathcal{S}h^{\mathrm{PEL},\square}}
\newcommand{\Criss}{\mathrm{Cris}(S/\Sigma)}
\newcommand{\pT}{^{(p)}}

\title{Stratification des variétés de Hilbert en présence de ramification}
\author{Diego Berger}
\date{}

\begin{document}
\maketitle
\begin{abstract}
Dans ce papier on étudie la géométrie de la fibre spéciale des modèles de Pappas-Rapoport des variétés de Shimura dans le cas Hilbert. Plus précisément on prouve que la stratification induite par le polygone de Hodge est une bonne stratification, ce qui est faux dans le cas Hilbert-Siegel pour $g\geq 2$. Ensuite on utilise des résultats connus sur le produit de convolution de grassmanniennes affines pour décrire la dimension des strates et obtenir que le morphisme d'oubli vers le modèle PEL de Kottwitz est plat en restriction aux strates de Hodge.
\end{abstract}
\tableofcontents

\section{Introduction}

\subsection{Motivations}
L'étude de la géométrie de la réduction modulo $p$ des variétés de Shimura est un vaste sujet qui a eu de nombreuses conséquences arithmétiques. Dans cet article nous nous concentrons sur les variétés modulaires de Hilbert. Ce sont des variétés sur  $\Q_p$ qui peuvent être vues comme des espaces de modules de variétés abéliennes munies de certaines données.  On dispose de plusieurs modèles sur $\Z_p$ de ces variétés. Le modèle de Kottwitz, noté $\NNaif$, est défini comme un prolongement sur $\Z_p$ du problème de module initial. Plus précisément $\NNaif$ est défini comme l'espace de module des quadruplés $(A,\lambda,\iota,\kappa)$ où $A$ est une variété abélienne de dimension $d>0$, $\lambda:A\rightarrow A^t$ est une $\Z_{(p)}^{\times}$-polarisation, $\iota:\OK_F\rightarrow \mathrm{End}(A)$ une action d'un anneau  d'entiers d'un corps de nombre totalement réel $F/\Q$ sur $A$ tel que $[F:\Q]=d$, et $\kappa$ une structure de niveau en dehors de $p$.  Lorsque l'extension $F/\Q$ est ramifiée, le modèle $\NNaif$ n'est plus lisse sur $\spec\, \Z_p$. Pappas et Rapoport ont défini dans \cite{PRII} un modèle lisse $\SSplit$ sur $\spec\, \OK_K$ où $K$ est une extension contenant les clôtures galoisiennes de toutes les extensions $F_v$ où $v|p$. Ce modèle est défini via une résolution des singularités du modèle local. C'est un espace de module classifiant les $5$-uplés $(A,\lambda,\iota,\kappa,\mathrm{Fil}(\omega))$ où $(A,\lambda,\iota,\kappa)$ est un quadruplé de $\NNaif$ et $\mathrm{Fil}(\omega)$ est une filtration du faisceau conormal satisfaisant certaines propriétés (voir \ref{DefPRData} pour plus de détails). L'un des objectifs dans l'étude de la géométrie de la réduction modulo $p$ des variétés de Shimura  est de définir des stratifications, dont les strates possèdent de bonnes propriétés. On peut trouver dans la littérature de nombreux travaux sur les stratifications de la réduction modulo $p$ du modèle $\NNaif$  (voir par exemple \cite{ViehmannWedhorn}). Concernant le modèle $\SSplit$ plusieurs stratifications ont été définies dans \cite{DiamondKassei}, \cite{ReduzziXiao}. Dans cet article nous allons nous  intéresser à la stratification de Hodge de $\SSplit$. Celle-ci est définie via le polygone de Hodge du faisceau conormal. 

\subsection{Principaux résultats} 
Pour ne pas alourdir les notations nous allons énoncer les principaux résultats de cet article dans le cadre d'une extension $L/\Q_p$ totalement ramifiée (l'extension $L/\Q_p$ jouant le rôle de $F_v/\Q_p$). Mis à part le Théorème \ref{TheoremC}, les résultats qui suivent sont valables et sont démontrés dans le cadre générale d'une extension totalement réelle $F/\Q$ sans condition sur la ramification en $p$.\\

Soit $L/\Q_p$ une extension totalement ramifiée de degré $e$. On note $\NNaif$ la fibre spéciale du modèle de Kottwitz et $\SSplit$ celle du modèle de Pappas-Rapoport des variétés modulaires de Hilbert (voir les Définitions \ref{SPEL}, \ref{SPR}). L'oubli de la filtration induit un morphisme :
\begin{equation*}
\pi :\SSplit\longrightarrow \NNaif
\end{equation*}
On dispose d'une stratification appelée stratification de Kottwitz-Rapoport (KR) du modèle PEL :
\begin{equation*}
\NNaif=\coprod_{\lambda\in\mathrm{Adm}(\mu)_K}\NNaif_{\lambda}
\end{equation*}
où $\mathrm{Adm}(\mu)_K$ désigne l'ensemble  $\mu$-admissible (voir la Remarque \ref{Admissible}). Dans le cas Hilbert la stratification (KR) de $\NNaif$ coïncide avec la stratification par le polygone de Hodge (voir Remarque \ref{Hodge=KR}). Il est naturel de se demander si cette stratification induit une (bonne) stratification (voir Définition \ref{bonnestratification}) de $\SSplit $. On prouve dans cet article que la réponse est oui :
\begin{IntroTheorem}\
\begin{enumerate}
\item Les strates $(\SSplit_{\lambda})_{\lambda\in\mathrm{Adm}(\mu)_K}$ forment une bonne stratification de $\SSplit$. Autrement dit pour tout $\lambda\in\mathrm{Adm}(\mu)_K$ on a la relation d'adhérence
\begin{equation*}
\overline{\SSplit_{\lambda}}=\bigcup_{\lambda'\leq\lambda}\SSplit_{\lambda'}
\end{equation*}
\item Pour tout $\lambda\in\mathrm{Adm}(\mu)_K$ la strate $\SSplit_{\lambda}$ est quasi-projective lisse de dimension $\langle \rho,|\mu_{\bullet}|+\lambda\rangle$.
\end{enumerate}
\end{IntroTheorem}
Ce théorème est surprenant car il est faux dans des cas plus généraux que le cas Hilbert (voir \cite{BijHer2}). La preuve du point $(2)$ repose sur les travaux de Haines (\cite{Haines}) sur la géométrie du morphisme de convolution dans le cadre des Grassmanniennes affines. Ses travaux nous permettent également de prouver le théorème suivant qui concerne la géométrie du morphisme $\pi$ :
\begin{theorem}
Pour tout $\lambda\in\mathrm{Adm}(\mu)_K$ la restriction
\begin{equation*}
\pi:\SSplit_{\lambda}\rightarrow\NNaif_{\lambda}
\end{equation*}
est un morphisme plat. 
\end{theorem} 
Les travaux de Haines n'ayant pas de restriction sur le groupe réductif, le théorème ci-dessus est en fait valable dans le cas Hilbert-Siegel plus général.\\

Dans \cite{ReduzziXiao} les auteurs ont défini des invariants de Hasse partiels $(m_i)_{1\leq i\leq e}$ et dans \cite{DiamondKassei} il est prouvé que les sous schémas localement fermés définis comme les lieux d'annulation de ces sections définissent une bonne stratification de $\SSplit$ :
\begin{equation*}
\SSplit=\coprod_{T\subset\{1,\dots,e\}} \SSplit_T
\end{equation*}
où $\SSplit_T$ est le sous schéma localement fermé défini par l'annulation des $(m_i)_{i\in T}$ et l'inversibilité des $(m_i)_{i\notin T}$. Il est alors naturel d'étudier l’interaction entre la stratification par le polygone de Hodge, et celle définit par ces invariants de Hasse partiels. Pour tout $\lambda\in\mathrm{Adm}(\mu)$ et tout $T\subset\{1,\dots,e\}$ on note $\SSplit_{(\lambda,T)}$ l'intersection de la strate $\SSplit_{\lambda}$ et de la strate $\SSplit_T$. On montre le résultat suivant dans le cas $e=4$ :
\begin{theorem}
\label{TheoremC}
Pour $L/\Q_p$ totalement ramifiée de degré $e=4$ la stratification :
\begin{equation*}
\SSplit=\coprod_{(\lambda,T)\in\mathscr{A}_{\mu_{\bullet}}}\SSplit_{(\lambda,T)}
\end{equation*}
est une bonne stratification où les relations d'adhérences sont données par la relation d'ordre naïve  sur $\mathscr{A}_{\mu_{\bullet}}$.
\end{theorem}
Dans le théorème ci-dessus l'ensemble $\mathscr{A}_{\mu_{\bullet}}$ désigne le sous ensemble des couples $(\lambda,T)$ tels que la strate $\SSplit_{(\lambda,T)}$ soit non vide. Il semble difficile de calculer ce sous ensemble en général. \\

\ \\

\textbf{Remerciements.} Je tiens à remercier Stéphane Bijakowski de m'avoir encouragé à écrire cet article et de m'avoir expliqué les résultats de déformations de la première section. Je tiens également à remercier Benoit Stroh, Thibault Alexandre et Arnaud Eteve pour toutes les discussions qui m'ont aidé à écrire cet article. Enfin je tiens à remercier Xu Shen pour ses commentaires et remarques.

\newpage

\section{Préliminaires}
Dans tout ce qui suit on fixe un nombre premier $p>0$.\ \\

Cette première partie est composée essentiellement de rappels sur la théorie de Dieudonné cristalline.  Nous présentons tous les outils nécessaires pour pouvoir déformer un groupe $p$-divisible le long d'un morphisme $k[\![t]\!]\rightarrow k$ où $k$ est un corps de caractéristique $p$.  Nous aurons besoin de la théorie de Dieudonné cristalline \cite{BBM}, de la théorie des display de Zink et Lau \cite{Zink},  \cite{Lau2},  \cite{Lau_2014} et de Grothendieck-Messing \cite{Messing1972}. En fait, nous pourrions nous contenter de la théorie des display car elle englobe les anneaux locaux complet de corps résiduels parfait comme $k^{\mathrm{perf}}[\![t]\!]$ (où $k^{\mathrm{perf}}$ désigne la perfection de $k$), ce qui nous suffit amplement pour nos problèmes de déformations.

\subsection{Théorie de Dieudonné cristalline}

\subsubsection{}Soit $k$ un corps parfait de caractéristique $p$ et $S$ un schéma sur $k$. Soit $A\rightarrow S$ un schéma abélien. On note $G=A[p^{\infty}]$ le groupe $p$-divisible associé. On note $\Sigma=\spec\, (W(k))$, et $\underline{G}$ le faisceau sur le site cristallin $\Criss$ induit par $G$. En suivant les notations de \cite{BBM} on note :
\begin{equation*}
\mathscr{E}(G):=\mathscr{E}\mathrm{xt}^1_{S/\Sigma}(\underline{G},\mathcal{O}_{S/\Sigma})
\end{equation*}
le cristal de Dieudonné contravariant de $G$. C'est un $\mathcal{O}_{S/\Sigma}$-cristal localement libre de rang $h$ où $h$ désigne la hauteur de $G$ (\cite{BBM}, Corollaire 1.4.7).  Notez qu'avec cette convention, les applications :
\begin{equation*}
F:\mathscr{E}(G)^{(p)}\rightarrow\mathscr{E}(G),\ \ \ \ \ \ \  V:\mathscr{E}(G)\rightarrow\mathscr{E}(G)\pT
\end{equation*}
(où $(\,\cdot\, )^{(p)}$ désigne le twist par le Frobenius) sont induites respectivement par :
\begin{equation*}
F:G\rightarrow G\pT,\ \ \ \ \ \ \  V:G\pT\rightarrow G
\end{equation*}
Si l'on évalue ce cristal sur l'épaississement $(S\xrightarrow{\mathrm{id}} S)$ on obtient une filtration de $\mathcal{O}_S$-modules localement libres (\cite{BBM}, Corollaire 3.3.5)
\begin{equation}
\label{EqHodgeFil}
0\longrightarrow \omega_{G}\longrightarrow\mathscr{E}(G)_{(S\xrightarrow{\mathrm{id}} S)}\longrightarrow \omega_{G^D}^\vee\longrightarrow 0
\end{equation}
appelée filtration de Hodge.  D'après \cite{BBM} Proposition 3.3.7 on dispose d'un isomorphisme 
\begin{equation*}
\mathscr{E}\mathrm{xt}^1_{S/\Sigma}(\underline{
A},\mathcal{O}_{S/\Sigma})\simeq\mathscr{E}\mathrm{xt}^1_{S/\Sigma}(\underline{G},\mathcal{O}_{S/\Sigma})
\end{equation*}
reliant le cristal de $A/S$ et celui de son groupe $p$-divisible, compatible aux filtrations de Hodge respectives. En combinant maintenant avec l'isomorphisme (\cite{BBM}, Proposition 2.5.8) :
\begin{equation*}
\mathscr{E}\mathrm{xt}^1_{S/\Sigma}(\underline{
A},\mathcal{O}_{S/\Sigma})_{(S\rightarrow S)}\simeq H^1_{\mathrm{dR}}(A/S)
\end{equation*}
on retrouve la filtration de Hodge induit par la suite spectrale de Hodge bien connue :
\begin{equation*}
0\longrightarrow \omega_A\longrightarrow H^1_{\mathrm{dR}}(A/S)\longrightarrow \omega_{A^t}^{\vee}\longrightarrow 0
\end{equation*}

\subsubsection{}
\label{omegaImV} Plaçons nous maintenant dans le cas où $S=\spec(k)$ est le spectre d'un corps $k$ parfait de caractéristique $p$. On note :
\begin{equation*}
\mathbb{D}(G):=\mathrm{Hom}_S(G,CW)
\end{equation*}
le module de Dieudonné contravariant au sens de Fontaine (\cite{Fontaine} ou \cite{BBM} Section 4.2). En évaluant notre cristal le long de l'épaississement $(W(k)\twoheadrightarrow k)$ on obtient un isomorphisme de $W(k)$ module compatible avec $F$ et $V$ des deux cotés :
\begin{equation*}
\mathscr{E}(G)_{(W(k)\twoheadrightarrow k)}\simeq\mathbb{D}(G)^{(p)}
\end{equation*}
(voir \cite{BBM} Théorème 4.2.14). La réduction modulo $p$ de cet  isomorphisme permet l'identification de la $p$-torsion :
\begin{equation*}
\mathscr{E}\mathrm{xt}^1_{S/\Sigma}(\underline{G[p]},\mathcal{O}_{S/\Sigma})\simeq\mathbb{D}(G[p])^{(p)}\simeq(\mathbb{D}(G)/p\mathbb{D}(G))^{(p)}
\end{equation*}
où $\mathscr{E}\mathrm{xt}^1_{S/\Sigma}(\underline{G[p]},\mathcal{O}_{S/\Sigma})$ est le cristal associé au groupe fini plat de $p$-torsion $G[p]$. Par définition $\mathscr{E}(G)$ étant un cristal, il commute aux changement de bases et par conséquent la réduction modulo $p$ ci dessus redonne l'identification bien connue entre le module de Dieudonné de la $p$-torsion et le premier groupe de cohomologie de De Rham : 
\begin{equation*}
\mathscr{E}(G)_{(W(k)\twoheadrightarrow k)}\otimes_{W(k)}k\simeq\mathscr{E}(G)_{(k\rightarrow k)}\simeq H^1_{\mathrm{dR}}(A/k)
\end{equation*}
Toujours d'après \cite{BBM} (Proposition 4.3.10) on dispose d'un isomorphisme :

\begin{equation*}
\omega_G^{(p)}\simeq \mathscr{E}(G)_{(S\rightarrow S)}/F(\mathscr{E}(G)_{(S\rightarrow S)}^{(p)})
\end{equation*}

\subsubsection{}
On s'intéresse au cas où $A/S$ est muni d'une action de $\mathcal{O}_L$ où $L/\Q_p$ est une extension de degré fini. D'après \cite{Fargues} on dispose du résultat suivant :

\begin{prop}[\cite{Fargues}, Lemme B.1]
\label{Fargues}
Soit $\mathscr{E}$ un $F$-cristal en $\mathcal{O}_{S/\Sigma}$-modules localement libre de rang fini sur $\mathrm{Cris}(S/\Sigma)$ muni d'une action de $\mathcal{O}_L$. Alors $\mathscr{E}$ est un $\mathcal{O}_{S/\Sigma}\otimes_{\Z_p}\mathcal{O}_{L}$-module localement libre sur $\mathrm{Cris}(S/\Sigma)$.
\end{prop}
\begin{remark}
D'après \cite{BBM} on dispose d'un isomorphisme :
\begin{equation*}
\omega_G\simeq \mathscr{E}\mathrm{xt}^1(\underline{G},\mathcal{J}_{S/\Sigma})_{(S\rightarrow S)}
\end{equation*}
où $\mathcal{J}_{S/\Sigma}$ est le faisceau d'idéal à puissances divisées. Or $\mathscr{E}\mathrm{xt}^1(\underline{G},\mathcal{J}_{S/\Sigma})$ n'est à priori pas un cristal et donc la proposition ci-dessus ne s'applique pas. 
\end{remark}

\subsubsection{}
\label{FactoVerFrob} Les morphismes $F$ et $V$ de $G$ induisent un diagramme commutatif aux lignes horizontales exactes : 

\[\begin{tikzcd}
	0 & {\omega_G} & {\mathscr{E}(G)_{(S\rightarrow S)}} & {\omega_{G^D}^{\vee}} & 0 \\
	0 & {\omega_G^{(p)}} & {\mathscr{E}(G)_{(S\rightarrow S)}^{(p)}} & {(\omega_{G^D}^{\vee})^{(p)}} & 0 \\
	0 & {\omega_G} & {\mathscr{E}(G)_{(S\rightarrow S)}} & {\omega_{G^D}^{\vee}} & 0
	\arrow[from=1-2, to=1-3]
	\arrow[from=1-3, to=1-4]
	\arrow[from=1-4, to=1-5]
	\arrow[from=1-1, to=1-2]
	\arrow[from=2-1, to=2-2]
	\arrow[from=2-2, to=2-3]
	\arrow[from=2-3, to=2-4]
	\arrow[from=2-4, to=2-5]
	\arrow["V"', from=1-2, to=2-2]
	\arrow["V"', from=1-3, to=2-3]
	\arrow["V"', from=1-4, to=2-4]
	\arrow[from=3-1, to=3-2]
	\arrow[from=3-2, to=3-3]
	\arrow[from=3-3, to=3-4]
	\arrow[from=3-4, to=3-5]
	\arrow["F"', from=2-2, to=3-2]
	\arrow["F"', from=2-3, to=3-3]
	\arrow["F"', from=2-4, to=3-4]
\end{tikzcd}\]
Les composées verticales sont nulles car égales à $p$. Puisque  $F$ est nul sur  $\omega_G^{(p)}$  le morphisme $F$ se factorise en :
\begin{equation*}
F:(\omega_{G^D}^{\vee})^{(p)}\longrightarrow \ \mathscr{E}(G)_{(S\rightarrow S)}
\end{equation*}
De même $V$ est nul sur $\omega_{G^D}^{\vee}$ (utiliser que $V_G=F_{G^D}^D$) donc on obtient une factorisation :
\begin{equation*}
V:\mathscr{E}(G)_{(S\rightarrow S)}\rightarrow \omega_{G}^{(p)}
\end{equation*}
Nous aurons besoin de cette factorisation pour définir l'invariant de Hasse primitif dans la section \ref{SectionHasse}. 

\subsection{Déformations}
\subsubsection{Grothendieck-Messing}
Nous allons maintenant rappeler quelques résultats sur les déformations des groupes de Barsotti-Tate. \\

 Soit $S_0\hookrightarrow S$ une immersion fermée nilpotente munie d'une structure de puissances divisées (PD-structure), avec $p$ localement nilpotent. On dispose d'un morphisme sur le site $\mathrm{Cris}(S_0/\Sigma)$ :
 \begin{equation*}
 (S_0\rightarrow S_0,\gamma_0)\longrightarrow (S_0\hookrightarrow S,\gamma)
 \end{equation*}
Si $\mathscr{F}$ est un cristal en $\OK_{S_0/\Sigma}$-module alors par définition on dispose d'un isomorphisme de $\OK_{S_0}$-modules canonique :
 \begin{equation*}
\mathscr{F}_{(S_0\hookrightarrow S )}\otimes_{\mathcal{O}_S}\mathcal{O}_{S_0}\simeq \mathscr{F}_{(S_0\rightarrow S_0)}
 \end{equation*}
 On note $\mathtt{BT}(S)$ la catégorie des groupes de Barsotti-Tate sur $S$.  Soit $G\in\mathtt{BT}(S)$ un tel groupe. On note $\mathscr{E}(G)$ le cristal associé et  $G|_{S_0}=G\times_{S}S_0$ le changement de base. On a alors un isomorphisme canonique de $\OK_S$-module :
 \begin{equation*}
 \mathscr{E}(G)_{(S\rightarrow S)}\simeq  \mathscr{E}(G|_{S_0})_{(S_0\hookrightarrow S)}
 \end{equation*}
 En effet d'un coté on dispose d'un isomorphisme :
 \begin{equation*}
 \mathscr{E}(G|_{S_0})_{(S_0\hookrightarrow S)}\simeq \mathscr{E}(G)_{(S_0\hookrightarrow S)}
 \end{equation*}
 provenant essentiellement de la définition du cristal $\mathscr{E}(G)$ (\cite{BBM1}, (2.3)). De l'autre coté puisque $\mathscr{E}(G)$ est un cristal, le morphisme $(S_0\rightarrow S)\rightarrow (S\rightarrow S)$ dans $\mathrm{Cris}(S/\Sigma)$ induit un isomorphisme :
 \begin{equation*}
  \mathscr{E}(G)_{(S_0\hookrightarrow S)}\otimes_{\OK_S}\OK_S=\mathscr{E}(G)_{(S\rightarrow S)}
 \end{equation*}
 et le résultat en découle. Cet isomorphisme est compatible aux filtrations de Hodge  :
 
 \[\begin{tikzcd}
	{\omega_G} & {\mathscr{E}(G)_{(S\rightarrow S)}} & {\mathscr{E}(G|_{S_0})_{(S_0\hookrightarrow S)}} \\
	{\omega_{G|_{S_0}}} & {\mathscr{E}(G|_{S_0})_{(S_0\rightarrow S_0)}} & {\mathscr{E}(G|_{S_0})_{(S_0\hookrightarrow S)}\otimes_{\mathcal{O}_S}\mathcal{O}_{S_0}}
	\arrow[from=1-1, to=2-1]
	\arrow[from=1-1, to=1-2]
	\arrow[from=2-1, to=2-2]
	\arrow[from=1-2, to=2-2]
	\arrow["\simeq", from=1-2, to=1-3]
	\arrow["\simeq", from=2-2, to=2-3]
	\arrow[from=1-3, to=2-3]
\end{tikzcd}\]
dans le sens où le diagramme est commutatif et la ligne du bas correspond à la réduction de la ligne du haut.

\begin{definition}
Soit $G_0\in\mathtt{BT}(S_0)$ un groupe de Barsotti-Tate sur $S_0$.  Une filtration  $\mathrm{Fil}^1\subset \mathscr{E}(G_0)_{(S_0\hookrightarrow S)}$ est dite admissible si c'est  un $\OK_S$-module localement facteur direct de $\mathscr{E}(G_0)_{(S_0\hookrightarrow S)}$ qui relève $\omega_{G_0}\subset  \mathscr{E}(G_0)_{(S_0\rightarrow S_0)}$. 
\end{definition}

D'après ce qui précède si $G\in\mathtt{BT}(S)$ relève $G_0\in\mathtt{BT}(S_0)$ alors $\omega_G\subset \mathscr{E}(G)_{(S\rightarrow S)}\simeq  \mathscr{E}(G|_{S_0})_{(S_0\hookrightarrow S)}$ est une filtration admissible.  On note $\mathtt{DefBT}(S_0\hookrightarrow S)$ la catégorie dont les objets sont les couples $(G_0,\mathrm{Fil}^1)$ où $G_0\in\mathtt{BT}(S_0)$, $\mathrm{Fil}^1\subset \mathscr{E}(G_0)_{(S_0\hookrightarrow S)}$ est une filtration admissible et où les morphismes $(G_0,\mathrm{Fil}^1)\rightarrow (G_0',\mathrm{Fil}^{1'})$ sont les morphismes $G_0\rightarrow G_0'$ compatibles avec les filtrations respectives. On peut maintenant énoncer le théorème de Grothendieck-Messing :

\begin{theorem}[\cite{Messing1972} V.1.6]
Le foncteur 
\begin{equation*}
\begin{array}{lrcl}
\mathtt{BT}(S)&\longrightarrow &\mathtt{DefBT}(S_0\hookrightarrow S)\\
G&\longmapsto &\big(G|_{S_0},\ \omega_G\subset \mathscr{E}(G|_{S_0})_{(S_0\hookrightarrow S)}\big)
\end{array}
\end{equation*}
est une équivalence de catégorie.
\end{theorem}

\subsubsection{Applications}

 Dans cette section nous allons voir comment appliquer le théorème de déformation de Grothendieck-Messing le long de l'immersion fermée $k[\![t]\!]\rightarrow k$, qui n'est pas munie de PD-structure. Cela nous sera utile lorsque nous voudrons calculer des relations d'adhérences entres strates de nos variétés de Shimura (voir Proposition \ref{PropAdherence}).\\
 
 Soit $R$ un anneaux de caractéristique $p$ et $I$ un idéal tel que $I^2=0$. On peut munir $I$ d'une PD-structure faisant du morphisme $\spec\, R/I\rightarrow\spec\, R$ une immersion fermée nilpotente avec PD-structure. En effet il suffit de poser $\gamma_1(x)=x$ et $\gamma_n(x)=0$ pour tout $n\geq 2$ et tout $x\in I$. Nous allons appliquer cette remarque à la suite d'immersions fermées :
 \begin{equation*}
 \cdots\rightarrow k[t]/(t^{n+1})\rightarrow k[t]/(t^n)\rightarrow k[t]/(t^{n-1})\rightarrow \cdots \rightarrow  k[t]/(t^2)\rightarrow k
 \end{equation*}
 Pour tout $n\geq 1$ on note $R_n=k[t]/(t^n)$, $I_n=\mathrm{Ker}\big(k[t]/(t^n)\rightarrow k[t]/(t^{n-1})\big)$ et $S_n=\spec\, R_n$. Chaque immersion fermée $S_{n-1}\hookrightarrow S_{n}$ est définie par un idéal $I_n=(t^{n-1})/(t^n)$ satisfaisant $I_n^2=0$. D'après ce qui précède on peut donc munir chacune de ces immersions fermées d'une PD-structure. Le résultat suivant fonctionne pour tout anneau $R$ local complet de corps résiduel parfait.
 
 \begin{prop}[\cite{Lau2}, Lemme 2.10]
 Soit $G_1\in\mathtt{BT}(k)$ et $\mathrm{Fil}^1\subset\mathscr{E}(G_1)_{(k\rightarrow k)}\otimes_k k[\![t]\!]$ un relèvement de $\omega_{G_1}\subset \mathscr{E}(G_1)_{(k\rightarrow k)}$ qui est localement un facteur direct. Alors il existe $G\in \mathtt{BT}(k[\![t]\!])$   tel que :
 \begin{enumerate}
 \item $G$ est un relèvement de $G_1$ le long de $\spec\, k\rightarrow \spec\, k[\![t]\!]$
 \item $\mathscr{E}(G)_{(k[\![t]\!]\rightarrow k[\![t]\!])}\simeq \mathscr{E}(G)_{(k\rightarrow k)}\otimes_{k}k[\![t]\!]$
 \item $\omega_{G}\simeq \mathrm{Fil}^1$ (via l'identification ci dessus)
 \end{enumerate}
 \end{prop}
 \begin{proof}
 Nous allons utiliser la théorie des display de Zink et Lau. Posons $R=k[\![t]\!]$ et $S=\spec\, R$. Notez que se donner un groupe $p$-divisible $G\in\mathtt{BT}(S)$  c'est se donner un système compatible de groupes $p$-divisibles $(G_n)_{n\geq 1}$ où $G_n\in\mathtt{BT}(S_n)$ (\cite{Lau_2014} Lemme 2.15). Nous allons donc construire pour tout $n\geq 2$ un groupe $p$-divisible $G_n\in\mathtt{BT}(S_n)$ satisfaisant les propriétés souhaitées. Puisque pour tout $n\geq 2$ l'anneau $R_n$ est un anneau local artinien sur lequel $p$ est nilpotent et de corps résiduel parfait, on dispose donc d'une équivalence de catégorie $\mathtt{BT}(S_n)\cong \mathtt{Disp}(R_n)$.  On note $\mathscr{P}_1=(P_1,Q_1,F_1,V^{-1}_1)$ le display associé à $G_1$ par Lau (\cite{Lau_2014}, Théorème A). Toujours d'après \textit{loc cit}, Théorème A, on dispose d'un isomorphisme canonique entre le cristal de Dieudonné de $G_1$ et le cristal associé à $\mathscr{P}_1$ par Zink (\cite{Zink} ou \cite{Lau_2014}) :
 \begin{equation*}
 \mathscr{E}(G_1)\simeq \mathbb{D}(\mathscr{P}_1)
 \end{equation*}
 Le cristal $\mathbb{D}(\mathscr{P}_1)$ étant défini sur le site $\mathrm{Cris}_{\mathrm{adm}}(R)\subset \mathrm{Cris}(R)$ des épaississements avec pd-structure $(B\rightarrow A,\delta)$ avec $A$ admissible (i.e. si le nilradical $\mathcal{N}_R$ est \og bounded nilpotent\fg\ et que $R_{\mathrm{red}}=R/\mathcal{N}_R$ est un anneau parfait de caractéristique $p$), l'isomorphisme ci-dessus est un isomorphisme de cristaux sur le site $\mathrm{Cris}_{\mathrm{adm}}(R)$ et on fait ici l'abus de notation d'également noter $\mathscr{E}(G_1)$ la restriction du cristal $\mathscr{E}(G_1)$ au site $\mathrm{Cris}_{\mathrm{adm}}(R)$. En particulier en évaluant sur l'épaississement tautologique $(S_1\rightarrow S_1)$ on obtient un isomorphisme canonique de $R_1$-modules :
 \begin{equation}
 \label{EQcristaux1}
 \mathscr{E}(G_1)_{(S_1\rightarrow S_1)}\simeq P_1\otimes_{\mathbb{W}(R_1)} R_1
 \end{equation}
    D'après la remarque précédente, $R_2\rightarrow R_1$ est muni d'une PD-structure nilpotente, et par conséquent se donner un relèvement $\mathscr{P}_2\in\mathtt{Disp}(R_2)$ de $\mathscr{P}_1$ c'est se donner un relèvement $\mathrm{Fil}^1_{(2)}\subset \mathbb{D}(\mathscr{P}_1)_{(R_2\twoheadrightarrow R_1)}$ (voir \cite{Lau_2014}, Corollaire 2.10, ou \cite{Lau2} Lemme 4.2). On pose $P_2=P_1\otimes_{\mathbb{W}(R_1)}\mathbb{W}(R_2)$ (notez qu'on dispose pour tout $n\geq 2$ d'un morphisme  $R_1\rightarrow R_n$). D'après \cite{Lau_2014}, Section 2.6, on a par construction du cristal $\mathbb{D}(\mathscr{P}_1)$ un isomorphisme :
 \begin{equation*}
\mathbb{D}(\mathscr{P}_1)_{(R_2\twoheadrightarrow R_1)} \simeq P_2\otimes_{\mathbb{W}(R_2)}R_2=P_1\otimes_{\mathbb{W}(R_1)}R_2
 \end{equation*}
 Via l'isomorphisme \eqref{EQcristaux1} on obtient :
 \begin{equation*}
 \mathscr{E}(G_1)_{(S_1\rightarrow S_1)}\otimes_{R_1} R_2\simeq   \mathbb{D}(\mathscr{P}_1)_{(R_2\twoheadrightarrow R_1)}
 \end{equation*}
 On définit le relèvement de la filtration de Hodge de $\mathscr{P}_1$ comme étant :
 \begin{equation*}
 \mathrm{Fil}_{(2)}^1:=\mathrm{Fil}^1\otimes_{R} R_2 \subset \mathscr{E}(G_1)_{(S_1\rightarrow S_1)}\otimes_{R_1} R_2\simeq  \mathbb{D}(\mathscr{P}_1)_{(R_2\twoheadrightarrow R_1)}
 \end{equation*}
 Ce relèvement  définit une display $\mathscr{P}_2\in\mathtt{Disp}(R_2)$ et donc un groupe $p$-divisible $G_2\in \mathtt{BT}(S_2)$. On dispose de nouveau d'une identification entre cristaux $ \mathscr{E}(G_2)\simeq \mathbb{D}(\mathscr{P}_2)$. Puisque $G_2$ est un relèvement de  $G_1$ le long de $S_1\hookrightarrow S_2$ on a :
 \begin{align*}
  \mathscr{E}(G_2)_{(S_2\rightarrow S_2)}&\simeq\mathscr{E}(G_1)_{(S_1\hookrightarrow S_2)}\\
 &\simeq  \mathbb{D}(\mathscr{P}_1)_{(R_2\twoheadrightarrow R_1)}\\
 &\simeq  \mathscr{E}(G_1)_{(S_1\rightarrow S_1)}\otimes_{R_1} R_2\\
 &\simeq  \mathscr{E}\otimes_R R_2
 \end{align*}
 où l'on a posé $\mathscr{E}:=\mathscr{E}(G_1)_{(S_1\rightarrow S_1)}\otimes_{R_1} R$. On continue le processus par induction en posant pour tout $n\geq 3$ :
 \begin{equation*}
 P_{n}:=P_1\otimes_{\mathbb{W}(R_1)}\mathbb{W}(R_{n}),\ \ \ \ \mathrm{Fil}_{(n)}^1:=\mathrm{Fil}^1\otimes_{R} R_{n}\subset \mathbb{D}(\mathscr{P}_{n-1})_{(R_{n}\twoheadrightarrow R_{n-1})}
 \end{equation*}
 ce qui nous fournit  à chaque étape un display $\mathscr{P}_{n}\in\mathtt{Disp}(R_{n})$ et un groupe $p$-divisible $G_{n}\in \mathtt{BT}(S_{n})$ satisfaisant les équations :
 \begin{equation*}
 \omega_{G_{n}}\simeq \mathrm{Fil}^1_{(n)},\ \ \ \ \mathscr{E}(G_{n})_{(S_n\rightarrow S_n)}\simeq \mathscr{E}\otimes_R R_n
 \end{equation*}
 Par passage à la limite on obtient un groupe $p$-divisible $G\in\mathtt{BT}(S)$ satisfaisant par construction :
 \begin{equation*}
 \mathscr{E}(G)_{(S\rightarrow S)}\simeq \lim_{\substack{\longleftarrow\\n}} \mathscr{E}\otimes_R R_n\simeq \mathscr{E}
 \end{equation*}
 \begin{equation*}
\omega_{G}\simeq \lim_{\substack{\longleftarrow\\n}} \mathrm{Fil}^1_{(n)}\simeq \lim_{\substack{\longleftarrow\\n}} \mathrm{Fil}^1\otimes_R R_n\simeq \mathrm{Fil}^1
 \end{equation*}
 ce qui démontre les points $(1),(2),(3)$.
 \end{proof}

\begin{remark}\ 

\begin{enumerate}
\item En fait la proposition ci-dessus est une reformulation en termes de cristaux de Dieudonné des lemmes 2.15 et 2.16 de \cite{Lau_2014}. Cette reformulation est possible grâce à l'identification entre le cristal de Dieudonné $\mathscr{E}(G)$ et le cristal associé à un display $\mathbb{D}(\mathscr{P})$. Nous aurions très bien pu nous passer de la théorie de Dieudonné cristalline et simplement utiliser la théorie des displays.
\item L'un des apports de la théorie des displays dans la preuve ci-dessus est qu'elle rend explicite le calcul du faisceau $\mathscr{E}(G_n)_{(R_{n+1}\twoheadrightarrow R_n)}$ : il suffit de prendre n'importe quel $\mathscr{P}_{n+1}$ qui relève le display $\mathscr{P}_n$ (voir \cite{Zink} Théorème 3). Cela découle du fait que le morphisme de \og frame\fg $(\mathscr{D}_{R_{n+1}/R_n}\rightarrow \mathscr{D}_{R_n})$ est cristallin (voir \cite{Lau_2014} Proposition 2.8). Dans la preuve ci-dessus à chaque étape nous avons défini la display $\mathscr{P}_{n+1}$ comme étant le changement de base de $\mathscr{P}_1$ le long de la section $R_1\rightarrow R_{n+1}$. La proposition ci-dessus n'est donc qu'un passage à la limite d'un fait bien connu.
\end{enumerate}
\end{remark}

Énonçons pour finir un lemme dont nous aurons besoin dans la section \ref{SectionHasse} : 

\begin{lemma}
\label{LemmaRelevement}
Soit $S_0=\spec\, R/I \hookrightarrow S=\spec\, R$ une immersion fermée telle que $I^2=0$. Soit $G_0\in\mathtt{BT}(S_0)$ un groupe $p$-divisible sur $S_0$. 
Soit $\mathrm{Fil}^1_{(1)}, \mathrm{Fil}^1_{(2)}\subset  \mathscr{E}(G_0)_{(S_0\hookrightarrow S)}$ deux  filtrations admissibles. Dans $\mathscr{E}(G_0)_{(S_0\hookrightarrow S)}^{(p)}$ on dispose de l'égalité :
\begin{equation*}
(\mathrm{Fil}^1_{(1)})^{(p)}= (\mathrm{Fil}^1_{(2)})^{(p)}
\end{equation*}
En particulier le faisceau $\omega_G^{(p)}$ ne dépend pas du relèvement $G\in\mathtt{BT}(S)$.
\end{lemma}
\begin{proof}
C'est un simple résultat d'algèbre commutative. On pose $M=\mathscr{E}(G_0)_{(S_0\hookrightarrow S)}$. C'est un $R$-module. Puisque $\mathscr{E}$ est un cristal on dispose de l'identification $\mathscr{E}(G_0)_{(S_0\rightarrow S_0)}=M/IM$. Soit donc $N_1,N_2\subset M$ tels que $N_1/IM=N_2/IM$. On veut montrer que $N_1^{(p)}=N_2^{(p)}\subset M\otimes_{R,\sigma} R$ où $\sigma:R\rightarrow R$ désigne le Frobenius. Mais c'est clair puisque $0=IM^{(p)}\subset M\otimes_{R,\sigma} R$ (voir que $im\otimes 1=m\otimes \sigma (i)$ et que $\sigma (i)=0$ car par hypothèse $I^2=0$).
\end{proof}

\subsection{Relation d'adhérence}

\begin{definition}
\label{bonnestratification}
Soit $S$ un espace topologique. Une stratification de $S$ par rapport à un ensemble partiellement ordonné $(I,\leq)$ est une décomposition :
\begin{equation*}
S=\coprod_{i\in I}S_i
\end{equation*}
telle que pour tout $i\in I$ on ait la relation d'adhérence :
\begin{equation*}
\overline{S_i}\subset\coprod_{j\leq i}S_j
\end{equation*}
Une stratification est appelée bonne stratification si de plus elle satisfait pour tout $i\in I$ :
\begin{equation*}
\overline{S_i}=\coprod_{j\leq i}S_j
\end{equation*}
\end{definition}

Dans cette partie nous allons énoncer un résultat dont nous aurons besoin par la suite lorsque nous nous intéresserons aux relations d'adhérences entre strates de nos variétés de Shimura. Ce résultat est certainement bien connu et est notamment utiliser dans \cite{BijHerAR} (Preuve du théorème 4.11) mais à défaut d'avoir trouvé une preuve, nous allons en proposer une ci-dessous. 

\begin{prop}
\label{PropAdherence}
Soit $X$ un un schéma  noethérien de caractéristique $p>0$ et $Y\subset X$ un sous schéma localement fermé. Les conditions suivantes sont équivalentes :

\begin{enumerate}
\item $x\in \overline{Y}$
\item Il existe $y\in Y$ tel que  $y\leadsto x$ 
\item Il existe une extension de corps $k/\kappa (x)$ et un morphisme de schémas $\spec\, k[\![t]\!]\rightarrow X$ qui envoi le point fermé sur $x$ et le point générique sur $y$.
\end{enumerate}
\end{prop}
\begin{proof}
$(3)\Rightarrow (1)$ découle de la continuité du morphisme. $(1)\Rightarrow (2)$ découle du lemme \ref{Lemma1} ci dessous et du fait que l'espace topologique sous-jacent à un schéma soit sobre et qu'un espace topologique noethérien sobre soit spectral (dans un espace topologique noethérien tout sous-ensemble est quasi-compact). Montrons que $(2)\Rightarrow (3)$. D'après le lemme \ref{Lemma2} ci dessous, il existe un anneau de valuation discrète $R$ satisfaisant les propriétés du point $(3)$. Le théorème de structure de Cohen nous assure que $\widehat{R}\simeq k[\![t]\!]$ où $\widehat{R}$ désigne la complétion le long de l'idéal maximal $\mathfrak{m}\subset R$ et $k=R/\mathfrak{m}$ désigne le corps résiduel. Le morphisme $\spec\, \widehat{R}\rightarrow X$ satisfait les propriétés souhaitées.
\end{proof}

\begin{lemma} 
\label{Lemma1}
Soit $X$ un espace topologique spectral, et $Y\subset X$ un sous ensemble constructible. Alors 
\begin{equation*}
\overline{Y}=\bigcup_{y\in Y}\overline{\{y\}}
\end{equation*}
\end{lemma}
\begin{proof}
Voir \cite{stacks-project} Lemme 5.23.6.
\end{proof}
\begin{lemma}
\label{Lemma2}
Soit $X$ un schéma noethérien et $y\leadsto x$ une spécialisation. Alors il existe un anneau de valuation discrète $R$ et un morphisme $\spec\, R\rightarrow X$ qui envoie le point fermé sur $x$ et le point générique sur $y$.
\end{lemma}
\begin{proof}
Voir \cite{stacks-project} Lemme 28.5.10.
\end{proof}

\section{Grassmanniennes Affines}

\subsection{Grassmanienne affine}

Le début de cette section est constitué essentiellement de rappels sur la géométrie des grassmanniennes affines que l'on peut retrouver dans \cite{XZhuIntro} par exemple. \\

Soit $G$ un groupe réductif connexe déployé lisse sur un corps $k$. On fixe un tore déployé $T\subset G$ sur $k$ et un borel $B$ le contenant On notera $\langle,\rangle :\mathbb{X}^{*}(T)\times \mathbb{X}_{*}(T)\rightarrow \mathbb{Z}$ le produit scalaire, $\mathbb{X}_{*}(T)^+$ l'ensemble des cocaractères $B$-dominants et $2\rho\in \mathbb{X}^{*}(T)$ la somme des racines positives. On note $\mathrm{Alg}_k$ la catégorie des $k$-algèbres. On définit le groupes de lacets $LG$ et le groupe d'arcs $L^+G$ comme les foncteurs $\mathrm{Alg}_k\rightarrow \mathrm{Sets}$ qui à une $k$ algèbre $R$ associent :
\begin{equation*}
LG(R)=G(R\kT ),\ \ \ \ L^+G(R)=G(R\kt )
\end{equation*}
On définit la grassmannienne affine pour le groupe $G$ comme le quotient (pour la topologie étale ou fppf puisque $G$ est lisse)
\begin{equation*}
\mathrm{Gr}_G=LG/L^+G
\end{equation*}
On montre que ce quotient est ind-représentable par un schéma propre sur $\spec\, k$. En utilisant le fait que $G$ soit supposé lisse on peut montrer que l'on a en fait une description explicite de ce quotient :
\begin{equation*}
\mathrm{Gr}_G(R)=\left\{(\mathcal{E},\beta)\  \bigg|\ \begin{array}{ll} \mathcal{E}\mathrm{\ est\ un\ }G\mathrm{\ torseur\ sur\ }\mathbb{D}_R,\\ \beta:\mathcal{E}|_{\mathbb{D}_R^*}\simeq \mathcal{E}^0|_{\mathbb{D}_R^*}\mathrm{\ est \ une\ trivialisation}
\end{array}\right\}
\end{equation*}
où $\mathbb{D}_R=\spec\,R\kt $ désigne le disque unité, $\mathbb{D}_R^*=\spec\,R\kT $  le disque unité épointé et $\mathcal{E}^0$ le $G$-torseur trivial sur $\mathbb{D}_{R}$. Dans la définition ci dessus $\mathcal{E}$ est un $G$-torseur pour la topologie fppf ou étale (encore une fois puisque $G$ est supposé lisse). Par la suite nous noterons $\mathcal{E}\dashrightarrow\mathcal{E}^{0}$ la trivialisation $\beta:\mathcal{E}|_{\mathbb{D}_R^*}\simeq \mathcal{E}^0|_{\mathbb{D}_R^*}$.\\

Par la suite nous aurons besoin du lemme facile suivant :
\begin{lemma}
\label{lemmeG1G2}
Soit $G_1, G_2$ deux groupes réductifs connexes déployés lisses sur  $k$. On dispose d'un isomorphisme canonique
\begin{equation*}
\Gr_{G_1\times G_2}\simeq\Gr_{G_1}\times\Gr_{G_2}
\end{equation*}
\end{lemma}
\begin{proof}
Il suffit de voir que le résultat est vrai au niveau des foncteurs :
\begin{equation*}
L(G_1\times G_2)(R)=(G_1\times G_2)(R\kT)=G_1(R\kT)\times G_2(R\kT)
\end{equation*}
\end{proof}

On dispose d'une action à gauche :
\begin{equation*}
\begin{array}{lrcl}
L^+G\times \mathrm{Gr}_G&\longrightarrow &\mathrm{Gr}_G\\
(g,(\mathcal{E},\beta))&\longmapsto & (\mathcal{E},g\cdot\beta)
\end{array}
\end{equation*}
dont le quotient est appelé champs de Hecke
\begin{equation*}
\mathrm{Hecke}_G=[L^+G\backslash LG/L^+G]
\end{equation*}
Par construction il associe à une $k$ algèbre $R$
\begin{equation*}
\mathrm{Hecke}_G(R)=\left\{(\mathcal{E},\mathcal{E}'\beta)\  \bigg|\ \begin{array}{ll} \mathcal{E},\mathcal{E}'\mathrm{\ sont\ des\ }G\mathrm{\ torseurs\ sur\ }\mathbb{D}_R,\\ \beta:\mathcal{E}|_{\mathbb{D}_R^*}\simeq \mathcal{E}'|_{\mathbb{D}_R^*}\mathrm{\ est \ un\ isomorphisme}
\end{array}\right\}
\end{equation*}
Désormais nous ferons l'abus de notation de retirer l'indice $G$ et de noter $\mathrm{Hecke}$ et $\mathrm{Gr}$.
On rappelle que par la décomposition de Cartan, on peut associer à une modification $\beta:\mathcal{E}\dashrightarrow\mathcal{E}'$, sa position relative $\mathrm{Inv}(\beta)\in \mathbb{X}_{*}(T)^+$ via la bijection 
\begin{equation*}
\begin{array}{lrcll}
&G(k\kt)\backslash G(k\kT)/G(k\kt)&\longrightarrow &\mathbb{X}_{*}(T)^+\\
&\lbrack g\rbrack &\longmapsto & \mathrm{Inv}(g)\\
&\lbrack u^{\lambda}\rbrack& \longmapsfrom & \lambda
\end{array}
\end{equation*}
où $u^{\lambda}=\lambda (u)\in T(k\kT)$. Pour tout élément $\lambda\in \mathbb{X}_{*}(T)^+$ on définit
\begin{equation*}
\mathrm{Gr}_{\lambda}:=\left\{(\mathcal{E},\beta)\in\Gr\  |\ \mathrm{Inv}(\beta)=\lambda\ \right\},\ \ \mathrm{Gr}_{\leq\lambda}:=\left\{(\mathcal{E},\beta)\in\Gr\  |\ \mathrm{Inv}(\beta)\leq\lambda\ \right\}
\end{equation*}
De la même manière on définit $\mathrm{Hecke}_{\lambda}$ et $\mathrm{Hecke}_{\leq\lambda}$. La proposition suivante résume la plupart des propriétés de la décomposition de $\Gr$ en $L^+G$-orbites.
\begin{prop}
\label{PropOrbiteGR}
\begin{enumerate}
\item Pour tout $\lambda\in \mathbb{X}_{*}(T)^+$ on a $\mathrm{Gr}_{\lambda}=L^+G\cdot u^{\lambda}$ est une $L^+G$-orbite. 
\item $\mathrm{Gr}_{\lambda}$ est une variété quasi-projective lisse de dimension $\langle 2\rho , \lambda\rangle$
\item On dispose d'une décomposition de $\Gr$ en $L^+G$-orbites
\begin{equation*}
\mathrm{Gr}=\coprod_{\lambda\in \mathbb{X}_{*}(T)^+}\mathrm{Gr}_{\lambda}
\end{equation*}
\item Pour tout $\lambda\in \mathbb{X}_{*}(T)^+$ on a la relation d'adhérence 
\begin{equation*}
\overline{\mathrm{Gr}_{\lambda}}=\bigcup_{\lambda'\leq\lambda}\Gr_{\lambda'}
\end{equation*}
\item L'ouvert dense $\Gr_{\lambda}\subset \Gr_{\leq \lambda}$ coïncide avec le lieu lisse de $\Gr_{\leq \lambda}$.
\end{enumerate}
\end{prop}
\begin{proof}
Les points $(1),(2),(3),(4)$ sont démontrés dans \cite{XZhuIntro} (Proposition. 2.1.5). Pour $(5)$ on pourra trouver une preuve dans \cite{SmoothAffGr} (Corollary B)
\end{proof}
\begin{remark}
La proposition précédente nous fournit une description de l'espace topologique sous-jacent au champs de Hecke borné :
\begin{equation*}
|\mathrm{Hecke}_{\leq\lambda}|\simeq \{\lambda'\in \mathbb{X}_{*}(T)^+\ |\ \lambda'\leq \lambda\}
\end{equation*}  
L'identification ci-dessus est un homéomorphisme où la topologie du membre de droite est celle induite par la relation d'ordre sur $\mathbb{X}_{*}(T)^+$
\end{remark}
\begin{remark}
Si $\mu\in \mathbb{X}_{*}(T)^+$ est minuscule alors on dispose d'une identification
\begin{equation*}
\mathrm{Gr}_{\mu}\simeq G/P_{\mu}
\end{equation*}
où $P_{\mu}$ désigne le sous groupe parabolique associé à $\mu$. En effet on dispose de deux morphismes dont on montre qu'ils sont réciproques l'un de l'autre :
\begin{equation*}
\begin{array}{lrcl}
&\Gr_{\mu}&\longrightarrow & G/P_{\mu}\\
&g\cdot u^{\mu} &\longmapsto &[\mathrm{ev}(g)]
\end{array},\ \ \ \ \begin{array}{lrcl}
&G/P_{\mu}&\longrightarrow & \Gr_{\mu}\\
&[g]&\longmapsto & g\cdot u^{\mu}
\end{array}
\end{equation*}
où $\mathrm{ev}:L^+G\rightarrow G,\ g\mapsto g\ (\mathrm{mod}\ u)$  et $G\hookrightarrow L^+G$ désigne le groupe des lacets \og constants\fg.
\end{remark}

\subsection{Produit de convolution}

On définit le produit de $n$-convolution $\mathrm{Gr}\tilde{\times}\dots\tilde{\times}\mathrm{Gr}$ comme étant le champ paramétrant les modifications  $(\beta_i:\mathcal{E}_i\dashrightarrow\mathcal{E}_{i-1})_{i=1,\dots n}$ :
\begin{equation*}
\mathcal{E}_n\dashrightarrow \mathcal{E}_{n-1} \dashrightarrow\dots\dashrightarrow \mathcal{E}_1\dashrightarrow \mathcal{E}^0
\end{equation*}
On dispose d'un morphisme :
\begin{equation*}
\begin{array}{lrcl}
&\mathrm{Gr}\tilde{\times}\dots\tilde{\times}\mathrm{Gr}&\longrightarrow &\Gr\\
&(\mathcal{E}_n\dashrightarrow \dots\dashrightarrow \mathcal{E}^0)&\longmapsto& (\mathcal{E}_n\dashrightarrow \mathcal{E}^0)
\end{array}
\end{equation*}
appelé morphisme de convolution. Ce procédé fournit également pour tout $1\leq i\leq n$ un morphisme :
\begin{equation*}
\begin{array}{lrcl}
&m_i:\mathrm{Gr}\tilde{\times}\dots\tilde{\times}\mathrm{Gr}&\longrightarrow &\Gr\\
&(\mathcal{E}_n\dashrightarrow \dots\dashrightarrow \mathcal{E}^0)&\longmapsto& (\mathcal{E}_i\dashrightarrow \mathcal{E}^0)
\end{array}
\end{equation*}
Ces morphismes mis ensemble nous donnent un isomorphisme :
\begin{equation*}
\prod_{i=1}^{n} m_i:\Gr\tilde{\times}\dots\tilde{\times}\Gr \simeq \Gr \times\dots\times\Gr
\end{equation*}
En particulier le produit de convolution $\Gr\tilde{\times}\dots\tilde{\times}\Gr$ est ind-représentable.
De la même manière que pour $\Gr$, on dispose d'une uniformisation du produit de convolution :
\begin{equation*}
\mathrm{Gr}\tilde{\times}\dots\tilde{\times}\mathrm{Gr}\simeq LG\times^{L^+G}\dots\times^{L^+G}LG\times^{L^+G}\Gr
\end{equation*}
Via cet isomorphisme le morphisme de convolution devient le morphisme de multiplication :
\begin{equation*}
\begin{array}{lrcl}
&LG\times^{L^+G}\dots\times^{L^+G}LG\times^{L^+G}\Gr&\longrightarrow &LG/L^+G\\
&(g_n,\dots,[g_1])&\longmapsto& [g_n\dots g_1]
\end{array}
\end{equation*}
\ \\

Par la suite pour alléger les notations et lorsque le nombre $n$ est explicite nous noterons le produit de $n$-convolution $\widetilde{\Gr}=\mathrm{Gr}\tilde{\times}\dots\tilde{\times}\mathrm{Gr}$. Pour tout $n$-uplet de cocaractères  $\lambda_{\bullet}=(\lambda_1,\dots,\lambda_n)$ on définit les sous schémas localement fermés du produit de convolution :
\begin{equation}
\label{ConvolutionRep}
\widetilde{\mathrm{Gr}}_{\lambda_{\bullet}}:=\left\{(\mathcal{E}_i,\beta_i)\in\widetilde{\Gr}\  |\ \mathrm{Inv}(\beta_i)=\lambda_i\ \right\},\ \ \widetilde{\mathrm{Gr}}_{\leq\lambda_{\bullet}}:=\left\{(\mathcal{E}_i,\beta_i)\in\widetilde{\Gr}\  |\ \mathrm{Inv}(\beta_i)\leq\lambda_i\ \right\}
\end{equation}
En particulier $\widetilde{\mathrm{Gr}}_{\lambda_{\bullet}}$ est représentable par un schéma.  Il n'est pas difficile de montrer qu'on dispose alors d'une bonne stratification pour tout $\mu_{\bullet}\in (\mathbb{X}_{*}(T)^+)^{n}$ :
\begin{equation*}
\widetilde{\mathrm{Gr}}_{\leq\mu_{\bullet}}=\bigcup_{\lambda_{\bullet}\leq \mu_{\bullet}}\widetilde{\mathrm{Gr}}_{\lambda_{\bullet}}
\end{equation*}
où $\lambda_{\bullet}\leq \mu_{\bullet}\Leftrightarrow \lambda_i\leq \mu_i\ \forall i=1,\dots n$. Notez que si chacun des $\mu_i$ est minuscule cette stratification est constituée d'une seule strate à savoir $\widetilde{\mathrm{Gr}}_{\leq\mu_{\bullet}}=\widetilde{\mathrm{Gr}}_{\mu_{\bullet}}$. Par la suite si  $\lambda_{\bullet}=(\lambda_1,\dots,\lambda_n)$ est un $n$-uplet de cocaractères on notera $|\lambda_{\bullet}|:=\lambda_1+\dots+\lambda_n$.

\begin{prop}
\label{LisseConv}
   $\widetilde{\Gr}_{\mu_\bullet}$ est lisse. En particulier si $\mu_i$ est minuscule pour tout $i=1,\dots n$, alors $\widetilde{\Gr}_{\leq\mu_\bullet}$ est lisse.
\end{prop}
\begin{proof}
On commence par regarder le morphisme de projection
\begin{equation*}
\begin{array}{lrcl}
&\widetilde{\Gr}_{(\mu_1,\mu_2)}&\longrightarrow & \Gr_{\mu_1}\\
&(\mathcal{E}_2\dashrightarrow\mathcal{E}_1\dashrightarrow\mathcal{E}^0)&\longmapsto &(\mathcal{E}_1\dashrightarrow\mathcal{E}^0)
\end{array}
\end{equation*}
En termes de groupes de lacets, ce morphisme correspond à la projection sur la première coordonnée :
\begin{equation*}
LG_{\mu_1}\times^{L^+G}\Gr_{\mu_2}\rightarrow\Gr_{\mu_1}
\end{equation*}
où $LG_{\mu_1}=p^{-1}(\Gr_{\mu_1})\subset LG$ avec $p:LG\rightarrow \Gr$ la projection. Par définition ce morphisme est une fibration avec pour fibre $\Gr_{\mu_2}$  (identification  non canonique) et est donc lisse. Enfin puisque $\Gr_{\mu_1}$ est lisse, il s'en suit que $\widetilde{\Gr}_{(\mu_1,\mu_2)}$ l'est également. Le résultat s'en déduit par récurrence.
\end{proof}

\begin{remark}
En particulier si $\mu_{\bullet}=(\mu_1,\dots\mu_n)$ avec chacun des $\mu_i$ minuscule, alors le morphisme de convolution 
\begin{equation*}
m_{\mu_{\bullet}}:\widetilde{\Gr}_{\leq\mu_{\bullet}}\rightarrow \Gr_{\leq|\mu_{\bullet}|}
\end{equation*}
est une résolution des singularités. Elle est parfois appelé résolution de Demazure en référence à \cite{Demazure1974}. Dans le cas $G=\mathrm{GL}_2$ et $\mu_i=(1,0)$, la preuve précédente nous dit que le produit de convolution est obtenu par des fibrations successives en $\Gr_{\mu_i}=G/P_{\mu_i}\simeq\mathbb{P}^1_k$.
\end{remark}

Le théorème suivant décrit la géométrie du morphisme de convolution :

\begin{theorem}[T.Haines, \cite{Haines}]\ 
\label{TheoHaines}
Soit $\mu_{\bullet}=(\mu_1,\dots\mu_n)$ un $n$-uplet de cocaractères quelconques. Alors
\begin{enumerate}
\item Le morphisme $m_{\mu_{\bullet}}:\widetilde{\mathrm{Gr}}_{\leq\mu_{\bullet}}\rightarrow \mathrm{Gr}_{\leq|\mu_{\bullet}|}$  est localement trivial en restriction à $\Gr_{\lambda}\subset \mathrm{Gr}_{\leq|\mu_{\bullet}|}$ pour tout $\lambda\leq |\mu_{\bullet}|$. 
\item Si chacun des $\mu_i$ est minuscule alors pour tout $y\in\mathrm{Gr}_{\lambda}$ la fibre $m_{\mu_{\bullet}}^{-1}(y)$ est équidimensionnelle de dimension $\langle \rho,|\mu_{\bullet}|-\lambda\rangle $.
\end{enumerate}

\end{theorem}
\begin{proof}
Le point $(1)$ correspond au Lemme 2.1 de \cite{Haines} et le point $(2)$ correspond au Théorème 1.1 de \textit{loc cit}. 
\end{proof}

\begin{remark}
Le point $(1)$ découle directement d'un fait plus général : si $p:X\rightarrow Y$ est un morphisme $G$-équivariant tel que $G$ agit transitivement sur $Y$, alors en choisissant un point de base $y_0\in Y$ on obtient un isomorphisme $G$-équivariant :
\begin{equation*}
G\times^{H}p^{-1}(y_0)\simeq X
\end{equation*}
où $H=\mathrm{Stab}_G(y_0)$. Dans notre situation le morphisme $m:m^{-1}(\mathrm{Gr}_{\lambda})\rightarrow \Gr_{\lambda}$ est bien $L^+G$ équivariant et $\Gr_{\lambda}$ est une $L^+G$-orbite par définition.
\end{remark}

\subsubsection{Exemples}

En utilisant le fait qu'un $\mathrm{GL}_n$-torseur $\mathcal{E}$ sur $\mathbb{D}_R$ (où $R$ est une $k$-algèbre) correspond à un $R\kt$-module $\Lambda$ localement libre de rang $n$, on obtient une description de la grassmannienne affine pour $\mathrm{GL}_n$ en termes de réseaux :

\begin{equation*}
\Gr(R)=\left\{\Lambda\subset\ R\kT^n\ \bigg|\begin{array}{lrcl}\Lambda\ R\kt\text{-}\mathrm{module\ localement\ libre},\\
\Lambda\otimes_{R\kt}R\kT\simeq R\kT^n
\end{array}\ \right\}
\end{equation*}
En prenant $T=\mathbb{G}_m^n$ le tore des matrices diagonales la décomposition de Cartan prend la forme :
\begin{equation*}
\begin{array}{lrcl}
&(\Z^n)^+&\longrightarrow & \mathrm{GL}_n(k[\![u]\!])\backslash \mathrm{GL}_n(k(\!(u)\!))/\mathrm{GL}_n(k[\![u]\!])\\
&\lambda=(\lambda_1,\dots,\lambda_n)&\longmapsto &u^{\lambda}=\mathrm{diag}(u^{\lambda_1},\dots u^{\lambda_n})
\end{array}
\end{equation*}
En termes de réseaux cela se traduit comme suit : pour tout réseau $\Lambda\subset k(\!(u)\!)^n$ il existe une base $(e_1,\dots, e_n)$ de $\Lambda_0=k\kt^n$ et $\lambda=(\lambda_1,\dots,\lambda_n)\in (\Z^n)^+$ tels que $(u^{\lambda_1}e_1,\dots,u^{\lambda_n}e_n)$ soit une base de $\Lambda$. Ici on a utilisé la notation
\begin{equation*}
(\Z^n)^+=\{(\lambda_1,\dots,\lambda_n)\ |\ \lambda_1\geq\dots\geq\lambda_n\ \}
\end{equation*}
Désormais nous noterons $\Lambda_0:=k\kt^n$ le réseau associé au $\mathrm{GL}_n$-torseur trivial $\mathcal{E}^0$.
\begin{example} 
\label{ExempleGr}

Dans tout ce qui suit $R$ désigne une $k$-algèbre. Voici quelques exemples :

\begin{enumerate}
\item Pour $\mu=(1^d,0^{n-d})$ on obtient en termes de réseaux :
\begin{equation*}
\Gr_{(1^d,0^{n-d})}(R)=\{\Lambda\subset \Lambda_{0}:=R[\![u]\!]^n\ |\ u\Lambda_0\subset\Lambda\subset \Lambda_0,\ \mathrm{dim}_k\Lambda_0/\Lambda=d\  \}
\end{equation*}
On retrouve bien la variété $\mathrm{GL}_n/P_{\mu}=\mathrm{Grass}(n-d,d)$ via :
\begin{equation*}
\Lambda\mapsto (R^n=\Lambda_{0}/u\Lambda_{0}\rightarrow \Lambda_{0}/\Lambda)
\end{equation*}
\item Pour $\mathrm{GL}_2$ et $\mu=(e,0)$ on trouve :
\begin{equation*}
\Gr_{\leq(e,0)}(R)=\{\Lambda\ |\ u^e\Lambda_0\subset\Lambda\subset \Lambda_0,\ \mathrm{dim}_k\Lambda_0/\Lambda=e\  \}
\end{equation*}
Soit $\Lambda\in \Gr_{\leq(e,0)}(k)$.  Il existe un plus petit entier $i\leq e$ tel que $u^i\Lambda\subset u^e\Lambda_0$. Notons $N(\Lambda)$ ce nombre (notation non standard). On obtient alors la description suivante des différentes strates 
\begin{equation*}
\Gr_{(i,e-i)}(R)=\{\Lambda\ |\ N(\Lambda)=i,\ \mathrm{dim}_k\Lambda_0/\Lambda=e\  \}
\end{equation*}
\item On s'intéresse au produit de $e$-convolution pour le groupe $\mathrm{GL}_2$ . Pour des cocaractères $\mu_i=(1,0)$ on a la description suivante
\begin{equation*}
\Gr_{(1,0)}\tilde{\times}\dots\tilde{\times}\Gr_{(1,0)}(R)=\left\{\Lambda_e\subset\dots\subset\Lambda_1\subset\Lambda_0\ |\ \ u\Lambda_i\subset\Lambda_{i-1},\ \mathrm{dim}_k\Lambda_i/\Lambda_{i-1}=1\  \right\}
\end{equation*}
Le morphisme de convolution prend la forme 
\begin{equation*}
\begin{array}{lrcl}
m:&\Gr_{(1,0)}\tilde{\times}\dots\tilde{\times}\Gr_{(1,0)}&\longrightarrow& \Gr_{\leq (e,0)}\\
&(\Lambda_e\subset\dots\Lambda_1\subset\Lambda_0)&\longmapsto& \Lambda_e
\end{array}
\end{equation*}
\item  On peut donner une autre interprétation du produit de convolution précédent. On définit 
\begin{equation}
\label{Mdef}
M(R)=\left\{\Lambda_1\subset\dots\subset\Lambda_e\subset\Lambda_0\ \bigg|\ \begin{array}{lrcl}\forall\ i<e : u\Lambda_i\subset\Lambda_{i-1},\ \mathrm{dim}_k\Lambda_i/\Lambda_{i-1}=1,\\
\Lambda_e\in\Gr_{\leq (e,0)}
\end{array} \  \right\}
\end{equation} 
On dispose d'un isomorphisme 
\begin{equation}
\label{Miso1}
\begin{array}{lrcl}
\Gr_{(1,0)}\tilde{\times}\dots\tilde{\times}\Gr_{(1,0)}&\longrightarrow&M\\
(\Lambda_e\subset\dots\subset\Lambda_1\subset\Lambda_0)&\longmapsto& (u^{e-1}\Lambda_1\subset\dots\subset u\Lambda_{e-1}\subset\Lambda_e)
\end{array}
\end{equation}
Cet isomorphisme s'insère dans le diagramme commutatif suivant 
\begin{equation}
\label{Miso2}
\begin{tikzcd}
	{\mathrm{Gr}_{(1,0)}\tilde{\times}\dots\tilde{\times}\mathrm{Gr}_{(1,0)}} && M \\
	& {\mathrm{Gr}_{\leq (e,0)}}
	\arrow[from=1-1, to=1-3]
	\arrow["\pi", from=1-3, to=2-2]
	\arrow["m"', from=1-1, to=2-2]
\end{tikzcd}
\end{equation}

où $\pi:(\Lambda_1\subset\dots\subset\Lambda_e\subset\Lambda_0)\mapsto \Lambda_e$.
\end{enumerate}
\end{example}
\begin{remark}
Dans les exemples ci-dessus nous avons fait quelques abus de notations. Si $R\in\mathrm{Alg}_k$ est une $k$-algèbre et $(\mathcal{E},\beta)\in \Gr(R)$, alors $\mathrm{Inv}(\beta)$ n'est pas bien défini : la position relative n'est définie qu'en un point $x\in\spec\, R$. En particulier nous aurions dû adopter la notation plus rigoureuse :
\begin{equation*}
\Gr_{\leq \mu}(R)=\{(\mathcal{E},\beta)\ |\ \mathrm{Inv}_x(\beta)\leq \mu\, \text{ pour tout}\, x\in\spec\, R\}
\end{equation*}
Nous n'avons donc pas nécessairement $\mathrm{Inv}_x(\beta)=\mathrm{Inv}_{x'}(\beta)$ pour $x\neq x'\in\spec\,  R$ (prendre par exemple $\spec\, R$ non connexe).
\end{remark}

\begin{definition}(Non standard)
Soit $\Lambda\subset k(\!(u)\!)^n$ un réseau défini par un point $x\in\Gr$. On définit :
\begin{equation*}
\mathrm{Hodge}(x)=\mathrm{Hodge}(\Lambda):=\mathrm{Inv}(\Lambda,\Lambda_0)\in\mathbb{X}_{\bullet}(T)^+
\end{equation*}
\end{definition}
\begin{remark}
Concrètement pour $\mathrm{GL}_n$ si $\Lambda\subset k(\!(u)\!)^n$ est un réseau alors pour $N$ assez grand $u^N\Lambda_0\subset \Lambda$ et l'invariant $\mathrm{Hodge}(\Lambda)=(a_1,\dots,a_n)$ est caractérisé par :
\begin{equation*}
\Lambda/u^N\Lambda_0\simeq\bigoplus_{i=1}^n k[u]/(u^{N-a_i})
\end{equation*}
\end{remark}
\begin{remark}
Si $(\Lambda_1\subset\dots\subset\Lambda_e)$ est un point de $\widetilde{\Gr}_{\mu_{\bullet}}$ alors on notera $\mathrm{Hodge}(\Lambda_k)=\mathrm{Inv}(\Lambda_k,\Lambda_0)\in\mathbb{X}_{\bullet}(T)^+$. 
\end{remark}
\begin{example}
Les invariants de Hodge de la filtration :
\begin{equation*}
\Lambda_1=\langle u^2 e_1,u^3e_2\rangle \subset \Lambda_2=\langle u^2 e_1,u^2e_2\rangle\subset \Lambda_3=\langle u^2 e_1,ue_2\rangle
\end{equation*}
sont :
\begin{equation*}
\mathrm{Hodge}(\Lambda_1)=(3,2),\ \ \mathrm{Hodge}(\Lambda_2)=(2,2),\ \ \mathrm{Hodge}(\Lambda_3)=(2,1)
\end{equation*}
\end{example}

\begin{remark}
Cette notation prendra sens lorsque nous aurons relié la stratification de Hodge de notre variété de Shimura à celle de la Grassmannienne affine (voir \ref{StrataHodge}).
\end{remark}

\subsection{Résultats}
On en vient maintenant au principal résultat de cet article : 

\begin{prop}
\label{StratGrConv}
Si $\mu_i=(1,0)$ pour tout $i=1,...,e$, alors $\{m^{-1}(\mathrm{Gr}_{\lambda})\}_{\lambda\leq |\mu|}$ définit une bonne stratification de $\mathrm{Gr}_{\mu_1}\tilde{\times}\dots\tilde{\times}\mathrm{Gr}_{\mu_e}$. En d'autres termes si on note $X_{\lambda}:=m^{-1}(\mathrm{Gr}_{\lambda})$ alors pour tout $\lambda\leq |\mu_{\bullet}|$ 
\begin{equation*}
\overline{X_{\lambda}}=\bigcup_{\lambda '\leq\lambda}X_{\lambda '}
\end{equation*}
\end{prop}

\begin{proof}

Nous allons démontrer l'assertion pour l'espace de module $M$ de l'exemple (\ref{ExempleGr}) (équation (\ref{Mdef})) car c'est cet espace de module que nous allons considérer par la suite dans le cadre des modèles entiers des variétés de Shimura. Bien sûr pour obtenir le résultat pour $\mathrm{Gr}_{\mu_1}\tilde{\times}\dots\tilde{\times}\mathrm{Gr}_{\mu_e}$ il suffit de réécrire la preuve ci dessous en appliquant l'isomorphisme (\ref{Miso1}) et d'utiliser la commutativité du diagramme (\ref{Miso2}).\\

Soit $\lambda =(i,j)<(e,0)=|\mu_{\bullet}|$ (pour $\lambda=(e,0)$ il n'y a rien à démontrer). Soit $x\in X_{\lambda}$ un point de corps résiduel $k$. Notons $\Lambda_1\subset\dots\subset\Lambda_e$ la filtration associée. Soit $(v_{k})_{k\geqslant 1}$ des éléments tels que $\Lambda_{k}=\Lambda_{k-1}\oplus (k\cdot v_k)$ (somme directe en tant que $k$-espaces vectoriels).  On définit des entiers $(s_k)_k$ pour tout $k\geq 1$ 
\begin{equation*}
s_k=\mathrm{min}\{s\ |\ u^s\Lambda_k\subset u^e\Lambda_0\}=N(\Lambda_k)
\end{equation*}
où $N(\Lambda_k)$ est l'entier définit en (\ref{ExempleGr}) et $\Lambda_0=k\kt^2$. On a 
\begin{equation*}
\mathrm{Hodge}(\Lambda_k)=(e-k+s_k,e-s_k)
\end{equation*}
On note $k_0=\mathrm{max}\{k|s_k=s_{k-1}\}$.
Notez  que $k_0\neq 0$ car on a supposé $(i,j)<(e,0)$.  Notons $(a,b)=\mathrm{Hodge}(\Lambda_{k_0})$ avec $a\geqslant b$. On a alors par hypothèse  $(a+1,b)=\mathrm{Hodge}(\Lambda_{k_0-1})$. Soit $(e_1,e_2)$ une base de $\Lambda_0$ telle que :
\begin{equation*}
\Lambda_{k_0-1}=\langle u^{a+1}e_1, u^b e_2\rangle\
\end{equation*}
Dans cette base on peut écrire
\begin{equation*}
v_{k_0}=x u^a e_1+y u^{b-1} e_2
\end{equation*}
On a nécessairement $y=0$ car sinon on aurait $u^{s_{k_0}}\cdot v_{k_0}\notin u^e\Lambda_0$. Par conséquent on peut écrire
\begin{equation*}
v_{k_0}=u^{a}e_1,\ \ \ \ \Lambda_{k_0}=\Lambda_{k_0-1}\oplus(k\cdot u^ae_1)=\langle  u^{a}e_1,u^b e_2\rangle
\end{equation*}
Ensuite par définition pour tout $n\geq 1$ on a :
\begin{equation*}
u\cdot v_{k_0+n}\in\Lambda_{k_0+n-1}=\Lambda_{k_0-1}\bigoplus_{\ell=0}^{n-1}(k\cdot v_{k_0+\ell})
\end{equation*}
On peut donc fixer une décomposition par récurrence  :
\begin{equation*}
v_{k_0+n}=\frac{w_n}{u}+\sum_{\ell=0}^{n-1} x_{n,\ell} \frac{v_{k_0+\ell}}{u},\ \ \ \ w_n\in\Lambda_{k_0-1},\  x_{n,\ell}\in k
\end{equation*}
Nous allons maintenant définir une déformation sur $R=k\llbracket t\rrbracket$ de notre filtration initiale. Soit $J:=\{\, n\, |\, x_{n,n-1}=0\, \}$. On travaille dans $(k\llbracket t\rrbracket \otimes k\llbracket u\rrbracket)^{2}$. On commence par déformer trivialement la filtration pour tout $\ell\leq k_0-1$ :
\begin{equation*}
\tilde{\Lambda}_{\ell}:=\Lambda_{\ell}\otimes k\llbracket t\rrbracket\ \ \ \ \forall \ell\leq k_0-1
\end{equation*}
Pour $\ell=k_0$ on définit
\begin{equation*}
\tilde{v}_{k_0}=u^a e_1+t u^{b-1}e_2=v_{k_0}+tu^{b-1}e_2
\end{equation*}
et on pose
\begin{equation*}
\tilde{\Lambda}_{k_0}=\tilde{\Lambda}_{k_0-1}\oplus (k\cdot \tilde{v}_{k_0})
\end{equation*}
Ensuite on déforme par récurrence sur $n$ en fonction de si $n\in J$ ou $n\notin J$ :
\begin{enumerate}
\item $(n\in J)$ Dans ce cas on définit 
\begin{equation*}
\tilde{v}_{k_0+n}=v_{k_0+n}+t \frac{\tilde{v}_{k_0+(n-1)}}{u}
\end{equation*}
\item $(n\notin J)$ Dans ce cas on définit 
\begin{equation*}
\tilde{v}_{k_0+n}=\frac{w_n}{u}+\sum_{\ell=0}^{n-1} x_{n,\ell} \frac{\tilde{v}_{k_0+\ell}}{u}
\end{equation*}
\end{enumerate}
Dans les deux situations on définit la déformation de $\Lambda_{k_0+n}$ comme étant : 
\begin{equation*}
\tilde{\Lambda}_{k_0+n}=\tilde{\Lambda}_{k_0+(n-1)}\oplus (k\cdot\tilde{v}_{k_0+n})
\end{equation*}
Il faut vérifier que l'équation $u\cdot \tilde{v}_{k_0+n}\in\tilde{\Lambda}_{k_0+(n-1)}$ est bien satisfaite. Pour la situation $(2)$ c'est évident. Pour la situation $(1)$ il faut voir que  par hypothèse on a
\begin{equation*}
u\cdot\tilde{v}_{k_0+n}= w_n+\sum_{\ell=0}^{n-2}x_{n,\ell}\tilde{v}_{k_0+\ell}+t \tilde{v}_{k_0+(n-1)}\in\tilde{\Lambda}_{k_0+(n-2)}\oplus (k\cdot\tilde{v}_{k_0+(n-1)})=\tilde{\Lambda}_{k_0+(n-1)}
\end{equation*}
Au point fermé $t=0$ on a   $\tilde{v}_{k_0}=v_{k_0}$ et par suite $\tilde{v}_{k_0+n}=v_{k_0+n}$ pour tout $n\geq 1$. Par conséquent cette filtration correspond bien à une déformation de notre filtration initiale. Pour calculer $\mathrm{Hodge}(\tilde{\Lambda}_e\otimes k(\!(t)\!)) $ il faut voir que par construction on a 
\begin{equation*}
\tilde{s}_{k_0+n}=\tilde{s}_{k_0+n}+1\ \ \forall n\geq 1
\end{equation*}
et que par conséquent puisque $\tilde{s}_{k_0}=s_{k_0}+1$ on trouve 
\begin{equation*}
\tilde{s}_{e}=\tilde{s}_{k_0+e-k_0}=\tilde{s}_{k_0}+e-k_0=s_{k_0}+1+e-k_0=s_e+1
\end{equation*}
On a donc bien en fibre générique :
\begin{equation*}
\mathrm{Hodge}(\tilde{\Lambda}_e\otimes k(\!(t)\!))=(i+1,j-1)
\end{equation*}

\end{proof}

\begin{remark}
L'idée de la preuve est la suivante.
\begin{enumerate}
\item Si $\Lambda_e=\langle u^ie_1,u^je_2\rangle$ alors on aimerait déformer sur $k\llbracket t\rrbracket $ en prenant l'élément $\tilde{v}_e=u^ie_1+tu^{j-1}e_2$. Le problème est que cet élément ne satisfait pas nécessairement $u\cdot\tilde{v}_e\in\Lambda_{e-1}$. Il faut donc déformer $\Lambda_{e-1}$ également. Le problème est que cette déformation doit de nouveau satisfaire l'équation $u\cdot\tilde{\Lambda}_{e-1}\subset \Lambda_{e-2}$...
\item Il existe un rang $k_0$ tel que là déformation $\tilde{\Lambda}_{k_0}$ existe. Autrement dit on peut trouver un élément $\tilde{v}_{k_0}$ de la \og bonne valuation\fg, c'est-à-dire celle de $v_{k_0}$ moins $1$.
\item Ensuite on déforme par récurrence les $v_{k_0+n}$ en à \og divisant par $u$\fg à chaque étape de sorte à faire apparaître du $u^{j-1}$ dans la décomposition de $\tilde{v}_e$.
\end{enumerate}
Donnons un exemple explicite de déformation. On considère la filtration
\begin{equation*}
\Lambda_1=\langle u^3e_1, u^2e_2\rangle,\ \ \Lambda_2=\langle u^2e_2,u^2e_1\rangle,\ \Lambda_3=\langle ue_2,u^2e_1\rangle
\end{equation*}
On représente cette filtration par une matrice 
\begin{equation*}
\begin{bNiceMatrix}[first-row,first-col]
    & u^2e_1 & ue_1 & e_1 & u^2e_2 &ue_2&e_2      \\
v_1 &    &    &   & *   & &&\\
v_2 & *   &    &    &    & && \\
v_3 &    &    &    &   &* &
\end{bNiceMatrix}
\end{equation*}
La multiplication par $u$ consiste à décaler les colonnes vers la gauche. La déformation construit dans la preuve précédente est donnée par la matrice :
\begin{equation*}
\begin{bNiceMatrix}[first-row,first-col]
    & u^2e_1 & ue_1 & e_1 & u^2e_2 &ue_2&e_2      \\
v_1 &    &    &   & *   & &&\\
v_2 & *   &    &    &    &t_2 && \\
v_3 &    &  t_3  &    &   &* &t_3t_2
\end{bNiceMatrix}
\end{equation*}
\end{remark}
\begin{remark}
Le théorème ci-dessus n'est pas vrai dans le cas général (voir \cite{BijHer2} Proposition 3.9 pour un contre exemple). En fait dans la preuve ci dessus on utilise un fait spécifique au cas $G=\mathrm{GL}_2$ : si $\Lambda\in\Gr_{\leq (e,0)}$ alors avec les notations de \ref{ExempleGr} :
\begin{equation*}
\Lambda\in\Gr_{(i,j)}\ \Leftrightarrow\ N(\Lambda)=i
\end{equation*}
L'invariant $N(\Lambda)$ est égal à l'indice de nilpotence de $u\in\mathrm{End}(\Lambda/u^e\Lambda_0)$ ce qui rend le calcul de $\mathrm{Hodge}(\Lambda)=\mathrm{Inv}(\Lambda,\Lambda_0)$ beaucoup plus simple à calculer en pratique. La preuve du Théorème \ref{StratGrConv} consiste simplement à déformer une filtration $\Lambda_1\subset\dots\subset \Lambda_e$ de sorte à faire apparaître le bon indice de nilpotence en fibre générique.
\end{remark}

\section{Modèles locaux}
\subsection{Notations}
\label{Notations}
Soit $F$ un corps totalement réel de degré $d>1$ sur $\mathbb{Q}$. On note $\mathcal{O}_F$ son anneau d'entiers. Pour tout $v|p$ on note $e_v$ l'indice de ramification et $f_v$ le degré résiduel. On note $F_v$  la complétion de $F$ en $v$ et $\mathcal{O}_v$ son anneau d'entiers. On note $F^{\mathrm{nr}}_v$ la sous extension maximale non ramifiée et $\mathcal{O}_v^{\mathrm{nr}}$ son anneau d'entiers. Soit $K/\mathbb{Q}_p$ une extension qui contient tous les plongement $F_v\rightarrow \overline{\mathbb{Q}}_p$ pour tout $v|p$. On note $\mathcal{O}_K$ son anneau d'entiers, $k$ son corps résiduel, et on fixe une uniformisante $\varpi\in\mathcal{O}_K$. On dispose d'une décomposition 

\begin{equation}
\label{Eq1}
\mathcal{O}_F\otimes_{\mathbb{Z}}\mathcal{O}_K \cong \prod_{v|p} \prod_{\tau\in \Sigma_v^{\mathrm{nr}}}\mathcal{O}_v\otimes_{\mathcal{O}_v^{\mathrm{nr}},\tau}\mathcal{O}_K
\end{equation}
où $\Sigma_v^{\mathrm{nr}}=\mathrm{Hom}(F^{\mathrm{nr}}_v,\overline{\mathbb{Q}}_p)$. Si on fixe une uniformisante $\varpi_v$ de $\mathcal{O}_{v}$ alors on peut identifier 
\begin{equation}
\label{Eq2}
\mathcal{O}_v\otimes_{\mathcal{O}_v^{\mathrm{nr}},\tau} k\simeq k[u] /(u^{e_v})
\end{equation} 
On obtient donc une décomposition non canonique :
\begin{equation*}
\mathcal{O}_F\otimes_{\mathbb{Z}}k \simeq \prod_{v|p} \prod_{\tau\in \Sigma_v^{\mathrm{nr}}} k[u]/(u^{e_v})
\end{equation*}

\subsection{Modèle local PEL}
\label{ModeleLocalPEL}
Le groupe réductif qui nous intéresse est : 
\begin{equation*}
G=\mathrm{Res}_{F/\Q}\mathrm{GL}_2
\end{equation*}
 Lorsqu'on le change de base à $K$ on obtient une décomposition
\begin{equation*}
G\otimes_{\Q}K=\prod_{v|p}\prod_{\tau\in\Sigma_v^{\mathrm{nr}}}(\mathrm{Res}_{F_{v}/F_v^{\mathrm{nr}}}\mathrm{GL}_2)\otimes_{\Q_p}K
\end{equation*}

Pour simplifier les notations nous allons dans un premier temps décrire le modèle local PEL pour le groupe $G=\mathrm{Res}_{L/L^{\mathrm{nr}}}\mathrm{GL}_2$ avec $L/\Q_p$ une extension finie d'indice de ramification $e$ (jouant le rôle de $F_v/\Q_p$). On fixe un plongement $\tau:L^{\mathrm{nr}}\hookrightarrow \overline{\Q}_p$ et on note $\Sigma_{\tau}=\mathrm{Hom}_{L^{\mathrm{nr}}}(L,\overline{\Q}_p)$ l'ensemble des plongements qui prolongent $\tau$. On fixe une numérotation $\Sigma_{\tau}\simeq\{\varphi_1,\dots ,\varphi_e\}$. \\

Soit $V$ un $L$-espace vectoriel de dimension $2$. On fixe une base $(e_1,e_2)$ de $V$. On note $\Lambda$ le $\mathcal{O}_L$-module libre de base $(e_1,e_2)$. Le modèle local PEL pour le groupe $G$ noté $M^{\mathrm{PEL}}$, est le schéma sur $\spec\, \OK_L^{\mathrm{nr}}$ représentant le foncteur qui à $(S\rightarrow \spec\, \OK_L^{\mathrm{nr}})$ associe l'ensemble  $M^{\mathrm{PEL}}(S)$ des $\mathcal{O}_L\otimes_{\OK_L^{\mathrm{nr}}}\mathcal{O}_S$-sous modules $\mathscr{F}\subset \Lambda_{S}:=\Lambda\otimes_{\OK_L^{\mathrm{nr}}}\mathcal{O}_S$ tels que 
\begin{itemize}[label=\textbullet]
\item $\mathscr{F}$ est localement sur $S$ (pour la topologie Zariski) un $\mathcal{O}_S$-facteur direct de $\Lambda_{S}$ de rang $e$.
\item Pour tout $a\in \mathcal{O}_L$ on a l'égalité polynomiale suivante 
\begin{equation*}
\mathrm{det}(a\ |\ \mathscr{F})=\prod_{\varphi_i\in\Sigma_{\tau}}\varphi_i(a)
\end{equation*}
\end{itemize}

On note $\mathcal{G}=\underline{\mathrm{Aut}}_{\mathcal{O}_L}(\Lambda)$ le schéma en groupe sur $\spec\, \OK_L^{\mathrm{nr}}$ des automorphismes de $\Lambda$ compatibles avec l'action de $\mathcal{O}_L$. On a alors la proposition suivante :

\begin{prop}
$\mathcal{G}$ est lisse sur $\spec\, \OK_L^{\mathrm{nr}}$
\end{prop}
\begin{proof}
Voir \cite{RapZink} (Proposition A.4).
\end{proof}
En fait la proposition ci dessus est démontrée pour $\mathcal{G}=\underline{\mathrm{Aut}}_{\mathcal{O}_L}((\Lambda)_{i\in I})$ où $(\Lambda)_{i\in I}$ est une chaîne périodique de réseaux (voir \cite{RapZink}). Cette situation apparaît lorsque l'on autorise du niveau en $p$, ce qui n'est pas notre cas ici. Dans notre situation ce groupe est en fait explicite :
\begin{equation*}
\mathcal{G}=\mathrm{Res}_{\mathcal{O}_L/\OK_L^{\mathrm{nr}}}\mathrm{GL}_2
\end{equation*}
et est bien sûr lisse.

\subsection{Modèle local de Pappas-Rapoport}

On considère maintenant le foncteur $M^{\mathrm{PR}}$ qui à un schéma $(S\rightarrow\spec\, \mathcal{O}_K)$ associe l'ensemble $M^{\mathrm{PR}}(S)$ des filtrations $(\mathscr{F}^{(i)})_{i=1,...,e}$ de $\mathcal{O}_L\otimes_{\OK_{L^{\mathrm{nr}}}}\mathcal{O}_S$-sous modules de $\Lambda_{S}$ :
\begin{equation*}
0=\mathscr{F}^{(0)}\subset \mathscr{F}^{(1)}\subset\cdots\subset \mathscr{F}^{(e)}\subset\Lambda_{S}
\end{equation*}
telles que :
\begin{itemize}[label=\textbullet]
\item Les $\mathscr{F}^{(i)}$ sont Zariski-localement des $\mathcal{O}_S$-facteurs directs de $\Lambda_{S}$ de rang $i$
\item Pour tout $a\in L$ et pour tout $i=1,...,e$ 
\begin{equation*}
\label{ConditionA}
(a\otimes 1-1\otimes\varphi_i(a))\cdot \mathscr{F}^{(i)}\subset \mathscr{F}^{(i-1)}
\end{equation*}
\end{itemize}
\begin{prop}
\label{MPRrep}
Ce foncteur est représentable par uns schéma projectif sur $\spec\, \OK_K$.
\end{prop}
\begin{proof}
Il suffit de voir que l'on peut le plonger dans un produit de Grassmanniennes convenables. Plus précisément, on peut utiliser \ref{MPRIsoConv} et le fait que le produit de convolution soit représentable (voir \ref{ConvolutionRep}).
\end{proof}
 On dispose d'un morphisme d'oubli 
\begin{equation*}
\begin{array}{lrcl}
&\pi:\Split&\longrightarrow& \Naif\otimes_{\OK_L^{\mathrm{nr}}}\OK_K\\
&(\mathscr{F}^{(i)})&\longmapsto& \mathscr{F}^{(e)}
\end{array}
\end{equation*}

Désormais pour alléger les notations nous noterons de la même manière $\Naif=\Naif\otimes_{\OK_L^{\mathrm{nr}}}\OK_K$.

\subsection{Plongement dans les grassmanniennes affines}

Nous allons maintenant décrire les plongements des fibres spéciales des modèles $\Split$ et $\Naif$ dans certaines grassmanniennes affines. On suit presque à la lettre \cite{PRII}. On note $\Naiff=\Naif\otimes_{\OK_K}k$ et $\Splitt=\Split\otimes_{\OK_K}k$ les fibres spéciales de ces deux modèles. Comme en \eqref{Eq2}  le choix d'une uniformisante $\varpi\in\OK_L$ nous fournit une identification :
\begin{equation*}
\OK_L\otimes_{\Z_p} k\simeq k[\![u]\!]/(u^e),\ \ \ \ \varpi\otimes 1\mapsto u
\end{equation*}
Cela induit un isomorphisme de  $\OK_L\otimes_{\Z_p} k$-modules :
\begin{equation*}
\Lambda\otimes_{\Z_p} k\simeq \Ltilde\otimes_{k\kt} k[\![u]\!]/(u^e)
\end{equation*}
(on rappelle que $\Lambda\subset V$ est un $\OK_L$-réseau fixé (\ref{ModeleLocalPEL}) et que $\Lambda_0=k\kt^2$).
On note $p$ la projection :
\begin{equation*}
p:\Ltilde\rightarrow\Ltilde\otimes_{k\kt} k[\![u]\!]/(u^e)
\end{equation*}
Soit $(S\rightarrow\spec\, k)$ un schéma et $\mathscr{F}\in \Naiff(S)$. L'identification précédente permet de voir $\mathscr{F}$ comme un sous module de $\Ltilde\otimes_{k\kt} \OK_S[\![u]\!]/(u^e)$. On définit le $\OK_S\kt$-module $\Lambda_{\mathscr{F}}$ comme étant 
\begin{equation}
\label{LambdaF}
\Lambda_{\mathscr{F}}:=p^{-1}\bigg(\mathscr{F}\subset \Ltilde\otimes_{k\kt} \OK_S[\![u]\!]/(u^e)\bigg)
\end{equation}
On obtient donc finalement par ce procédé une immersion fermée
\begin{equation*}
\overline{M}^{\mathrm{PEL}}\lhook\joinrel\nrightarrow \Gr
\end{equation*}
\begin{remark}
Par construction on dispose des inclusions de réseaux

\begin{equation*}
u^e \Lambda_{0,S}\subset\Lambda_{\mathscr{F}}\subset \Lambda_{0,S}
\end{equation*} 
dont les gradués sont des $\OK_S$-modules localement libres de rang $e$. Notez que l'image de cette immersion est entièrement caractérisée par les inclusions ci-dessus et le rang des gradués.
\end{remark}

\begin{prop}
\label{PlongPEL}
L'immersion fermée $\iota$ est équivariante pour l'action de $\mathcal{G}\otimes k$ à gauche et $L^+G$ à droite. Elle induit un isomorphisme
\begin{equation*}
\Naiff\simeq  \bigcup_{\lambda\leqslant (e,0)}\mathrm{Gr}_{\lambda}
\end{equation*}
\end{prop}
\begin{proof}
Découle de la remarque précédente et de l'exemple \ref{ExempleGr}.
\end{proof}
\begin{remark}
\label{Admissible}
En d'autres termes l'ensemble des copoids $\mu$-admissibles (voir \cite{Gortz}, section 4.3) est ici tout à fait explicite
\begin{equation*}
\mathrm{Adm}(\mu)_K=\{\lambda\leqslant (e,0)\}
\end{equation*}
(Ici  on est dans le cas parahorique maximal $K=\mathcal{G}(\Z_p)$)
\end{remark}

De la même manière on dispose d'une immersion fermée pour le modèle de Pappas-Rapoport dans un produit de grassmanniennes affines :
\begin{equation}
\label{PlongPR}
\begin{array}{lrcl}
&\iota^{\mathrm{PR}}:\Splitt&\longrightarrow& \Gr\times...\times \Gr\\
&(\mathscr{F}^{(1)},...,\mathscr{F}^{(e)})&\longmapsto& (\Lambda_{\mathscr{F}^{(1)}},\dots,\Lambda_{\mathscr{F}^{(e)}})
\end{array}
\end{equation}

\begin{remark}
Par construction on dispose des inclusions 
\begin{equation*}
u^e \Lambda_{0}\subset\Lambda_{\mathscr{F}^{(1)}}\subset...\subset  \Lambda_{\mathscr{F}^{(e)}}\subset \Lambda_{0}
\end{equation*} 
dont les gradués sont localement libres de rang $1$ pour $j=1,...,e-1$ et de rang $e$ pour $j=e$. D'après l'exemple \ref{ExempleGr} cette immersion fermée induit un isomorphisme $\overline{M}^{\mathrm{PR}}\simeq M$ où $M$ est l'espace de module décrit en \eqref{Mdef}.
\end{remark}

\begin{prop}
\label{MPRIsoConv}
L'immersion fermée $\iota^{\mathrm{PR}}$ induit un isomorphisme équivariant pour l'action de $\mathcal{G}$ à gauche et $L^+G$ à droite :
\begin{equation*}
\Splitt\simeq  \mathrm{Gr}_{(1,0)}\tilde{\times}\dots \tilde{\times}\mathrm{Gr}_{(1,0)}
\end{equation*}
où le produit de convolution est pris $e$ fois.
\end{prop}
\begin{proof}
Découle de la remarque précédente en composant avec l'isomorphisme \eqref{Miso1}.
\end{proof}

\begin{prop}
\label{CarreCart}
Le carré suivant est cartésien
\[\begin{tikzcd}
	{\overline{M}^{\mathrm{PR}}} & { \mathrm{Gr}_{(1,0)}\tilde{\times}\dots \tilde{\times}\mathrm{Gr}_{(1,0)}} \\
	{\overline{M}^{\mathrm{PEL}}} & {\mathrm{Gr}_{\leq (e,0)}}
	\arrow["\pi"', from=1-1, to=2-1]
	\arrow[from=2-1, to=2-2]
	\arrow["m", from=1-2, to=2-2]
	\arrow[from=1-1, to=1-2]
	\arrow[from=1-2, to=2-2]
\end{tikzcd}\]
\end{prop}
\begin{proof}
Il s'agit essentiellement de montrer que le diagramme suivant est commutatif. Cela découle du fait que le diagramme \eqref{Miso2} est commutatif.
\end{proof}

\subsection{Cas général}
Revenons maintenant au cas général où l'on considère une extension $F/\Q$ de degré $d$. On note $G=\mathrm{Res}_{F/\Q}\mathrm{GL}_2$. Pour tout $v|p$ et tout $\tau\in\Sigma_v^{\mathrm{nr}}$, on note $M_{v,\tau}^{\mathrm{PEL}}$ le modèle local pour le groupe $\mathrm{Res}_{F_{v}/F_v^{\mathrm{nr}}}\mathrm{GL}_2$ définit dans la section précédente.\\

\subsubsection{Modèles locaux}
\label{LambdaPEL}
Soit $V$ un $F$-espace vectoriel de dimension $2$. On fixe une base $(e_1,e_2)$ de $V$. On note $\Lambda$ le $\mathcal{O}_F$-module libre de base $(e_1,e_2)$. On fixe (à conjugaison près) un cocaractère $\mu :\mathbb{G}_{m,\mathbb{C}}\rightarrow G_{\mathbb{C}}$, et on suppose que ce dernier induit une décomposition en espaces propres : 
\begin{equation*}
V\otimes \mathbb{C}=V_0\oplus V_1
\end{equation*}
Enfin on suppose que le corps de définition de $\mu$ est $\Q$ (ces hypothèses seront vérifiées dans le cadre des variétés de Shimura de type Hilbert). On définit le modèle local $M^{\mathrm{PEL}}$ comme étant le foncteur qui à $(S\rightarrow\spec\, \Z_p)$ associe l'ensemble 
$M^{\mathrm{PEL}}(S)$ des $\mathcal{O}_F\otimes_{\Z_p}\mathcal{O}_S$-sous modules $\mathscr{F}\subset \Lambda_{S}:=\Lambda\otimes_{\Z_p}\mathcal{O}_S$ tels que 
\begin{itemize}[label=\textbullet]
\item $\mathscr{F}$ est Zariski-localement sur $S$  un $\mathcal{O}_S$-facteur direct de $\Lambda_{S}$ de rang $[F:\Q]=d$;
\item Pour tout $a\in \mathcal{O}_F$ on a l'égalité polynomiale suivante (condition de Kottwitz) :
\begin{equation*}
\mathrm{det}(a\, |\, \mathscr{F})=\mathrm{det}(a\, |\, V_0)
\end{equation*}

\end{itemize}
On note $\mathcal{G}=\underline{\mathrm{Aut}}_{\mathcal{O}_L}(\Lambda)$ le schéma en groupe sur $\spec\, \Z_p$ des automorphismes de $\Lambda$ compatibles avec l'action de $\mathcal{O}_F$. On a alors la proposition suivante :

\begin{prop}
Après changement de base à $\OK_K$, on dispose d'un isomorphisme canonique
\begin{equation*}
M^{\mathrm{PEL}}\otimes_{\Z_p}\OK_K\simeq \prod_{v|p}\prod_{\tau\in\Sigma_v^{\mathrm{nr}}} M_{v,\tau}^{\mathrm{PEL}}
\end{equation*}
\end{prop}
\begin{proof}
Découle de \eqref{Eq1}.
\end{proof}
Comme précédemment on note $\mathcal{G}=\underline{\mathrm{Aut}}_{\OK_F}(\Lambda)$. On a alors pour les mêmes raisons une décomposition 
\begin{equation*}
\mathcal{G}\otimes \Z_p\simeq\prod_{v|p}\mathrm{Res}_{\mathcal{O}_v/\Z_p}\mathrm{GL}_2
\end{equation*}
En particulier $\mathcal{G}\otimes \Z_p$ est bien lisse. \\

Cette décomposition suggère la définition suivante pour le modèle de Pappas-Rapoport dans le cas général d'une extension $F/\Q$
\begin{definition} Le modèle local de Pappas-Rapoport pour le groupe $G=\mathrm{Res}_{F/\Q}\mathrm{GL}_2$ est :
\begin{equation*}
M^{\mathrm{PR}}= \prod_{v|p}\prod_{\tau\in\Sigma_v^{\mathrm{nr}}} M_{v,\tau}^{\mathrm{PR}}
\end{equation*}
\end{definition}

\subsubsection{Plongements dans les grassmanniennes affines}
La \og compatibilité\fg des grassmannienne affines avec le produit  vu au Lemme \ref{lemmeG1G2} combinée aux décompositions des modèles locaux de la section précédente nous donne sans trop d'efforts les isomorphismes non canoniques (dépend des différents choix d'uniformisantes) :
\begin{equation*}
\overline{M}^{\mathrm{PEL}}\simeq \prod_{v|p}\prod_{\tau\in\Sigma_v^{\mathrm{nr}}} \mathrm{Gr}_{\leq(e_v,0)},\ \ \ \ \ \ \overline{M}^{\mathrm{PR}}\simeq \prod_{v|p}\prod_{\tau\in\Sigma_v^{\mathrm{nr}}} \widetilde{\mathrm{Gr}}_{(e_v,0)}
\end{equation*}
où $\widetilde{\mathrm{Gr}}_{(e_v,0)}$ désigne le produit de convolution de $e_v$ copie de $\mathrm{Gr}_{(1,0)}$. Finalement le carré cartésien de la proposition \ref{CarreCart} devient 
\begin{equation}
\label{CarreCart3}
\begin{tikzcd}
	{\overline{M}^{\mathrm{PR}}} & {\prod_{v|p}\prod_{\tau\in\Sigma_v^{\mathrm{nr}}} \widetilde{\mathrm{Gr}}_{(e_v,0)}} \\
	{\overline{M}^{\mathrm{PEL}}} & {\prod_{v|p}\prod_{\tau\in\Sigma_v^{\mathrm{nr}}} \mathrm{Gr}_{\leq(e_v,0)}}
	\arrow["\pi"', from=1-1, to=2-1]
	\arrow[from=2-1, to=2-2]
	\arrow[from=1-1, to=1-2]
	\arrow["{(m_{v,\tau})_{v,\tau}}", from=1-2, to=2-2]
\end{tikzcd}
\end{equation}

\section{Modèles entiers}

\subsection{Modèle entier PEL}
\label{SPEL}

On garde les notations de \ref{Notations} et \ref{LambdaPEL}. On considère l'espace de module $\NNaif$ sur $\spec\, \Z_p$ qui à un schéma localement noethérien $(S\rightarrow\spec\, \Z_p)$ associe l'ensemble des quadruplés $(A,\lambda,\iota,\eta)$ à $\Z_{(p)}^{\times}$-isogénie où :
\begin{enumerate}
\item $A\rightarrow S$ est un schéma abélien de dimension $g=d=[F:\Q]$.
\item $\lambda:A\rightarrow A^{\vee}$ est une $\Z_{(p)}^{\times}$-polarisation
\item $\eta$ est une structure de niveau en dehors de $p$ (voir \cite{KWLan}, section 1.4.1)
\item $\iota : \OK_F\hookrightarrow \mathrm{End}(A)\otimes_{\Z}\Z_{(p)}$ morphisme respectant les involutions des deux cotés.
\item $(A,\lambda,\iota,\eta)$ satisfait la condition de déterminant de Kottwitz :
\begin{equation*}
\mathrm{det}(a\, |\, \omega_{A/S})=\mathrm{det}(a\, |\, V_0)\ \ \ \ \ \ \forall\, a\in\OK_F
\end{equation*}
(voir \cite{KWLan} Définition 1.3.4.1 pour plus de détails).
\end{enumerate}
On a alors le résultat suivant dû à Mumford puis Kottwitz :
\begin{prop}[\cite{GIT}, \cite{KottwitzPEL} section 5]
\label{PELrepresentable}
$\NNaif$ est représentable par un schéma quasi-projectif sur $\spec\, \Z_p$.
\end{prop}

\begin{remark}
L'espace de module définit ci-dessus est celui définit par Kottwitz. On consultera l'article de I.Vollaard \cite{Vollaard} pour plus de détails concernant l'équivalence des différents espaces de modules considérés dans le cas Hilbert.
\end{remark}

Nous allons maintenant expliquer le lien entre le modèle entier $\NNaif$ est le modèle local $\Naif$. On considère le foncteur $\NNNaif$ qui à un schéma localement noethérien $(S\rightarrow\spec\, \Z_p)$ associe les $5$-uplés $(A,\lambda,\iota,\eta,\gamma)$ où 
\begin{equation*}
\gamma:H_{\mathrm{dR}}^1(A/S)\simeq \Lambda\otimes_{\Z_p}\OK_S
\end{equation*}
est une trivialisation  du premier groupe de cohomologie de deRham (en tant que $\OK_F\otimes_{\Z_p}\OK_S$-module).  Notez que $\NNNaif$ est muni d'une action de $\mathcal{G}=\underline{\mathrm{Aut}}_{\OK_F}(\Lambda)$ où ce dernier agit sur la trivialisation $\gamma$. L'oubli de la trivialisation $\gamma$ fournit un morphisme 
\begin{equation*}
\varphi:\NNNaif\rightarrow \NNaif
\end{equation*}
L'un des résultats majeurs de \cite{RapZink} est le suivant
\begin{prop}[\cite{RapZink} Théorème 3.16]
$\varphi:\NNNaif\rightarrow \NNaif$ est un $\mathcal{G}$-torseur.
\end{prop}

La trivialisation $\gamma$ nous fournit un morphisme vers le modèle local PEL
\begin{equation*}
\begin{array}{lrcl}
\psi: &\NNNaif&\longrightarrow &\Naif\\
&(A,\lambda,\iota,\gamma)&\longmapsto & \gamma (\omega_{A/S})\subset\Lambda_S
\end{array}
\end{equation*}
Le fait que ce morphisme soit bien défini découle essentiellement de la définition de l'espace de module $\NNaif$. On a alors la proposition suivante qui résulte de \cite{deJong} et \cite{RapZink} :
\begin{prop}
$\psi:\NNNaif\longrightarrow \Naif$ est un morphisme lisse $\mathcal{G}$-équivariant de dimension relative $\mathrm{dim}\, \mathcal{G}$.
\end{prop}
\begin{proof}
Le fait que le morphisme soit $\mathcal{G}$-équivariant est évident. La lissité découle de Grothendieck-Messing, voir \cite{deJong} (Proposition 4.5 du document .dvi du même nom sur sa page web).
\end{proof}\ \\

En d'autres termes on dispose d'un diagramme de modèle local au sens de \cite{RapZink} :
\[\begin{tikzcd}
	& {\mathcal{S}h^{\mathrm{PEL},\square}} \\
	{\mathcal{S}h^{\mathrm{PEL}}} && {M^{\mathrm{PEL}}}
	\arrow["\varphi"', from=1-2, to=2-1]
	\arrow["\psi", from=1-2, to=2-3]
\end{tikzcd}\]
Ce qui correspond à un morphisme de champs algébriques
\begin{equation}
\label{StoMPEL}
\NNaif\rightarrow \big[\Naif/\mathcal{G}\big]
\end{equation}
lisse de dimension relative $\mathrm{dim}\, \mathcal{G}$.

\subsection{Modèle entier de Pappas-Rapoport}

On définit le modèle entier de Pappas-Rapoport $\mathcal{S}h^{\mathrm{PEL}}$ comme le produit cartésien :
\begin{equation}
\label{CarreCart2}
\begin{tikzcd}
	{\mathcal{S}h^{\mathrm{PR}}} & {\big[M^{\mathrm{PR}}/\mathcal{G}_{\mathcal{O}_K}\big]} \\
	{\mathcal{S}h^{\mathrm{PEL}}_{\mathcal{O}_K}} & {\big[M^{\mathrm{PEL}}_{\mathcal{O}_K}/\mathcal{G}_{\mathcal{O}_K}\big]}
	\arrow[from=1-1, to=1-2]
	\arrow[from=1-2, to=2-2]
	\arrow[from=1-1, to=2-1]
	\arrow[from=2-1, to=2-2]
\end{tikzcd}
\end{equation}
où l'indice $(\cdot)_{\mathcal{O}_K}$ désigne le changement de base $\otimes_{\mathbb{Z}_p}\OK_K$.

Avant de rendre la définition de ce modèle explicite, mentionnons un corollaire direct de la Proposition \ref{PELrepresentable} :

\begin{prop}
 $\mathcal{S}h^{\mathrm{PR}}$ est représentable par un schéma quasi projectif lisse sur $\spec\, \OK_K$.
\end{prop}
\begin{proof}
Puisque le carré \eqref{CarreCart2} est cartésien il suffit de montrer que $\Split$ est représentable. Cela découle de la proposition \ref{MPRrep}. Le même raisonnement et le fait que le morphisme de convolution soit propre montrent que $\SSplit$ est quasi projectif puisque $\NNaif$ l'est. Pour la lissité, il suffit de voir que le modèle local $\Split$ est lisse. D'après \cite{PRII} Théorème 5.3 le modèle local est plat par conséquent il suffit de montrer que la fibre spéciale est lisse (la fibre générique étant toujours lisse). La lissité de $\overline{\Split}$ découle de la Proposition \ref{LisseConv}.
\end{proof}

\begin{remark}
En fait la preuve du théorème 5.3 de \cite{PRII} repose sur un diagramme de torseurs reliant le modèle local $\Split$ à un produit de modèles locaux non ramifiés. Dans le cas PEL sans niveau en $p$, ces derniers sont lisses (et pas seulement plats) ce qui donne directement la lissité sur $\spec\, \OK_K$, sans passer par la fibre spéciale.
\end{remark}

\subsection{Donnée de Pappas-Rapoport}

Nous allons maintenant donner une définition plus explicite du  modèle entier de Pappas-Rapoport. Pour ce faire nous allons devoir définir la notion de \og donnée de Pappas-Rapoport\fg pour un groupe $p$-divisible. 

Pour simplifier les notations, comme pour le \S\ref{ModeleLocalPEL} nous allons ici travailler avec une seule place $v|p$. Soient $L/\Q_p$ une extension finie de degré $d>1$ (jouant le rôle de $F_v/\Q_p$), $K/\Q_p$ une extension contenant la clôture Galoisienne de $L$, et $S$ un schéma sur $\spec\, \mathcal{O}_K$. On note $L^{nr}$ l'extension maximale non ramifiée dans $L$. On note $e$ l'indice de ramification. Soit $\varpi$ une uniformisante de $L$. On fixe un plongement $\tau:L^{nr}\hookrightarrow K$. On définit l'ensemble $\Sigma_{\tau}=\mathrm{Hom}_{L^{\mathrm{nr}}}(L,\overline{\mathbb{Q}}_p)$ comme étant l'ensemble des plongements $\tau ':L\hookrightarrow K$ qui donne $\tau$ en restriction à $L^{nr}$. C'est un ensemble de cardinal $e$. On fixe un ordre $\Sigma_{\tau}=\{\varphi_1,...,\varphi_e\}$. 
\subsubsection{Définition}
\label{Hypothese}
Soit $\mathscr{F}$ un $\mathcal{O}_S$ module  localement libre muni d'une action de $\mathcal{O}_L$ tel que $\mathcal{O}_{L^{nr}}$ agit sur $\mathscr{F}$ via $\tau$ (on rappelle que $S$ est un schéma sur $\spec\ \mathcal{O}_K$). On suppose qu'il existe un faisceau $\mathscr{E}$ localement libre de rang $h$ en tant que $\mathcal{O}_L\otimes_{\mathcal{O}_{L^{nr}}}\mathcal{O}_S$-module tel que $\mathscr{F}$ soit un localement un facteur direct de $\mathscr{E}$.
\begin{definition}
Une donnée de Pappas-Rapoport pour $(\mathscr{E},\mathscr{F})$ par rapport à $\Sigma_{\tau}$ est une filtration 

\begin{equation*}
0=\mathscr{F}^{(0)}\subset \mathscr{F}^{(1)}\subset\cdots\subset \mathscr{F}^{(e)}=\mathscr{F}
\end{equation*}
telle que pour tout $1\leqslant j\leqslant e$
\begin{itemize}[label=\textbullet]
\item Les $\mathscr{F}^{(j)}$ sont localement des $\mathcal{O}_S$-facteurs directs, stables par $\mathcal{O}_L$.
\item $([\varpi]-\varphi_j(\varpi))\cdot \mathscr{F}^{(j)}\subset \mathscr{F}^{(j-1)}$
\item $\mathscr{F}^{(j)}/\mathscr{F}^{(j-1)}$ est localement libre de rang $1$
\end{itemize}
\end{definition}
La deuxième condition impose que l'action de $\mathcal{O}_L$ sur le quotient $\mathscr{F}^{(j)}/\mathscr{F}^{(j-1)}$ se fasse via le plongement $\mathcal{O}_L\xrightarrow{\varphi_j} \mathcal{O}_K$.

\subsubsection{Groupes $p$-divisibles}

Soit $S$ un schéma sur $\spec\ \mathcal{O}_K$. Soit $G$ un groupe $p$-divisible sur $S$ de hauteur $hd$ muni d'une action $\iota:\mathcal{O}_L\rightarrow\mathrm{End}_S(G)$. Notons $\mathscr{E}(G)$ le cristal associé. Avec les notations de \eqref{EqHodgeFil}, en évaluant ce cristal sur l'épaississement tautologique $(S\rightarrow S)$ on obtient la filtration de Hodge :
\begin{equation*}
0\longrightarrow \omega_{G}\longrightarrow\mathscr{E}(G)_{(S\rightarrow S)}\longrightarrow\omega_{G^D}^{\vee}\longrightarrow 0
\end{equation*}
D'après la Proposition \ref{Fargues} le faisceau $\mathscr{E}(G)_{(S\rightarrow S)}$ est libre en tant que $\mathcal{O}_L\otimes_{\mathcal{O}_{L^{nr}}}\mathcal{O}_S$-module et l'hypothèse \ref{Hypothese} est donc bien satisfaite. Pour simplifier les notations nous noterons $\mathscr{E}=\mathscr{E}(G)_{(S\rightarrow S)}$.
Cette filtration est compatible avec les décompositions induites par \ref{Eq1}
\begin{equation*}
\mathscr{E}=\bigoplus_{\tau\in\Sigma^{\mathrm{nr}}}\mathscr{E}_{\tau},\ \ \ \ \ \omega_G=\bigoplus_{\tau\in\Sigma^{\mathrm{nr}}}\omega_{G,\tau}
\end{equation*}
où $\Sigma^{\mathrm{nr}}=\mathrm{Hom}(L^{\mathrm{nr}},\overline{\Q}_p)$. Autrement dit pour tout $\tau\in\Sigma^{\mathrm{nr}}$ on dispose d'une suite exacte 
\begin{equation*}
0\longrightarrow \omega_{G,\tau}\longrightarrow\mathscr{E}_{\tau}\longrightarrow\omega_{G^D,\tau}^{\vee}\longrightarrow 0
\end{equation*}
induite par la filtration de Hodge.\\

\begin{definition}
\label{DefPRData}
Une donnée de Pappas-Rapoport pour $(G,\iota)$ par rapport à $(\Sigma_{\tau})_{\tau\in\Sigma^{\mathrm{nr}}}$ est la donnée pour tout $\tau\in\Sigma^{\mathrm{nr}}$ d'une filtration de Pappas-Rapoport pour $(\mathscr{E}_{\tau},\omega_{G,\tau})$ par rapport à $\Sigma_{\tau}$. Autrement dit c'est la donnée pour tout $\tau\in\Sigma^{\mathrm{nr}}$ d'une filtration

\begin{equation*}
0=\omega_{G,\tau}^{(0)}\subset \omega_{G,\tau}^{(1)}\subset\cdots\subset \omega_{G,\tau}^{(e)}=\omega_{G,\tau}
\end{equation*}
telle que pour tout $1\leqslant j\leqslant e$
\begin{itemize}[label=\textbullet]
\item Les $\omega_{G,\tau}^{(j)}$ sont localement des $\mathcal{O}_S$-facteurs directs, stables par $\mathcal{O}_L$.
\item $([\varpi]-\varphi_{\tau,j}(\varpi))\cdot \omega_{G,\tau}^{(j)}\subset \omega_{G,\tau}^{(j-1)}$
\item $\omega_{G,\tau}^{(j)}/\omega_{G,\tau}^{(j-1)}$ est localement libre de rang $1$
\end{itemize}
\end{definition}
\ \\

\subsubsection{Cas général}

Revenons à notre extension $F/\Q$. Soit donc $G$ un groupe $p$-divisible sur $S$ muni d'une action $\iota:\mathcal{O}_F\rightarrow\mathrm{End}_S(G)$. Pour chaque place $v|p$ on note $\Sigma_v^{\mathrm{nr}}=\mathrm{Hom}(F_v^{\mathrm{nr}},\overline{\Q}_p)$,  pour tout $\tau\in\Sigma_v^{\mathrm{nr}}$ on note $\Sigma_{v,\tau}=\mathrm{Hom}_{F_v^{\mathrm{nr}}}(F_v,\overline{\Q}_p)$. On a alors une première décomposition 
\begin{equation*}
\mathscr{E}=\bigoplus_{v|p}\mathscr{E}_{v},\ \ \ \ \ \omega_G=\bigoplus_{v}\omega_{G,v}
\end{equation*}
et une deuxième décomposition pour chacune des places $v|p$ :
\begin{equation}
\label{decompow}
\mathscr{E}_v=\bigoplus_{\tau\in\Sigma_v^{\mathrm{nr}}}\mathscr{E}_{v,\tau},\ \ \ \ \ \omega_{G,v}=\bigoplus_{\tau\in\Sigma_v^{\mathrm{nr}}}\omega_{G,v,\tau}
\end{equation}
Comme précédemment on fixe un ordre sur chacun des $\Sigma_{v,\tau}$.
\begin{definition}
Une donnée de Pappas-Rapoport pour $(G,\iota)$ par rapport à $(\Sigma_{v,\tau})_{v,\tau}$ est la donnée pour toute place $v|p$ et tout  $\tau\in\Sigma_v^{\mathrm{nr}}$ d'une filtration de Pappas-Rapoport pour $(\mathcal{E}_{v,\tau},\omega_{G,v,\tau})$ par rapport à $\Sigma_{v,\tau}$. Autrement dit c'est la donnée pour toute place $v|p$ et tout $\tau\in\Sigma_v^{\mathrm{nr}}$ d'une filtration

\begin{equation*}
0=\omega_{G,v,\tau}^{(0)}\subset \omega_{G,v,\tau}^{(1)}\subset\cdots\subset \omega_{G,v,\tau}^{(e)}=\omega_{G,v,\tau}
\end{equation*}
telle que pour tout $1\leqslant j\leqslant e$
\begin{itemize}[label=\textbullet]
\item Les $\omega_{G,v,\tau}^{(j)}$ sont localement des $\mathcal{O}_S$-facteurs directs, stables par $\mathcal{O}_F$.
\item $([\varpi]-\varphi_{v,\tau,j}(\varpi))\cdot \omega_{G,v,\tau}^{(j)}\subset \omega_{G,v,\tau}^{(j-1)}$
\item $\omega_{G,v,\tau}^{(j)}/\omega_{G,v,\tau}^{(j-1)}$ est localement libre de rang $1$
\end{itemize}
\end{definition}
\begin{remark}
Dans la définition ci dessus nous faisons l'abus d'appeler \og donnée de Pappas Rapoport\fg ce qui correspond en fait à une \og donnée de Pappas Rapoport dans le cas Hilbert\fg \ c'est-à-dire une donnée de Pappas-Rapoport où la dimension des gradués $\omega_{G,v,\tau}^{(j)}/\omega_{G,v,\tau}^{(j-1)}$ est égale à $d_j=1$ pour tout $1\leq j\leq e$. On consultera \cite{BijHer2} et \cite{BijHer1} pour plus de détails sur les filtrations de Pappas-Rapoport dans un cadre PEL plus général.
\end{remark}

\subsubsection{Définition explicite de $\mathcal{S}h^{\mathrm{PR}}$}
\label{SPR}
Revenons à notre modèle entier $\mathcal{S}h^{\mathrm{PR}}$. Il découle de la définition \eqref{CarreCart2} que le foncteur $\mathcal{S}h^{\mathrm{PR}}$ associe à un schéma localement noethérien $S\rightarrow \spec\, \mathcal{O}_K$, les quintuplés $(A,\lambda,\iota,\eta,\omega_G^{(\cdot)})$ à $\Z_{(p)}^{\times}$-isogénie près où 
\begin{itemize}[label=\textbullet]
\item $A\rightarrow S$ est un schéma abélien de dimension $d$.
\item $\lambda:A\rightarrow A^t$ est une $\Z_{(p)}^{\times}$-polarisation
\item $\iota: \mathcal{O}_F\rightarrow\mathrm{End}(A)\otimes_{\Z} \Z_{(p)}$
\item $\eta$ est une structure de niveau rationnelle en dehors de $p$ (voir \cite{KWLan} section 1.4.1)
\item $\omega_G^{(\cdot)}$ est une donnée de Pappas-Rapoport pour $G=A[p^{\infty}]$.
\end{itemize}

\section{Stratification de Hodge}

Désormais tout ce qui suit porte sur la fibre spéciale de nos modèles entiers. Nous noterons donc désormais pour alléger les notations $\Naiff=\Naif$ et $\Split=\Splitt$. Pour simplifier nous allons commencer par traiter le cas d'une seule extension totalement ramifiée $L/\Q_p$ de degré $e$ (jouant le rôle de $F_v/F_v^{\mathrm{nr}}$. Pour tout $1\leq i\leq e$ on note $\mu_i=(1,0)$ et $\mu_{\bullet}=(\mu_1,\dots,\mu_e)$. On a donc  $\mu:=|\mu_{\bullet}|=(e,0)$. 
\subsection{Le cas PEL}
Dans la section \ref{SPEL} nous avons vu que le diagramme de modèle local nous fournissait dans le cas PEL un morphisme lisse \eqref{StoMPEL}

\begin{equation*}
\mathcal{S}h^{\mathrm{PEL}}\longrightarrow  \big[\Naif/\mathcal{G}\big]
\end{equation*}

D'après la proposition \ref{PlongPEL} on dispose d'un plongement équivariant du modèle local dans la Grassmannienne affine induisant après passage au quotient un morphisme 
\begin{equation*}
\big[\Naif/\mathcal{G}\big]\longrightarrow \big[L^+G\backslash LG/L^+G\big]=\mathrm{Hecke}
\end{equation*}

La description de l'image de ce morphisme \ref{PlongPEL} nous fournit donc finalement un morphisme lisse vers le champ de Hecke borné :

\begin{equation*}
\zeta:\mathcal{S}h^{\mathrm{PEL}}\longrightarrow  \mathrm{Hecke}_{\mathrm{Adm}(\mu)_K}
\end{equation*}
où $\mathrm{Adm}(\mu)_K$ est définit en \ref{Admissible}.

\subsubsection{Définition}
\label{StrataHodge}

Soit $x\in |\NNaif|$. On appelle polygone de Hodge du point $x$ l'élément 
\begin{equation*}
\mathrm{Hodge}(x)=\zeta(x)
\end{equation*}
Pour tout $\lambda\in\mathrm{Adm}(\mu)_K$  on définit le sous schéma localement fermé (structure réduite) $\NNaif_{\lambda}=\zeta^{-1}(\lambda)$. Les $(\NNaif_{\lambda})_{\lambda\in\mathrm{Adm}(\mu)_K}$ sont appelées les strates de Hodge de $\NNaif$.

\begin{remark}
\label{Hodge=KR}
Donnons une définition plus explicite du polygone de Hodge et des strates de Hodge. Soit $x\in|\NNaif|$ un point de corps résiduel $k$. Pour calculer l'image $\zeta(x)$ il nous faut suivre $x$ dans le diagramme de modèle local :
\[\begin{tikzcd}
	& {\mathcal{S}h^{\mathrm{PEL},\square}} \\
	{\mathcal{S}h^{\mathrm{PEL}}} && {M^{\mathrm{PEL}}} & {\mathrm{Gr}}
	\arrow["\varphi"', from=1-2, to=2-1]
	\arrow["\psi", from=1-2, to=2-3]
	\arrow["\iota", from=2-3, to=2-4]
\end{tikzcd}\]

Notons $\omega_x\hookrightarrow \mathscr{E}_x$ la filtration de Hodge associée au point $x$. On choisit un élément $\tilde{x}\in\varphi^{-1}(x)$. L'image $\iota\circ \psi(\tilde{x})$ dépend du choix $\tilde{x}$, mais son image dans le quotient $\mathrm{Hecke}=\big[L^+G\backslash \Gr\big]$ ne dépend pas de ce choix. Pour décrire $\zeta(x)$ il nous suffit donc de décrire $\iota\circ \psi(\tilde{x})$. Par définition de $\mathcal{S}h^{\mathrm{PEL},\square}$ sur $\tilde{x}$ on dispose d'une trivialisation $\mathscr{E}_x\simeq \Lambda$. Autrement dit on dispose d'une base :
\begin{equation*}
\mathscr{E}_x\simeq \frac{k\kt}{(u^{e})}\oplus \frac{k\kt}{(u^{e})}
\end{equation*}
Quitte à choisir une autre base on peut supposer que la filtration de Hodge soit donnée par
\begin{equation*}
\omega_x\simeq u^i\frac{k\kt}{(u^{e})}\oplus u^j\frac{k\kt}{(u^{e})}
\end{equation*}
Ensuite d'après \eqref{LambdaF} le plongement du modèle local dans la grassmannienne  affine consiste  à prendre l'image inverse de la filtration de Hodge le long de $k\kt\rightarrow k\kt/(u^e)$. On obtient finalement
\begin{equation*}
\Lambda_{\omega_x}=u^ik\kt\oplus u^jk\kt\subset\Lambda_0
\end{equation*}
Par conséquent on retrouve la définition usuelle (voir \cite{DelignePappas} Section 4.2, où \cite{BijHer1} Définition 1.1.7) du polygone de Hodge d'un groupe $p$-divisible
\begin{equation*}
\mathrm{Hodge}(x)=\mathrm{Hodge}(\omega_x)=(i,j)
\end{equation*}
\end{remark}

\begin{remark}
En fait le polygone de Hodge correspond plutôt  au polygone  de pentes $(\frac{i}{e},\frac{j}{e})$ (voir \cite{BijHer1}, Définition 1.1.7). Bien sûr on retrouve une définition à partir de l'autre et les stratifications induites coïncident.
\end{remark}

\subsubsection{Propriétés}\ 
\label{PropStratPEL}
La proposition suivante résume les différentes propriétés de la stratification de Hodge :
\begin{prop}[\cite{DelignePappas} section 4.2, \cite{Alcove} section 4]
\ 
\begin{enumerate}
\item Les strates $(\NNaif_{\lambda})_{\lambda\in\mathrm{Adm}(\mu)_K}$ forment une bonne stratification de $\NNaif$.
\item Pour tout $\lambda\in\mathrm{Adm}(\mu)_K$ la strate $\NNaif_{\lambda}$ est quasi-projective lisse de dimension $\langle 2\rho,\lambda\rangle$.
\item La strate $\NNaif_{(e,0)}$ coïncide avec le lieu lisse de $\NNaif$.
\end{enumerate}
\end{prop}

\begin{proof}
Pour le point $(1)$, cela découle de la lissité du morphisme $\zeta:\NNaif\rightarrow\mathrm{Hecke}_{\mathrm{Adm}(\mu)_K}$. 
Pour le point (3), cela découle du fait que $\mathrm{Gr}_{(e,0)}$ coïncide avec le lieu lisse de $\mathrm{Gr}_{\leq(e,0)}$ (voir Proposition \ref{PropOrbiteGR})
Enfin pour le point $(2)$, il suffit de voir que le morphisme 
\begin{equation*}
\NNaif_{\lambda}\rightarrow \big[\Naif_{\lambda}/\mathcal{G}\big]
\end{equation*}
est lisse de dimension relative $\mathrm{dim}\, \mathcal{G}$ et que le membre de droite était un point de dimension $\langle 2\rho,\lambda\rangle -\mathrm{dim}\, \mathcal{G}$.
\end{proof}
\begin{remark}
Pour le point $(2)$ on peut se passer du langage des champs et donner une preuve à la main : en restriction à une strate $\lambda\in\mathrm{Adm}(\mu)_K$ le diagramme de modèle local devient 
\[\begin{tikzcd}
	& {\mathcal{S}h^{\mathrm{PEL},\square}_{\lambda}} \\
	{\mathcal{S}h^{\mathrm{PEL}}_{\lambda}} && {M^{\mathrm{PEL}}_{\lambda}}
	\arrow["\varphi"', from=1-2, to=2-1]
	\arrow["\psi", from=1-2, to=2-3]
\end{tikzcd}\]
D'après la proposition \ref{PropOrbiteGR}  la strate $M^{\mathrm{PEL}}_{\lambda}$ est de dimension $\langle 2\rho,\lambda\rangle$. Le morphisme $\psi$ est de dimension relative $\mathrm{dim}\, G=\mathrm{dim}\, \mathcal{G}$ et le morphisme $\varphi$ est un $\mathcal{G}$-torseur. On trouve donc bien 
\begin{equation*}
\mathrm{dim}\, \mathcal{S}h^{\mathrm{PEL}}_{\lambda}=\langle 2\rho,\lambda\rangle +\mathrm{dim}\, \mathcal{G}-\mathrm{dim}\, \mathcal{G}=\langle 2\rho,\lambda\rangle
\end{equation*}
\end{remark}

\begin{remark}
Dans \cite{GorenAndreatta} la strate $\NNaif_{(e,0)}$ est appelée lieu de Rapoport car elle coïncide avec le lieu où le faisceau $\omega$ est libre en tant que $\OK_L\otimes_{\Z_p}\OK_S$-module.
\end{remark}

\begin{remark}
On peut être plus explicite sur la dimension. Si $\lambda=(i,j)\leq (e,0)$ alors 
\begin{align*}
\langle 2\rho,\lambda \rangle &=\langle \alpha_1-\alpha_2,i\lambda_1+j\lambda_2\rangle\\
&=i-j\\
 &=e-2j
\end{align*}
où $\alpha_1:\begin{pmatrix}
t_1 &0\\
0&t_2
\end{pmatrix}\mapsto t_1$, $\alpha_2:\begin{pmatrix}
t_1 &0\\
0&t_2
\end{pmatrix}\mapsto t_2$, $\lambda_1:t\mapsto \begin{pmatrix}
t&0\\
0&1
\end{pmatrix}$ et $\lambda_2:t\mapsto \begin{pmatrix}
1&0\\
0&t
\end{pmatrix}$.
\end{remark}

\subsection{Le cas Pappas-Rapoport}

Nous allons maintenant nous intéresser à la stratification du modèle de Pappas-Rapoport $\SSplit$ par le polygone  $\mathrm{Hodge}(\omega)$.

\subsubsection{Définition}
Pour tout $\lambda\in\mathrm{Adm}(\mu)_K$  on définit le sous schéma localement fermé (structure réduite) $\SSplit_{\lambda}=\pi^{-1}(\NNaif_{\lambda})$ où
\begin{equation*}
\pi:\SSplit\rightarrow \NNaif
\end{equation*}
est le morphisme d'oubli. Les $(\SSplit_{\lambda})_{\lambda\in\mathrm{Adm}(\mu)_K}$ sont appelées les strates de Hdoge de $\SSplit$. Plus explicitement :  si $x\in\SSplit$ alors on peut regarder le polygone de hodge $\mathrm{Hodge}(x):=\mathrm{Hodge}(\omega^{(e)})$ où $\omega^{(e)}=\omega_x$ est le dernier cran de la filtration de Pappas-Rapoport $\omega^{(1)}\subset\dots\subset\omega^{(e)}$ associée au point $x$.
\begin{remark}
On fera attention à ne pas confondre les deux morphismes :
\[\begin{tikzcd}
	& {\mathcal{S}h^{\mathrm{PR}}} \\
	{\big[L^+G\backslash\widetilde{\mathrm{Gr}}_{\mu_{\bullet}}\big]=\mathrm{Hecke}_{\mu_{\bullet}}} && {\mathrm{Hecke}_{\mathrm{Adm}(\mu)_K}=\big[L^+G\backslash\mathrm{Gr}_{\leq|\mu_{\bullet}|}\big]}
	\arrow["{\tilde{\zeta}}"', from=1-2, to=2-1]
	\arrow["\zeta\circ\pi", from=1-2, to=2-3]
\end{tikzcd}\]

Le morphisme de gauche concerne les classes d'isomorphismes de filtrations de Pappas-Rapoport $(\omega^{(1)}\subset \dots\subset \omega^{(e)}\subset \mathscr{E})$, et celui de droite concerne les classes d'isomorphismes de filtrations de Hodge $(\omega\subset \mathscr{E})$.

\end{remark}

\subsubsection{Propriétés}

Le théorème est le suivant :

\begin{theorem}\ 
\label{TheoremPrincipal}
\begin{enumerate}
\item Les strates $(\SSplit_{\lambda})_{\lambda\in\mathrm{Adm}(\mu)_K}$ forment une bonne stratification de $\SSplit$. Autrement dit pour tout $\lambda\in\mathrm{Adm}(\mu)_K$ on a la relation d'adhérence
\begin{equation*}
\overline{\SSplit_{\lambda}}=\bigcup_{\lambda'\leq\lambda}\SSplit_{\lambda'}
\end{equation*}
\item Pour tout $\lambda\in\mathrm{Adm}(\mu)_K$ la strate $\SSplit_{\lambda}$ est quasi-projective lisse de dimension $\langle \rho,|\mu_{\bullet}|+\lambda\rangle$.
\end{enumerate}
\end{theorem}

\begin{proof}
Commençons par montrer le point $(1)$. D'après la proposition \ref{StratGrConv} on dispose d'une bonne stratification du produit de convolution
\begin{equation*}
\widetilde{\Gr}_{\mu_{\bullet}}=\bigcup_{\lambda\in\mathrm{Adm}(\mu)_K} m^{-1}(\Gr_{\lambda})
\end{equation*}
Puisque le morphisme de convolution $m:\widetilde{\Gr}_{\mu_{\bullet}}\rightarrow\Gr_{\leq|\mu_{\bullet}|}$ est $L^+G$-équivariant, cette stratification est stable sous l'action de $L^+G$. Autrement dit chacune des strates est une union de $L^+G$-orbites. Par conséquent cette stratification descend au quotient en une bonne stratification du champs $\mathrm{Hecke}_{\mu_{\bullet}}=\big[L^+G\backslash \Gr_{\mu_{\bullet}}\big]$ :
\begin{equation*}
\mathrm{Hecke}_{\mu_{\bullet}}=\bigcup_{\lambda\in\mathrm{Adm}(\mu)_K}\big[L^+G\backslash m^{-1}(\Gr_{\lambda})\big]
\end{equation*}
On conclut comme pour le cas PEL (voir \ref{PropStratPEL}) en utilisant que le morphisme
\begin{equation*}
\tilde{\zeta}:\SSplit\rightarrow\mathrm{Hecke}_{\mu_{\bullet}}
\end{equation*}
est lisse et donc préserve les relations d'adhérences. \\

 Montrons maintenant le point $(2)$. D'après le théorème \ref{TheoHaines}, en restriction à une strate $\Gr_{\lambda}$ associée à $\lambda\in\mathrm{Adm}(\mu)_K$ on dispose d'une fibration localement triviale 
\begin{equation*}
m:m^{-1}(\Gr_{\lambda})\longrightarrow \Gr_{\lambda}
\end{equation*}
En particulier au dessus d'une telle strate le morphisme est plat et par conséquent on a la relation
\begin{equation*}
\mathrm{\dim}\, m^{-1}(\Gr_{\lambda})=\mathrm{dim}\, \Gr_{\lambda}+\mathrm{dim}\, m^{-1}(y)
\end{equation*}
pour n'importe quel élément $y\in\Gr_{\lambda}$. Toujours d'après le théorème \ref{TheoHaines}, on a que la fibre  $m^{-1}(y)$ est équidimensionnelle de dimension $\langle \rho,|\mu_{\bullet}|-\lambda\rangle$ car $\mu_{\bullet}=(\mu_1,\dots\mu_e)$ est constitué de cocaractères minuscules. On définit au niveau du modèle local $\Split_{\lambda}=m^{-1}(\Gr_{\lambda})$ via l'identification $\Split\simeq\widetilde{\Gr}_{\mu_\bullet}$. Le modèle local nous donne un morphisme 
\begin{equation*}
\SSplit_{\lambda}\rightarrow \big[\Split_{\lambda}/\mathcal{G}\big]
\end{equation*}
qui est lisse de dimension relative $\mathrm{dim}\, \mathcal{G}$. Le membre de droite est de dimension $\Split_{\lambda} -\mathrm{dim}\, \mathcal{G}$. Finalement on trouve donc bien 
\begin{align*}
\mathrm{dim}\, \SSplit_{\lambda}&=\mathrm{dim}\, \big[\Split_{\lambda}/\mathcal{G}\big]+\mathrm{dim}\, \mathcal{G}\\
&=\mathrm{dim}\, \Split_{\lambda}+\mathrm{dim}\, \mathcal{G}-\mathrm{dim}\, \mathcal{G}\\
&=\mathrm{dim}\, m^{-1}(\Gr_{\lambda})\\
&=\mathrm{dim}\, \Gr_{\lambda}+\mathrm{dim}\, m^{-1}(y)\\
&=\langle 2\rho,\lambda\rangle+\langle \rho,|\mu_{\bullet}|-\lambda\rangle\\
&=\langle \rho,|\mu_{\bullet}|+\lambda\rangle
\end{align*}
\end{proof}

\begin{remark}
Encore une fois on peut être plus explicite concernant la dimension des strates. Pour $\lambda=(i,j)\leq (e,0)$ on obtient
\begin{align*}
\mathrm{dim}\, \SSplit_{(i,j)}&=\langle \rho,|\mu_{\bullet}|+\lambda\rangle\\
&=\frac{1}{2}\langle \alpha_1-\alpha_2,(e+i)\lambda_1+j\lambda_2\rangle\\
&=e+i-j\\
 &=e-j
\end{align*}
\end{remark}

\subsection{Cas général}

On revient maintenant au cas d'une extension $F/\Q$ de degré $d>1$. Pour tout $v|p$, $\tau\in\Sigma_{v}^{\mathrm{nr}}$ et tout $1\leq\dots i\leq e_v$ on note $\mu_{v,\tau,i}=(1,0)$, $\mu_{v,\tau,\bullet}=(\mu_{v,\tau,1},\dots,\mu_{v,\tau,e_v})$, $\mu_{v,\tau}=|\mu_{v,\tau,\bullet}|=(e_v,0)$ et enfin $\mu:=(\mu_{v,\tau})_{v,\tau}$ La compatibilité du modèle local au produit décrite en \eqref{CarreCart3} suggère la définition suivante 
 : soit $x\in|\SSplit|$, on définit l'invariant :
 \begin{equation*}
 \mathrm{Hodge}(x)=\big(\mathrm{Hodge}_{v,\tau}(x)\big)_{v,\tau}\in \prod_{v,\tau}\{\lambda\leq (e_v,0)\}=\mathrm{Adm}(\mu)_K
 \end{equation*}
où $\mathrm{Hodge}_{v,\tau}(x)=\mathrm{Hodge}(\omega_{x,v,\tau})$ est l'invariant de Hodge définit en \ref{StrataHodge} et $\omega_{x,v,\tau}$ est le faisceau associé à $x$ et défini en \eqref{decompow}. Autrement dit l'invariant $\mathrm{Hodge}(\omega)$ ci dessus correspond à la donnée des invariants $\mathrm{Hodge}(\omega_{v,\tau})$ pour chacun des termes dans la décomposition
\begin{equation*}
\omega=\bigoplus_{v|p}\bigoplus_{\tau\in\Sigma_v^{\mathrm{nr}}} \omega_{v,\tau}
\end{equation*}

On munit $\mathrm{Adm}(\mu)_K$ de la relation d'ordre suivante :
\begin{equation*}
(\lambda_{v,\tau})_{v,\tau}\leq (\lambda_{v,\tau}')_{v,\tau}\ \Longleftrightarrow\ \lambda_{v,\tau}\leq\lambda_{v,\tau}'\ \forall (v,\tau)
\end{equation*}

La généralisation du Théorème \ref{TheoremPrincipal} prend la forme suivante :

\begin{theorem}\ 
\begin{enumerate}
\item Les strates $(\SSplit_{\lambda_{v,\tau}})_{\lambda_{v,\tau}\in\mathrm{Adm}(\mu)_K}$ forment une bonne stratification de $\SSplit$. Autrement dit pour tout $(\lambda_{v,\tau})_{v,\tau}\in\mathrm{Adm}(\mu)_K$ on a la relation d'adhérence
\begin{equation*}
\overline{\SSplit_{(\lambda_{v,\tau})_{v,\tau}}}=\bigcup_{(\lambda_{v,\tau}')_{v,\tau}\leq (\lambda_{v,\tau})_{v,\tau}}\SSplit_{(\lambda_{v,\tau}')_{v,\tau}}
\end{equation*}
\item Pour tout $(\lambda_{v,\tau})_{v,\tau}\in\mathrm{Adm}(\mu)_K$ la strate $\SSplit_{\lambda_{v,\tau}}$ est quasi-projective lisse de dimension :
\begin{equation*}
\mathrm{dim}\, \SSplit_{(\lambda_{v,\tau})_{v,\tau}}=\sum_{v,\tau}\langle \rho,|\mu_{v,\tau,\bullet}|+\lambda_{v,\tau}\rangle
\end{equation*}
\end{enumerate}
\end{theorem}

\begin{remark}
On fera attention au fait que dans le cadre d'une extension $L/\Q_p$  d'indice d'inertie $f>1$, alors la stratification de Hodge ci-dessus ne coïncide pas avec la stratification de Hodge définie dans \cite{BijHer2}. En fait la stratification définie dans \textit{loc.cit.} est moins fine car elle est définie par le polygone de Hodge de \cite{BijHer1} qui est un moyenne sur $f$ des invariants de Hodge considérés pour notre stratification.  On consultera \cite{XuShen} Proposition 3.20 pour plus de détails.
\end{remark} 

\subsection{Autres résultats}

Lorsque la donnée PEL est ramifiée,  l'un des objectifs pour étudier la géométrie de la fibre spéciale de la variété de Shimura associée est de raffiner la stratification de Hodge.  Autrement dit pour  tout $\lambda\in\mathrm{Adm}(\mu)_K$ on aimerait définir une décomposition :
\begin{equation*}
\NNaif_{\lambda}=\coprod_{w\in W_{\Lambda}} S_w
\end{equation*} 
où $(W_{\lambda},\leq)$ est un certain ensemble partiellement ordonné qui dépend de $\lambda$.  Donnons quelques exemples :
\begin{enumerate}
\item Dans \cite{EKOR} les auteurs définissent pour chaque strate de Hodge un morphisme lisse :
\begin{equation*}
\zeta_{\lambda}:\NNaif_{\lambda}\longrightarrow \mathcal{G}_0^{\mathrm{rdt}}\text{-}\mathtt{Zip}^{J_{\lambda}}
\end{equation*}
où $\mathcal{G}_0^{\mathrm{rdt}}$ désigne le quotient réductif de $\mathcal{G}_0=\mathcal{G}\otimes_{\Z_p}\mathbb{F}_p$ et $\mathcal{G}_0^{\mathrm{rdt}}\text{-}\mathtt{Zip}^{\lambda}$ désigne le champs des $\mathcal{G}_0^{\mathrm{rdt}}\text{-}\mathtt{Zip}$ de type $J_{\lambda}$ (voir \cite{EKOR} Définition 1.1.5 et Proposition 4.2.6). La stratification est alors définie comme étant celle induite par celle du champs $\mathcal{G}_0^{\mathrm{rdt}}\text{-}\mathtt{Zip}^{J_{\lambda}}$ via le morphisme $\zeta_{\lambda}$. Leur construction est beaucoup plus générale et s'applique aux variétés de Shimura de type abélien sans hypothèse sur le niveau en $p$
\item Dans \cite{GorenAndreatta} les auteurs calculs explicitement les polygones de Newton sur chacune des strates. Plus précisément si $x\in\NNaif_{(i,j)}$ avec $(i,j)\in\mathrm{Adm}(\mu)_K$ alors existe une $(\OK_L\otimes W(k))$- base du module de Dieudonné associé dans laquelle le Frobenius est donné par la matrice (\cite{GorenAndreatta}, Propositon 4.10) :
\begin{equation*}
F=\begin{pmatrix}
\varpi^m&c\varpi^i\\
\varpi^j&0
\end{pmatrix}
\end{equation*}
où $m\geq j$ et $c\in(\OK_L\otimes W(k))^{\times}$. Ils définissent ensuite les quantités :
\begin{equation*}
n=\left\{\begin{array}{lrcl}
m&\text{si}\ m\leq i\\
i&\text{sinon}
\end{array}\right.\ \ \ \ \ \ \ \ \ \lambda(n)=\mathrm{min}\left\{\frac{n}{g},\frac{1}{2}\right\}
\end{equation*}
La stratification de $\NNaif_{(i,j)}$ est alors définie via cet invariant $n$ :
\begin{equation*}
\NNaif_{(i,j)}=\bigcup_{j\leq n\leq i} \NNaif_{(i,j),n}
\end{equation*}
Cet invariant leur permet de calculer explicitement le polygone de Newton sur chacune des strates $\NNaif_{(i,j),n}$. Plus précisément si $x\in \NNaif_{(i,j),n}$ alors (\cite{GorenAndreatta}, Théorème 9.2) :
\begin{equation*}
\mathrm{Newt}(x)=\{\lambda(n),\dots,\lambda(n),1-\lambda(n),\dots,1-\lambda(n)\}
\end{equation*}
(chacune des pentes avec multiplicité $e$).
\item Dans \cite{FuLiOort} les auteurs associent à tout point $x\in \NNaif$ un invariant $c(x)=(c_{\tau}(x))_{\tau\in\Sigma^{\mathrm{nr}}}$ appelé invariant de congruence qui mesure la position relative des filtrations (voir la \cite{FuLiOort} Définition 5.2 pour être plus précis) : 
\begin{equation*}
F(M_{\sigma^{-1}\tau})\subset M_{\tau}, \ \ \ \ V(M_{\sigma\tau})\subset M_{\tau}
\end{equation*}
où $(M,F,V)$ désigne le module de Dieudonné associé à $x$ et où les $(M_{\tau})_{\tau\in\Sigma^{\mathrm{nr}}}$ désignent les facteurs directs dans la décomposition 
\begin{equation*}
M\simeq \bigoplus_{\tau\in\Sigma^{\mathrm{nr}}} M_{\tau},\ \ \ F:M_{\sigma^{-1}\tau}\rightarrow M_{\tau}, \ \ \ V:M_{\sigma\tau}\rightarrow M_{\tau}
\end{equation*}
Cet invariant permet de définir une stratification (\cite{FuLiOort} Définition 6.1) :
\begin{equation*}
\NNaif_{\lambda}=\coprod_{c\in \tau_L} \mathcal{Q}_c(\NNaif_{\lambda})
\end{equation*}
où $\mathcal{Q}_c(\NNaif_{\lambda})$ désigne l'ensemble des points d'invariant de congruence $c\in\tau_L$ et $\tau_L$ est un certain ensemble défini dans \cite{FuLiOort} Définition 3.2.1. 
\end{enumerate}
Dans chacun de ces trois articles se pose deux problèmes :
\begin{enumerate}
\item Calculer l'adhérence de $S_w$ dans $\NNaif_{\lambda}$.
\item Calculer l'adhérence de $S_w$ dans $\NNaif$
\end{enumerate}
Dans \cite{FuLiOort} les auteurs parviennent à calculer le point $(1)$. Dans \cite{GorenAndreatta} et \cite{EKOR} les auteurs parviennent à calculer le point $(2)$. Le théorème suivant montre que le point $(1)$ est automatique pour $\SSplit_{\lambda}$ lorsque l'on regarde le tiré en arrière d'une stratification de $\NNaif_{\lambda}$ via le morphisme d'oubli $\pi:\SSplit\rightarrow\NNaif$. 

\begin{theorem}
Pour tout $\lambda\in\mathrm{Adm}(\mu)_K$ la restriction
\begin{equation*}
\pi:\SSplit_{\lambda}\rightarrow\NNaif_{\lambda}
\end{equation*}
est un morphisme plat. 
\end{theorem}
\begin{proof}
En restriction à une strate d'invariant $\lambda\in\mathrm{Adm}(\mu)_K$ le carré cartésien \eqref{CarreCart2} devient
\[\begin{tikzcd}
	{\mathcal{S}h^{\mathrm{PR}}_{\lambda}} & {\big[M^{\mathrm{PR}}_{\lambda}/\mathcal{G}\big]} \\
	{\mathcal{S}h^{\mathrm{PEL}}_{\lambda}} & {\big[M^{\mathrm{PEL}}_{\lambda}/\mathcal{G}\big]}
	\arrow["\overline{p}", from=1-2, to=2-2]
	\arrow[from=1-1, to=1-2]
	\arrow["\pi"', from=1-1, to=2-1]
	\arrow[from=2-1, to=2-2]
\end{tikzcd}\]
qui est lui aussi cartésien. Il suffit de montrer que le morphisme $\overline{p}$ est plat. Ce dernier s'inscrit dans un carré cartésien 
\[\begin{tikzcd}
	{M^{\mathrm{PR}}_{\lambda}} & {M^{\mathrm{PEL}}_{\lambda}} \\
	{\big[M^{\mathrm{PR}}_{\lambda}/\mathcal{G}\big]} & {\big[M^{\mathrm{PEL}}_{\lambda}/\mathcal{G}\big]}
	\arrow["{\overline{p}}"', from=2-1, to=2-2]
	\arrow[from=1-1, to=2-1]
	\arrow["q", from=1-2, to=2-2]
	\arrow["p", from=1-1, to=1-2]
\end{tikzcd}\]
Or d'après le théorème $\ref{TheoHaines}$ et $\eqref{CarreCart}$ le morphisme $p$ est localement trivial et donc en particulier plat (le schéma de base $\spec\, k$  est un corps). Le morphisme $q$ étant fidèlement plat, on conclut par \cite{stacks-project}, Lemme 100.25.4.
\end{proof}

\begin{coro}
Si $\NNaif_{\lambda}=\coprod_{w\in W} S_w$ est une bonne stratification alors 
\begin{equation*}
\SSplit_{\lambda}=\coprod_{w\in W} \pi^{-1}(S_w)
\end{equation*}
est également une bonne stratification.
\end{coro}
\begin{proof}
D'après la proposition précédente le morphisme $\pi$ est plat en restriction à une strate $\NNaif_{\lambda}$ et est donc en particulier ouvert. 
\end{proof}

\section{Invariants de Hasse partiels}
\label{SectionHasse}

Nous allons maintenant décrire l’interaction entre la stratification de Hodge définie dans la section précédente, et les invariants de Hasse partiels définis dans \cite{ReduzziXiao}, lorsque l'indice de ramification satisfait $e\leq 4$. Sans perdre de généralité on peut se restreindre à une place $v|p$ et donc considérer la situation d'un groupe $p$-divisible $G$ muni d'une action d'un anneau d'entier $\OK_L$ où $L/\Q_p$ est une extension de degré $d$.  Pour faciliter la lecture, nous allons commencer par un rappel des différentes définitions de \textit{loc  cit}.

\subsection{Définitions}
\subsubsection{Invariant de Hasse}
Soit $G$ un groupe de Barsotti-Tate sur un schéma  $S$ de caractéristique $p$.  Le Verschiebung $\mathrm{Ver}:G^{(p)}\rightarrow G$ induit un morphisme :
\begin{equation*}
\mathrm{Ver}^*:\omega_G\longrightarrow\omega_G^{(p)}
\end{equation*}
En prenant le déterminant de ce morphisme on obtient une section
\begin{equation}
\mathrm{Ha}(G)\in H^0(S,(\mathrm{det}\, \omega_G)^{\otimes (p-1)})
\end{equation}
appelée invariant de Hasse. On a alors la proposition suivante :
\begin{prop}
Soit $G$ un groupe $p$-divisible sur un corps $k$ de caractéristique $p>0$. Alors $\mathrm{Ha}(G)$ est inversible si et seulement si $G$ est ordinaire.
\end{prop}

\subsubsection{Invariants de Hasse partiels}
\label{SectionHasse}
Soit $(G,\iota)$ un groupe $p$-divisible sur $S$ muni d'une action $\iota:\mathcal{O}_L\rightarrow \mathrm{End}_S(G)$. On suppose que $(G,\iota)$ est muni d'une donnée de Pappas-Rapoport $(\mathscr{E}_{\tau},\omega_{G,\tau})_{\tau}$ une donnée de Pappas-Rapoport (voir \ref{DefPRData}).
Pour tout $2\leq i\leq e$ on définit l'application $M_{\tau}^{(i)}$ comme étant la multiplication par $\varpi$ au niveau des gradués :
\begin{equation*}
M_{\tau}^{(i)}:\omega_{G,\tau}^{(i)}/\omega_{G,\tau}^{(i-1)}\longrightarrow \omega_{G,\tau}^{(i)}/\omega_{G,\tau}^{(i-2)}
\end{equation*}
Ce morphisme induit une section appelée \textit{invariant de Hasse primitif} :
\begin{equation*}
m_{\tau}^{(i)}(G)\in H^0(S,\mathrm{det}\, (\omega_{G,\tau}^{[i-1]}/\omega_{G,\tau}^{[i-2]})\otimes \mathrm{det}\, (\omega_{G,\tau}^{[i]}/\omega_{G,\tau}^{[i-1]})^{-1})
\end{equation*}
Dans le cas Hilbert par exemple, c'est-à-dire le cas où $\mathrm{dim}(G)=dg=d$, on a la proposition suivante caractérisant le lieu de Rapoport :
\begin{prop}
\label{PropHassePartiel}
On suppose que $S=\spec\, k$. Les conditions suivantes sont équivalentes :
\begin{enumerate}
\item $\mathrm{Hodge}_{\tau}(G,\iota)=(e,0)$
\item $m_{\tau}^{(i)}\neq 0$ pour tout $2\leq i\leq e$
\item $\omega_{G,\tau}$ est un $\OK_L\otimes_{\mathbb{F}_p}\OK_S$-module libre de rang $1$.
\end{enumerate}
\end{prop}
\begin{proof}
$(1)$ implique $(2)$ et $(3)$ implique $(1)$ sont évidents. Pour $(2)$ implique $(3)$, il suffit de prendre un vecteur $v\in\omega_{\tau}^{(e)}\backslash \omega_{\tau}^{(e-1)}$, et de voir qu $\langle \varpi^{e-1}v,\dots,\varpi v,v\rangle$ est une base de $\omega_{\tau}^{(e)}$.
\end{proof}
\begin{remark}
La même preuve fonctionne également pour le cas Hilbert-Siegel.
\end{remark}

\begin{definition}
Soit $(G,\iota,\omega_{G,\tau}^{(\cdot)})$ un groupe $p$-divisible sur un corps $k$ avec donnée de Pappas-Rapoport. Pour tout $1\leq i \leq e$ on définit :
\begin{equation*}
\mathrm{Hodge}(\omega_{G,\tau}^{(i)}):=\mathrm{Hodge}(\Lambda_{\omega_{G,\tau}^{(i)}})=\mathrm{Inv}(\Lambda_{\omega_{G,\tau}^{(i)}},\Lambda_0)\in\mathbb{X}_{*}(T)^+
\end{equation*}
via la construction \eqref{LambdaF} (voir la section \ref{Hodge=KR} pour la description explicite de $\mathrm{Hodge}(\omega_{G,\tau}^{(e)})$).
\end{definition}

\begin{remark}
Dans la définition ci-dessus on adopte une convention différente que celle utilisée dans \cite{BijHodge}. Plus précisément, si $(G,\iota,\omega_{G,\tau}^{(\cdot)})$ un groupe $p$-divisible  sur un corps $k$ avec donnée de Pappas-Rapoport, alors le faisceau $\omega_{G,\tau}^{i}$ est un $\OK_L\otimes k$-module satisfaisant $\varpi^i\cdot \omega_{G,\tau}^{(i)}=0$. Via l'identification $\varpi\mapsto u$ on peut donc le voir comme un $k[u]/(u^e)$-module ou comme un $k[u]/(u^i)$-module. Dans la construction \eqref{LambdaF} on le voit comme un $k[u]/(u^e)$-module, alors que dans \cite{BijHodge} il est considéré comme un $k[u]/(u^i)$-module.
\end{remark}

Par la suite nous aurons besoin du lemme suivant : 

\begin{lemma}
\label{LemmaHodge}
Soit $\tau\in\Sigma^{nr}$, $2\leq i\leq e$ et $x\in \SSplit$. Notons $(G,\iota,\omega_{G,\tau}^{(\cdot)})$ le groupe $p$-divisible avec donnée de Pappas-Rapoport associé.  Si $m_{\tau}^{(i)}(x)=0$ alors on a l'égalité
\begin{equation*}
\mathrm{Hodge}(\omega_{G,\tau}^{(i)})=\mathrm{Hodge}(\omega_{G,\tau}^{(i-2)})-(1,1)
\end{equation*}
\end{lemma}
\begin{proof}
Par définition si $m_{\tau}^{(i)}=0$ on a $\omega_{\tau}^{(i)}\subset \varpi^{-1}(\omega_{\tau}^{(i-2)})$ (ici on utilise que les gradués sont de dimension $1$). Un simple calcul des dimensions respectives montre que cette inclusion est une égalité. Le résultat en découle.
\end{proof}
\begin{remark}
On fera attention au fait que la réciproque est bien sûr fausse.
\end{remark}
\begin{remark}
On fera également attention au fait que le lemme précédent n'est pas vrai dans le cas Hilbert-Siegel pour $g\geq 2$. Cela vient du fait que le determinant peut être nul sans que l'application $M_{\tau}^{(i)}$ soit nulle pour autant.
\end{remark}
\subsubsection{}  Le Verschiebung $V:G^{(p)}\rightarrow G$ induit une application 
\begin{equation}
V_{\tau}:\mathscr{E}_{\tau}\rightarrow \omega_{G,\sigma^{-1}\tau}^{(p)}=(\omega_{G,\sigma^{-1}\tau}^{(e)})^{(p)}
\end{equation}
(voir \ref{FactoVerFrob}). On considère la composée :
\begin{equation*}
\mathscr{E}_{\tau}[\varpi]\simeq \mathscr{E}_{\tau}/\mathscr{E}_{\tau}[\varpi^{e-1}]\xrightarrow{V_{\tau}}(\omega_{G,\sigma^{-1}\tau}^{(e)}/\omega_{G,\sigma^{-1}\tau}^{(e-1)})^{(p)}
\end{equation*}
 Autrement dit si $x\in\mathscr{E}_{\tau}[\varpi]$ alors $x=\varpi^{e-1}\cdot y$ pour un certains $y\in\mathscr{E}_{\tau}$, on lui associe alors $V_{\tau}(y)$.
En restreignant ce morphisme à $\omega_{G,\tau}^{(1)}\subset \mathscr{E}_{\tau}[\varpi]$ on obtient finalement un morphisme :
\begin{equation}
\label{defHa1}
\mathrm{Ha}_{\tau}:\omega_{G,\tau}^{(1)}\rightarrow (\omega_{G,\sigma^{-1}\tau}^{(e)}/\omega_{G,\sigma^{-1}\tau}^{(e-1)})^{(p)}
\end{equation}
et donc une section
\begin{equation*}
ha_{\tau}(G)\in H^0(S,\mathrm{det}\, (\omega_{G,\sigma^{-1}\tau}^{(e)}/\omega_{G,\sigma^{-1}\tau}^{(e-1)})^{\otimes p}\otimes \mathrm{det}\, (\omega_{G,\tau}^{(1)})^{-1})
\end{equation*}
également appelé invariant de Hasse partiel. 
\begin{remark}
Pour uniformiser les notations on définit $m_{\tau}^{(1)}:=ha_{\tau}$.
\end{remark}

\subsection{Stratification}

Dans \cite{DiamondKassei} les auteurs se sont intéressés à la stratification de la fibre spéciale induite par ces invariants de Hasse partiels. Plus précisément si on note
\begin{equation*}
\SSplit_T=\{x\in\SSplit |\ m_{\tau}^i(x)=0\ \text{ssi}\ (\tau,i)\in T\},\ \ \ T\subset\Sigma^{nr}\times \{1,\dots,e\}
\end{equation*}
les sous schémas localement fermés définis comme les lieux d'annulations des invariants de Hasse partiels d'indice contenu dans $T$, alors on a le théorème suivant
\begin{theorem}[\cite{DiamondKassei}, Proposition 5.8]
\label{TheoDK}
Les $(\SSplit_{T})_{T}$ pour $T\subset\Sigma^{nr}\times \{1,\dots,e\}$ définissent une bonne stratification de $\SSplit$. Plus précisément on a pour tout $T\subset\Sigma^{nr}\times \{1,\dots,e\}$ :
\begin{equation*}
\overline{\SSplit_{T}}=\bigcup_{T'\subset {T}}\SSplit_{T'}
\end{equation*}
De plus chacune des strates $\SSplit_{T}$  est non vide, quasi affine, et équidimensionnelle de dimension $d-|T|$.
\end{theorem}

\subsection{Interaction avec la stratification de Hodge : le cas $e=4$}

On dispose de deux stratifications de la fibre spéciale du modèle de Pappas Rapoport :

\begin{equation*}
\SSplit=\coprod_{\lambda\in\mathrm{Adm}(\mu)_K}\SSplit_{\lambda},\ \ \ \SSplit=\coprod_{T\in\mathscr{T}}\SSplit_{T}
\end{equation*}
où  $\mathscr{T}=\mathscr{P}(\Sigma^{\mathrm{nr}}\times\{1,\dots,e\})$ désigne les sous ensembles de  $\Sigma^{\mathrm{nr}}\times\{1,\dots,e\}$.
Il est naturel de se demander dans quelle mesure la décomposition 
\begin{equation*}
\SSplit=\coprod_{(\lambda,T)\in\mathrm{Adm}(\mu)_K\times\mathscr{T}}\SSplit_{(\lambda,T)},\ \ \ \ \SSplit_{(\lambda,T)}:=\SSplit_{\lambda}\cap\SSplit_{T}
\end{equation*}
fourni une bonne stratification. Deux problèmes se posent alors :
\begin{enumerate}
\item Quelles sont les strates $\SSplit_{(\lambda,T)}$ qui sont non vides ? 
\item Si cette stratification est une bonne stratification, quelle est la relation d'ordre sur les couples $(\lambda,T)$ qui décrit les relations d'adhérences des strates ?
\end{enumerate}
Pour le point $(1)$, la Proposition \ref{PropHassePartiel} nous dit que pour $\lambda=(e,0)$ alors $\SSplit_{(\lambda,T)}=\emptyset$ si et seulement si il existe un indice $(\tau, i)\in T\in\mathscr{T}$ avec $i\geq 2$. De manière générale, il semble difficile de prédire quelles sont les strates non-vides. Nous allons répondre à cette question et  décrire l'interaction entre le stratification de Hodge et celle par les invariants $(m^{(i)})_{i}$ dans le cadre d'une extension $L/\Q_p$ totalement ramifiée de degré $e=4$. L'extension étant totalement ramifiée, nous pouvons omettre l'indice $\tau$ et on a donc $\mathscr{T}=\mathscr{P}(\{1,\dots,e\})$.

\begin{definition}
On définit les ensembles :
\begin{equation}
\label{DefTlambda}
\mathscr{T}_{\lambda}=\{T\in \mathscr{P}(\{2,\dots,e\})\ |\ \SSplit_{(\lambda,T)}\neq\emptyset\}\ \ \ \ \forall\, \lambda\in\mathrm{Adm}(\mu)_K
\end{equation}
\begin{equation*}
\mathscr{T}_{\mu_{\bullet}}=\{(\lambda,T)\in\mathrm{Adm}(\mu)_K\times\mathscr{P}(\{2,\dots,e\})\ |\ \SSplit_{(\lambda,T)}\neq\emptyset\}
\end{equation*}
\end{definition}
Par définition de ces ensembles on a une décomposition de l'ensemble $\mathscr{T}_{\mu_{\bullet}}$ :
\begin{equation*}
\mathscr{T}_{\mu_{\bullet}}=\coprod_{\lambda\in\mathrm{Adm}(\mu)_K}\mathscr{T}_{\lambda}
\end{equation*}
\begin{remark}
La stratification définie par l'ensemble $\mathscr{T}_{\mu_{\bullet}}$  est purement \og linéaire \fg\  dans le sens où elle est définie par les invariants $m_i$ pour $i\geq 2$ et par l'invariant de Hodge, qui sont des invariants \og linéaires \fg\ (en opposition à l'invariant $m_1=ha$ qui est un invariant $\sigma$-linéaire). 
\end{remark}
\begin{remark}
En fait la stratification induite par l'ensemble $\mathscr{T}_{\mu_{\bullet}}$ correspond à une stratification du produit de convolution $\widetilde{\Gr}_{\mu_{\bullet}}$. La stratification définie ci-dessus est donc \og linéaire\fg\  dans le sens où elle ne dépend que du modèle local $\Split$.
\end{remark}
\begin{definition}
On définit les ensembles suivants :
\begin{equation*}
\mathscr{T}_{\mu_{\bullet}}^{(m_1=0)}=\{(\lambda,T)\in\mathrm{Adm}(\mu)_K\times \mathscr{T}\ |\  \SSplit_{(\lambda,T)}\neq\emptyset,\ m_1\in T\}
\end{equation*}
\begin{equation*}
\mathscr{T}_{\mu_{\bullet}}^{(m_1\neq 0)}=\{(\lambda,T)\in\mathrm{Adm}(\mu)_K\times \mathscr{T}\ |\  \SSplit_{(\lambda,T)}\neq\emptyset,\ m_1\notin T\}
\end{equation*}
\begin{equation*}
\mathscr{A}_{\mu_{\bullet}}=\{(\lambda,T)\in\mathrm{Adm}(\mu)_K\times\mathscr{T}\ |\ \SSplit_{(\lambda,T)}\neq\emptyset\}
\end{equation*}
\end{definition}
On a par définition de ces ensembles :
\begin{equation*}
\mathscr{A}_{\mu_{\bullet}}=\mathscr{T}_{\mu_{\bullet}}^{(m_1=0)}\coprod\mathscr{T}_{\mu_{\bullet}}^{(m_1\neq0)}
\end{equation*}

Dans un premier temps nous allons décrire l'interaction entre les strates de Hodge et les strates définies par les $m^{(i)}$ pour $i=2,3,4$.  Nous intégrerons l'invariant $m_1=ha$ à notre raisonnement par la suite. Notre problème devient alors simplement un problème d'algèbre linéaire sur la filtration de Pappas-Rapoport $\omega^{(\cdot)}$. Pour alléger les notations nous noterons $X_{\lambda}^{(m_i)_{i\in T}}=\SSplit_{(\lambda,T)}$.

\subsubsection{Invariant $(m_i)_{i\geq 2}$}
 Le schéma ci-dessous décrit les strates non vides et les relations d'adhérences entre elles avec la convention $A\rightarrow B$ si $A\subset\overline{B}$. 
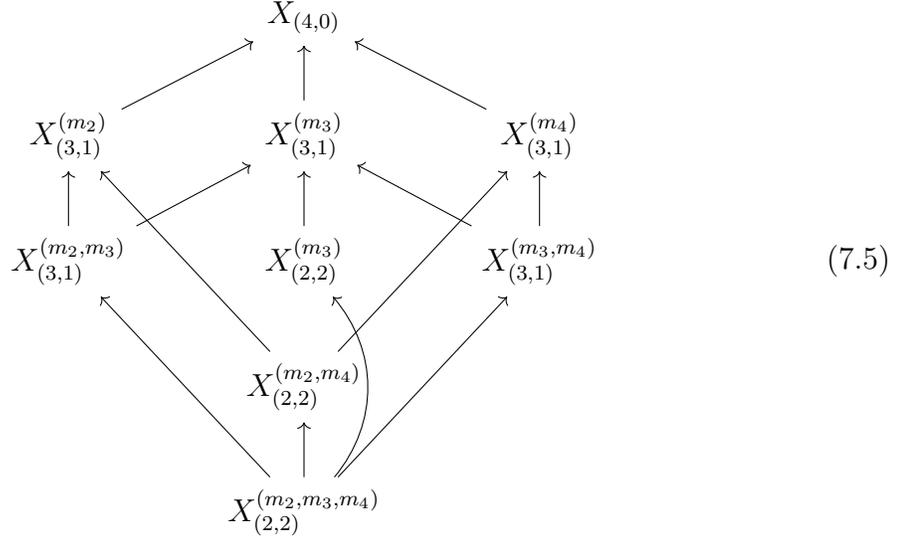
\begin{equation}
\label{deformation(m_i)}
\begin{tikzcd}
	& {X_{(4,0)}} \\
	{X_{(3,1)}^{(m_2)}} & {X_{(3,1)}^{(m_3)}} & {X_{(3,1)}^{(m_4)}} \\
	{X_{(3,1)}^{(m_2,m_3)}} & {X_{(2,2)}^{(m_3)}} & {X_{(3,1)}^{(m_3,m_4)}} \\
	& {X_{(2,2)}^{(m_2,m_4)}} \\
	& {X_{(2,2)}^{(m_2,m_3,m_4)}}
	\arrow[from=2-2, to=1-2]
	\arrow[from=2-1, to=1-2]
	\arrow[from=2-3, to=1-2]
	\arrow[from=3-2, to=2-2]
	\arrow[from=3-1, to=2-1]
	\arrow[from=3-3, to=2-3]
	\arrow[from=4-2, to=2-1]
	\arrow[from=4-2, to=2-3]
	\arrow[from=5-2, to=3-1]
	\arrow[from=5-2, to=3-3]
	\arrow[from=5-2, to=4-2]
	\arrow[bend right=40, from=5-2, to=3-2]
	\arrow[from=3-1, to=2-2]
	\arrow[from=3-3, to=2-2]
\end{tikzcd}
\end{equation}

Le problème étant pour le moment purement un problème d'algèbre linéaire, nous allons utiliser les notations utilisées dans le cadre des grassmanniennes affines (via la construction \eqref{LambdaF}). 

 Tout d'abord commençons par prouver l'assertion sur les strates vides. Nous allons donner tous les détails pour le cas $(1)$. Les autres cas étant très similaire, nous donnerons seulement les arguments importants.
\begin{enumerate}
\label{StrataVide}

\item $X_{(3,1)}^{(m_2,m_4)}=\emptyset$. A priori on a une décomposition en union disjoint :
\begin{equation*}
X_{}^{(m_2,m_4)}= X_{(3,1)}^{(m_2,m_4)}\cup X_{(2,2)}^{(m_2,m_4)}
\end{equation*} 
(on rappelle que la strate $X_{(4,0)}^{(m_2,m_4)}$ est vide d'après le Lemme \ref{PropHassePartiel}). Soit $x\in X_{}^{(m_2,m_4)} $. Pour simplifier les notations on note $\Lambda_1\subset\dots\subset\Lambda_4$ la filtration $\Lambda_{\omega_{x}^{(1)}}\subset\dots\subset\Lambda_{\omega_x^{(4)}}$ associée à ce point (voir \eqref{LambdaF}). D'après le Lemme \ref{LemmaHodge} puisque $m_2(x)=m_4(x)=0$ on a 
\begin{align*}
\mathrm{Hodge}(\Lambda_4)&=\mathrm{Hodge}(\Lambda_2)-(1,1)\\
&=\mathrm{Hodge}(\Lambda_0)-(1,1)-(1,1)\\
&=(2,2)
\end{align*}
Par conséquent $X^{(m_2,m_4)}=X_{(2,2)}^{(m_2,m_4)}$ et $X_{(3,1)}^{(m_2,m_4)}=\emptyset$.
\item $X_{(2,2)}^{(m_4)}=\emptyset$. Puisque $m_4=0$ on a d'après \ref{LemmaHodge} :
\begin{align*}
\mathrm{Hodge}(\Lambda_4)&=\mathrm{Hodge}(\Lambda_2)-(1,1)\\
&=(4,2)-(1,1)\\
&=(3,1)
\end{align*}
où l'on a utilisé $m_2\neq 0$ pour obtenir $\mathrm{Hodge}(\Lambda_2)=(4,2)$.
\item $X_{(2,2)}^{(m_3,m_4)}=\emptyset$. Même raison que ci-dessus.
\item $X_{(2,2)}^{(m_2)}=\emptyset$. Soit $v\in \omega^{(4)}\backslash \omega^{(3)}$. On a $u\cdot v\in \omega^{(3)}\backslash \omega^{(2)}$ car $m_4\neq 0$. Ensuite puisque $m_2=0$ on a $\omega^{(2)}=\mathscr{E}[u]$ et par conséquent $u^2\cdot v\neq 0$. En particulier $\mathrm{Hodge}(\omega^{(4)})>(2,2)$
\item $X_{(2,2)}^{(m_2,m_3)}=\emptyset$. Même raison que ci-dessus.
\end{enumerate}

Il assez simple de trouver des filtrations adéquates pour prouver que  les strates du schéma ci-dessus sont non vides.  En d'autres termes on a calculé les ensembles $\mathscr{T}_{\lambda}$ pour tout $\lambda\in\mathrm{Adm}(\mu)_K$ (avec les notations de \eqref{DefTlambda}) : 
\begin{align}
\mathscr{T}_{(4,0)}&=\{\emptyset\}\\
\mathscr{T}_{(3,1)}&=\{(m_2),(m_3),(m_4),(m_2,m_3),(m_3,m_4)\}\\
\mathscr{T}_{(2,2)}&=\{(m_3),(m_2,m_4),(m_4),(m_2,m_3,m_4)\}
\end{align}

Montrons maintenant l'assertion sur les relations d'adhérences.\\

D'après le Théorème $\ref{TheoDK}$ les invariants $(m_i)_{i}$ forment une bonne stratification de $\SSplit$. En particulier d'après le Lemme $\ref{Lemma1}$ cela veut dire que l'on peut \og inverser\fg un invariant sans toucher aux autres.  Plus précisément :
\begin{prop}
Si $x\in X_{\lambda}^{(m_i)_{i\in T}}$ avec $T\neq\emptyset $ alors pour tout $i_0\in T$ il existe $\lambda'\in\mathrm{Adm}(\mu)_K$ et $y\in X_{\lambda'}^{(m_i)_{i\in T'}}$ tel que $x\in\overline{\{y\}}$ où $T'=T\backslash\{i_0\}$. 
\end{prop}
\begin{proof}
C'est la combinaison de \ref{TheoDK} et \ref{Lemma1}. 
\end{proof}

Le problème est que dans la proposition ci dessus on ne contrôle pas le polygone de Hodge lors de la déformation. Cependant, nous avons prouver que certaines strates $X_{\lambda}^{(m_i)_{i\in T}}$ étaient vides, et nous pouvons l'utiliser pour décrire le polygone de Hodge lors des déformations. 

\begin{example}
\label{Example11}
Soit $x\in X_{(2,2)}^{(m_2,m_3,m_4)}$ et soit $y\in X_{\lambda}^{(m_2,m_4)}$ un point tel que $x\in\overline{\{y\}}$ (fourni par la proposition précédente). D'après \ref{StrataVide} la strate $X_{(3,1)}^{(m_2,m_4)}$ est vide ce qui impose l'égalité $\lambda=(2,2)$. Cela montre que pour tout point $x\in X_{(2,2)}^{(m_2,m_3,m_4)}$ il existe un $y\in X_{(2,2)}^{(m_2,m_4)}$. Autrement dit cela prouve la relation d'adhérence $X_{(2,2)}^{(m_2,m_3,m_4)}\subset\overline{X_{(2,2)}^{(m_2,m_4)}}$.
\end{example}

Nous pouvons également utiliser le fait que l'invariant de Hodge ne peut qu'augmenter par générisation (autrement dit le polygone de Hodge s'abaisse par générisation). Plus précisément si $x\in X_{\lambda}$ et $y\in X_{\lambda'}$ tel que $x\in\overline{\{y\}}$ alors $\lambda\leq \lambda '$, ce qui peut être déduit du Théorème \ref{TheoremPrincipal} qui est cependant beaucoup plus fort.

\begin{example}
\label{Example22}
Soit $x\in X_{(3,1)}^{(m_2,m_3)}$ et soit $y\in X_{\lambda}^{(m_3)}$ un point tel que $x\in\overline{\{y\}}$ (fourni par la proposition précédente). L'invariant de Hodge ne pouvant qu'augmenter par générisation on a nécessairement $\mathrm{Hodge}(y)\geq (3,1)$. Le cas $(4,0)$ étant impossible d'après \ref{StrataVide} on a nécessairement $\mathrm{Hodge}(y)=(3,1)$. Cela montre que pour tout point $x\in X_{(3,1)}^{(m_2,m_3)}$ il existe un $y\in X_{(3,1)}^{(m_3)}$. Autrement dit cela prouve la relation d'adhérence $X_{(3,1)}^{(m_2,m_3)}\subset\overline{X_{(3,1)}^{(m_3)}}$.
\end{example}

Les deux exemples ci dessus fonctionnent pour toutes les relations d'adhérences à l’exception de :
\begin{equation*}
X_{(2,2)}^{(m_2,m_3,m_4)}\subset\overline{X_{(2,2)}^{(m_3)}},\ \ \ \ \ X_{(2,2)}^{(m_3)}\subset\overline{X_{(3,1)}^{(m_3)}}
\end{equation*}
Dans le premier cas ce qui pose problème c'est que la strate $X_{(3,1)}^{(m_3)}$ est non vide et par conséquent le raisonnement ci dessus ne fonctionne pas. Dans le second cas il s'agit de déformer le polygone de Hodge au sein de la strate définie par l'équation $m_3=0$ et $m_2,m_4\neq 0$.

\begin{enumerate}
\label{Deformation1}
\item $X_{(2,2)}^{(m_2,m_3,m_4)}\subset\overline{X_{(2,2)}^{(m_3)}}$. Soit $x\in X_{(2,2)}^{(m_2,m_3,m_4)}$ et $0\subset\omega^{(1)}\subset\omega^{(2)}\subset\omega^{(3)}\subset\omega^{(4)}$ la filtration associée. Par définition on a les égalités :
\begin{equation*}
\omega^{(2)}=\mathscr{E}[u]=\omega^{(1)}\oplus\langle v_2\rangle,\ \ \ \omega^{(3)}=u^{-1}(\omega^{(1)}),\ \ \  \ \omega^{(4)}=u^{-1}(\omega^{(2)})=\mathscr{E}[u^2]
\end{equation*}
pour un certain $v_2\in\omega^{(2)}$ que l'on fixe. On définit la déformation sur $k[\![t]\!]$ comme suit : 
\begin{equation*}
\widetilde{\mathscr{E}}=\mathscr{E}\otimes_k k[\![t]\!],\ \ \ \ \widetilde{\omega}^{(1)}=\omega^{(1)}\otimes_k k[\![t]\!]
\end{equation*}
\begin{equation*}
\widetilde{\omega}^{(2)}=\widetilde{\omega}^{(1)}\oplus\langle v_2+t v\rangle,\ \ \ \widetilde{\omega}^{(3)}=u^{-1}(\widetilde{\omega}^{(1)}),\ \ \ \ \widetilde{\omega}^{(4)}=\widetilde{\mathscr{E}}[u^2]
\end{equation*}
pour un certain  $v\in u^{-1}(\widetilde{\omega}^{(1)})\backslash \widetilde{\mathscr{E}}[u]$ que l'on choisit. Par construction en fibre générique cette filtration satisfait bien les équations $m_2\neq 0,m_4\neq 0$ , $m_3= 0$ et $\mathrm{Hodge}(\widetilde{\omega}^{(4)})=(2,2)$.
\item $X_{(2,2)}^{(m_3)}\subset\overline{X_{(3,1)}^{(m_3)}}$. Soit $x\in X_{(2,2)}^{(m_3)}$ et $0\subset\omega^{(1)}\subset\omega^{(2)}\subset\omega^{(3)}\subset\omega^{(4)}$ la filtration associée. Par définition on a les égalités :
\begin{equation*}
\omega^{(3)}=u^{-1}(\omega^{(1)}),\ \ \  \ \omega^{(4)}=\mathscr{E}[u^2]
\end{equation*}
Soit $v^{(4)}\in\omega^{(4)}$ tel que $\omega^{(4)}=\omega^{(3)}\oplus\langle v^{(4)}\rangle$. On choisit un élément $v\in u^{-1}(\omega^{(3)})\backslash\mathscr{E}[u^2]$. On définit la déformation sur $k[\![t]\!]$ comme suit : 
\begin{equation*}
\widetilde{\mathscr{E}}=\mathscr{E}\otimes_k k[\![t]\!],\ \ \ \ \widetilde{\omega}^{(i)}=\omega^{(i)}\otimes_k k[\![t]\!],\ \ i=1,2,3
\end{equation*}
et 
\begin{equation*}
\widetilde{\omega}^{(4)}=\widetilde{\omega}^{(3)}\oplus\langle v^{(4)}+tv\rangle
\end{equation*}
En fibre générique cette filtration satisfait bien les équations $m_2m_4\neq 0$ , $m_3= 0$ et $\mathrm{Hodge}(\widetilde{\omega}^{(4)})=(3,1)$ car par construction $u^2(v_4+tv)\neq 0$.
\end{enumerate}

\subsubsection{Invariant $m_1=ha$}

Nous allons maintenant nous intéresser à l'invariant $m_1=ha$. On aimerait pouvoir \og inverser\fg\ $\mathrm{ha}$ sans modifier les autres invariants de Hasse partiels et sans modifier l'invariant de Hodge. On a la proposition suivante valable pour une extension $L/\Q_p$ totalement ramifiée de degré quelconque :
\begin{prop}
\label{PropHA}
Soit $x\in\SSplit_{(\lambda,T)}$. On suppose que $1\in T$ c'est-à-dire que $m_1(x)=0$. Alors il existe $y\in \SSplit_{(\lambda,T')}$ tel que $y\rightsquigarrow x$ où $T'=T\backslash\{1\}$. 
\end{prop}
\begin{proof}
Soit $x\in\SSplit_{(\lambda,T)}$ défini sur un corps $k$. On note $\mathscr{E}$ le cristal associé. Nous allons construire une déformation sur $k[\![t]\!]$ de ce cristal telle qu'en fibre générique on ait $\mathrm{ha}\neq 0$, et sans que les autres invariants ne soient modifiés. Soit $\mathscr{F}^{(1)}:=F((\omega^{(e+1)})^{(p)})$ où $\omega^{(e+1)}:=\varpi^{-1}(\omega^{(e-1)})$. D'après \cite{BijDuality} Proposition 3.12, il existe un isomorphisme $\mathscr{E}[\varpi]/\mathscr{F}^{(1)}\simeq \big(\omega^{(e)}/\omega^{(e-1)}\big)^{(p)} $ faisant commuter le diagramme suivant :
\[\begin{tikzcd}
	{\omega^{(1)}} & {\mathscr{E}[\varpi]/\mathscr{F}^{(1)}} \\
	& {\big(\omega^{(e)}/\omega^{(e-1)}\big)^{(p)}}
	\arrow[from=1-1, to=1-2]
	\arrow["\simeq", from=1-2, to=2-2]
	\arrow["{ha}"', from=1-1, to=2-2]
\end{tikzcd}\]
Puisque les gradués sont de dimension $1$, la condition $ha=0$ est équivalente à l'égalité $\omega^{(1)}=\mathscr{F}^{(1)}$ dans $\mathscr{E}[\varpi]$. Dans tout ce qui suit on notera $\mathscr{E}_{(n)},\mathscr{F}^{(1)}_{(n)},\omega_{(n)}^{(k)}$ etc les objets définis sur $R_n=k[t]/(t^n)$.

D'après le lemme \ref{LemmaRelevement} on dispose d'un relèvement $\mathscr{F}_{(2)}^{(1)}$ sur $R_2$ qui ne dépend pas du choix des relèvements $\omega_{(2)}^{(e)}$ et $\omega_{(2)}^{(e-1)}$ sur $R_2$. Soit $\Lambda_{(1)}^{(1)}\subset\dots\subset\Lambda_{(1)}^{(e)}\in\widetilde{\Gr}_{\mu_{\bullet}}(R_1)$ la filtration associée à  $0\subset\omega_{(1)}^{(1)}\subset\dots\subset\omega_{(1)}^{(e)}$ via \ref{PlongPR}. Soit $(g_1,\dots,g_e)\in \mathrm{GL}_2(R_1(\!(u)\!))$ tel que :
\begin{equation*}
\Lambda_{(1)}^{(i)}=g_i\cdot\Lambda_{(1)}^{(i-1)},\ \ \ \ \Lambda_{(1)}^{(0)}=u^e\cdot\Lambda_0
\end{equation*}
Soit $\Lambda_{\mathscr{F}_{(2)}^{(1)}}$ le $R_2[\![u]\!]$-réseau associé à $\mathscr{F}_{(2)}^{(1)}$ via \ref{PlongPEL}. On définit ensuite :
\begin{equation*}
\Lambda_{(2)}^{(i+1)}=g_{i+1}\cdot\Lambda_{(1)}^{(i)},\ \ \ \ \Lambda_{(2)}^{(1)}\neq\Lambda_{\mathscr{F}_{(2)}^{(1)}}
\end{equation*}
où chacun des $g_{i+1}\in \mathrm{GL}_2(R_2(\!(u)\!))$ est vu ici via la section $R_1\rightarrow R_2$ et $\Lambda_{(2)}^{(1)}$ désigne n'importe quel $R_2[\![u]\!]$-réseau relevant $\Lambda_{(1)}^{(1)}$ et satisfaisant la condition $\Lambda_{(2)}^{(1)}\neq\Lambda_{\mathscr{F}_{(2)}^{(1)}}$. Si on note $g_{\mathscr{F}}\in\mathrm{GL}_2(R_2(\!(u)\!))$ le lacet définissant $\Lambda_{\mathscr{F}_{(2)}^{(1)}}$ alors il suffit par exemple de choisir $\Lambda_{(2)}^{(1)}=(g\cdot g_{\mathscr{F}})\cdot u^e\Lambda_0$ avec $g\in \mathrm{GL}_2(R_2[\![u]\!])$ relevant $\mathrm{Id}\in \mathrm{GL}_2(R_2[\![u]\!])$ et $g\neq\mathrm{Id}$. Toujours grâce au procédé \ref{PlongPR} on obtient une filtration $0\subset\omega_{(2)}^{(1)}\subset\dots\subset\omega_{(2)}^{(e)}$ satisfaisant $\omega_{(2)}^{(1)}\neq\mathscr{F}_{(2)}^{(1)}$. On continue le processus par récurrence sur $n\geq 2$. Supposons donnés $\mathscr{E}_{(n)},\omega_{(n)}^{(k)}$ etc sur $R_n$. On définit alors la déformation sur $R_{n+1}$ comme étant (avec les notations similaires au cas $n=2$) :
\begin{equation*}
\Lambda_{(n+1)}^{(i+1)}=g_{i+1}\cdot\Lambda_{(n+1)}^{(i)},\ \ \ \ \Lambda_{(n+1)}^{(1)}\otimes_{R_{n+1}}R_n =\Lambda_{(n)}^{(1)}
\end{equation*}
où $\Lambda_{(n+1)}^{(1)}$ désigne n'importe quel relèvement de $\Lambda_{(n)}^{(1)}$ sur $R_{n+1}$. Cette propriété de relèvement nous assure que la condition $\Lambda_{(n+1)}^{(1)}\neq \Lambda_{\mathscr{F}_{(n+1)}^{(1)}}$ soit satisfaite (car elle l'est après réduction à $R_n$). Toujours via le procédé \ref{PlongPR} on obtient une filtration  $0\subset\omega_{(n+1)}^{(1)}\subset\dots\subset\omega_{(n+1)}^{(e)}$ sur $R_{n+1}$ satisfaisant $\omega_{(n+1)}^{(1)}\neq\mathscr{F}_{(n+1)}^{(1)}$. On obtient par passage à la limite un groupe $p$-divisible $G\in\mathtt{BT}^{\mathrm{PR}}(R)$ muni d'une action $\iota:\OK_L\rightarrow\mathrm{End}(G)$ et d'une filtration de Pappas-Rapoport. Par Gronthendieck-Messing et Serre-Tate cela nous fournit un morphisme :
 \begin{equation*}
 \spec\, k[\![t]\!]\rightarrow \SSplit
 \end{equation*}
En fibre générique la filtration satisfait les propriétés souhaitées : par construction  les positions relatives de la filtration n'ont pas changé et par conséquent les invariants $(m_i)_i$ et l'invariant de Hodge n'ont pas changé. Enfin $\mathrm{lim}\, \mathscr{F}_{(n)}^{(1)}\neq \mathrm{lim}\, \omega_{(n)}^{(1)}$, toujours par construction, ce qui nous assure qu'en fibre générique on ait bien $ha\neq 0$.
\end{proof}

\begin{remark}
Ce qui est assez surprenant dans la preuve ci-dessus est que l'on change une donnée $\sigma$-linéaire (l'invariant $m_1$) via une déformation d'une donnée linéaire (filtration de Pappas-Rapoport). C'est le Lemme \ref{LemmaRelevement} qui rend ce processus possible. Si nous avions voulu déformer directement sur $k[\![t]\!]$, il aurait été difficile de calculer le faisceau $\mathscr{F}^{(1)}$.
\end{remark}

\begin{remark}
La proposition ci-dessus nous dit que l'on peut inverser l'invariant $m_1$ sans toucher aux autres invariants. C'est une déformation au sein d'une strate fixée par les invariants $(m_i)_{i\geq 2}$ et l'invariant de Hodge.
\end{remark}

Nous devons maintenant traiter la réciproque : l'équation $m_1=0$ est-elle satisfaite sur la fibre générique des déformations définies en \ref{deformation(m_i)} ? Autrement dit, peut on modifier l'invariant de Hodge et les $(m_i)_i$  tout en satisfaisant $\mathrm{ha}=0$ ? Pour les mêmes raisons que dans les exemples \ref{Example11} et \ref{Example22} il suffit de traiter les cas :
\begin{equation*}
X_{(2,2)}^{(m_1,m_2,m_3,m_4)}\subset\overline{X_{(2,2)}^{(m_1,m_3)}},\ \ \ \ \ X_{(2,2)}^{(m_1,m_3)}\subset\overline{X_{(3,1)}^{(m_1,m_3)}}
\end{equation*}
Le problème est que les déformations de \ref{Deformation1} ne sont pas assez précises car on ne contrôle pas le Frobenius lors de la déformation. Nous allons donc adapter ces déformations à l'idée de la preuve de la Proposition \ref{PropHA} qui est de déformer étape par étape le long des $R_{n+1}\rightarrow R_n$ et d'utiliser le Lemme \ref{LemmaRelevement}.
\begin{enumerate}
\item $X_{(2,2)}^{(m_1,m_2,m_3,m_4)}\subset\overline{X_{(2,2)}^{(m_1,m_3)}}$. Soit $x\in X_{(2,2)}^{(m_1m_2,m_3,m_4)}$ et $0\subset\omega_{(1)}^{(1)}\subset\omega_{(1)}^{(2)}\subset\omega_{(1)}^{(3)}\subset\omega_{(1)}^{(4)}$ la filtration associée. Par définition on a les égalités :
\begin{equation*}
\omega_{(1)}^{(1)}=\mathscr{F}_{(1)}^{(1)},\ \ \ \ \omega_{(1)}^{(2)}=\mathscr{E}_{(1)}[u]
\end{equation*}
\begin{equation*}
\omega_{(1)}^{(3)}=u^{-1}(\omega_{(1)}^{(1)}),\ \ \  \ \omega_{(1)}^{(4)}=u^{-1}(\omega_{(1)}^{(2)})=\mathscr{E}_{(1)}[u^2]
\end{equation*}
On fixe des bases :
\begin{equation*}
\langle v_{(1)}^{(1)}\rangle=\omega_{(1)}^{(1)}, \ \ \ \langle v_{(1)}^{(1)},v_{(1)}^{(2)}\rangle=\omega_{(1)}^{(2)}
\end{equation*} 
Comme dans la preuve de la Proposition \ref{PropHA}, d'après le Lemme \ref{LemmaRelevement}, on dispose d'un relèvement $\mathscr{F}_{(2)}^{(1)}$ de $\mathscr{F}_{(1)}^{(1)}$ sur $R_2$. On fixe un relèvement $v_{(2)}^{(1)}$ sur $R_2$ de $v_{(1)}^{(1)}$ tel que $\mathscr{F}_{(2)}^{(1)}=\langle v_{(2)}^{(1)}\rangle$ et un  élément $\alpha_{(2)}\in u^{-1}(\mathscr{F}_{(2)}^{(1)})\backslash \mathscr{E}_{(2)}[u]$. On choisit également un élément $v_{(2)}^{(2)}\in\mathscr{E}_{(2)}[u]$ qui relève $v_{(1)}^{(2)}$. On définit alors la déformation sur $R_2$ comme suit : 
\begin{equation*}
\mathscr{E}_{(2)}=\mathscr{E}\otimes_{R_1}R_2\ \ \ \ \omega_{(2)}^{(1)}=\langle v_{(2)}^{(1)}\rangle=\mathscr{F}_{(2)}^{(1)}
\end{equation*}
\begin{equation*}
\omega_{(2)}^{(2)}=\omega_{(2)}^{(1)}\oplus\langle v_{(2)}^{(2)}+t\alpha_{(2)} \rangle,\ \ \ \omega_{(2)}^{(3)}=u^{-1}(\omega_{(2)}^{(1)}),\ \ \ \ \omega_{(2)}^{(4)}=\mathscr{E}_{(2)}[u^2]
\end{equation*}
Par construction on a bien $u\cdot( v_{(2)}^{(2)}+t\alpha_{(2)})\neq 0$. On continue le processus par récurrence sur $n\geq 2$ . Supposons donnés $\mathscr{E}_{(n)},\omega_{(n)}^{(k)}$ etc sur $R_n$. D'après le lemme \ref{LemmaRelevement} on dispose d'un relèvement canonique  $\mathscr{F}_{(n+1)}^{(1)}$ sur $R_{n+1}$. On définit la déformation sur $R_{n+1}$ comme suit : 
 \begin{equation*}
\mathscr{E}_{(n+1)}=\mathscr{E}\otimes_{R_1}R_{n+1}\ \ \ \ \omega_{(n+1)}^{(1)}=\mathscr{F}_{(n+1)}^{(1)}
\end{equation*}
\begin{equation*}
\omega_{(n+1)}^{(3)}=u^{-1}(\omega_{(n+1)}^{(1)}),\ \ \ \ \omega_{(n+1)}^{(4)}=\mathscr{E}_{(n+1)}[u^2]
\end{equation*}
 et on prend n'importe quel relèvement $\omega_{(n+1)}^{(2)}\subset u^{-1}(\mathscr{F}_{(n+1)}^{(1)})$ de $\omega_{(n)}^{(2)}$, la condition $\omega_{(n+1)}^{(2)}\neq\mathscr{E}_{(n+1)}[u]$ étant assurée par cette propriété de relèvement. On obtient par passage à la limite un groupe $p$-divisible $G\in\mathtt{BT}^{\mathrm{PR}}(R)$ qui satisfait les propriétés souhaitées.
\item $X_{(2,2)}^{(m_1,m_3)}\subset\overline{X_{(3,1)}^{(m_1,m_3)}}$. Soit $x\in X_{(2,2)}^{(m_3)}$ et $0\subset\omega_{(1)}^{(1)}\subset\omega_{(1)}^{(2)}\subset\omega_{(1)}^{(3)}\subset\omega_{(1)}^{(4)}$ la filtration associée. Par définition on a les égalités :
\begin{equation*}
\omega_{(1)}^{(1)}=\mathscr{F}_{(1)}^{(1)}, \ \ \ \ \omega_{(1)}^{(3)}=u^{-1}(\omega_{(1)}^{(1)}),\ \ \  \ \omega_{(1)}^{(4)}=\mathscr{E}_{(1)}[u^2]
\end{equation*}
Soit $v_{(1)}^{(4)}\in\omega_{(1)}^{(4)}$ tel que $\omega_{(1)}^{(4)}=\omega_{(1)}^{(3)}\oplus\langle v_{(1)}^{(4)}\rangle$. 
Comme dans la preuve de la Proposition \ref{PropHA}, d'après le Lemme \ref{LemmaRelevement}, on dispose d'un relèvement $\mathscr{F}_{(2)}^{(1)}$ de $\mathscr{F}_{(1)}^{(1)}$ sur $R_2$. 
 On définit la déformation sur $R_2$ comme suit : 
\begin{equation*}
\mathscr{E}_{(2)}=\mathscr{E}\otimes_{R_1}R_2\ \ \ \ \omega_{(2)}^{(1)}=\mathscr{F}_{(2)}^{(1)},\ \ \omega_{(2)}^{(3)}=u^{-1}(\omega_{(2)}^{(1)})
\end{equation*}
Ensuite on choisit un élément $\alpha_{(2)}\in u^{-1}(\omega_{(2)}^{(3)})\backslash\mathscr{E}_{(2)}[u^2]$ et un élément $v_{(2)}^{(4)}\in u^{-1}(\omega_{(2)}^{(3)}\cap\mathscr{E}_{(2)}[u]$ qui relève $v_{(1)}^{(4)}$. On définit alors : 
\begin{equation*}
\omega_{(2)}^{(4)}=\omega_{(2)}^{(3)}\oplus\langle v_{(1)}^{(4)}+t\alpha_{(2)}\rangle
\end{equation*}
 On continue le processus par récurrence sur $n\geq 2$ . Supposons donnés $\mathscr{E}_{(n)},\omega_{(n)}^{(k)}$ etc sur $R_n$. D'après le lemme \ref{LemmaRelevement} on dispose d'un relèvement canonique  $\mathscr{F}_{(n+1)}^{(1)}$ sur $R_{n+1}$. On définit la déformation sur $R_{n+1}$ comme suit : 
 \begin{equation*}
\mathscr{E}_{(n+1)}=\mathscr{E}\otimes_{R_1}R_{n+1},\ \ \ \ \omega_{(n+1)}^{(1)}=\mathscr{F}_{(n+1)}^{(1)},\ \ \ \ \omega_{(n+1)}^{(3)}=u^{-1}(\omega_{(n+1)}^{(1)})
\end{equation*}
et on prend n'importe quel relèvement $\omega_{(n+1)}^{(2)},\ \omega_{(n+1)}^{(4)}$. Les propriétés $\omega_{(n+1)}^{(2)}\neq \mathscr{E}_{(n+1)}[u]$ et $\omega_{(n+1)}^{(2)}\neq\mathscr{E}_{(n+1)}[u^2]$ sont assurées par cette propriété de relèvement.  On obtient par passage à la limite un groupe $p$-divisible $G\in\mathtt{BT}^{\mathrm{PR}}(R)$ qui satisfait les propriétés souhaitées.
\end{enumerate}

Avant de conclure, il reste à calculer l'ensemble $\mathscr{A}_{\mu_{\bullet}}$ des strates non vides. C'est l'objet du lemme suivant :
\begin{lemma}
Soit $(\lambda,T)\in\mathscr{T}_{\mu_{\bullet}}$. On a : 
\begin{equation*}
\mathcal{S}h_{(\lambda,T)}^{(m_1=0)}\neq \emptyset ,\ \ \ \mathcal{S}h_{(\lambda,T)}^{(m_1\neq 0)}\neq \emptyset
\end{equation*}
En d'autres termes l'ensemble des strates non vides est donné par : 
\begin{equation*}
\mathscr{A}_{\mu_{\bullet}}=\mathscr{T}_{\mu_{\bullet}}\cup(\mathscr{T}_{\mu_{\bullet}}\times \{m_1\})
\end{equation*}
\end{lemma}
\begin{remark}
Avec des mots plus concrets, le lemme ci dessus, combiné aux proposition précédentes, dit que l'invariant $\sigma$-linéaire $m_1$ interagit naïvement avec la stratification définie par l'ensemble $\mathscr{T}_{\mu_{\bullet}}$ (linéaire).
\end{remark}
\begin{proof}
D'après la proposition \ref{PropHA} et les calculs de déformations précédents il suffit de montrer que la strate $X_{(2,2)}^{(m_1,m_2,m_3,m_4)}$ est non vide.  Nous allons utiliser les travaux de Goren et Andreatta \cite{GorenAndreatta}. Soit $x\in\NNaif_{(2,2)}$ un point défini sur un corps $k$ algébriquement clos et $G$ le groupe $p$-divisible associé. D'après la Proposition 4.10 de \textit{loc cit} il existe une base $(e_1,e_2)$ de $\mathscr{E}(G)$ telle que la filtration de Hodge soit donnée par :
\begin{equation*}
\omega_G=\langle u^2 e_1,u^2 e_2\rangle
\end{equation*}
et que le Frobenius soit donné dans cette base par :
\begin{equation*}
F=\begin{pmatrix}
u^m &cu^2\\
u^2&0
\end{pmatrix}
\end{equation*}
où $m\geq 2$ et $c\in (\OK_L\otimes k)^{\times}$. On définit la filtration de Pappas Rapoport comme suit :
\begin{equation*}
\omega^{(1)}=\langle u^3 e_2\rangle ,\ \ \omega^{(2)}=\langle u^3 e_1,u^3 e_2\rangle ,\ \ \omega^{(3)}=\langle u^3 e_1 u^2 e_2\rangle  ,\ \ \omega^{(4)}=\langle u^2 e_1,u^2 e_2\rangle
\end{equation*}
Cela définit un point $\tilde{x}\in\SSplit$. Un simple calcul montre que pour cette filtration on a :
\begin{equation*}
\mathscr{F}^{(1)}:=F(u^{-1}(\omega^{(3)})^{(p)})=\langle cu^3e_2\rangle
\end{equation*}
En particulier $\omega^{(1)}=\mathscr{F}^{(1)}$ et par conséquent $m_1(\tilde{x})=0$. Il est également simple de voir que $m_2(\tilde{x})=m_3(\tilde{x})=m_4(\tilde{x})=0$.  Par conséquent $\tilde{x}\in  X_{(2,2)}^{(m_1,m_2,m_3,m_4)}$ et cette dernière strate est donc non vide.
\end{proof}

\begin{theorem}
Pour $L/\Q_p$ totalement ramifiée de degré $e=4$ la stratification :
\begin{equation*}
\SSplit=\coprod_{(\lambda,T)\in\mathscr{A}_{\mu_{\bullet}}}\SSplit_{(\lambda,T)}
\end{equation*}
est une bonne stratification où les relations d'adhérences sont données par la relation d'ordre naïve  sur $\mathscr{A}_{\mu_{\bullet}}$
\end{theorem}

\section{Conjecture}

\subsection{Conjecture} Une fois la stratification par le polygone de Hodge de $\NNaif$ réinterprétée via le morphisme lisse 
\begin{equation*}
\NNaif\rightarrow \mathrm{Hecke}=\big[L^+G\backslash \Gr\big]
\end{equation*}
il est naturel de stratifier le modèle de Pappas-Rapoport $\SSplit$ via le morphisme lisse
\begin{equation*}
\SSplit\rightarrow \mathrm{Hecke}_{\mu_{\bullet}}=\big[L^+G\backslash \widetilde{\Gr}_{\mu_{\bullet}}\big]
\end{equation*}
Le problème est alors de décrire les $L^+G$-orbites dans $\widetilde{\Gr}_{\mu_{\bullet}}$. Dans cet article nous nous sommes intéressé au morphisme de convolution 
\begin{equation*}
m:\widetilde{\Gr}_{\mu_{\bullet}}\rightarrow \Gr_{\leq |\mu_{\bullet}|}
\end{equation*}
Ce morphisme étant $L^+G$-équivariant les strates $(m^{-1}(\Gr_{\lambda}))_{\lambda\leq |\mu_{\bullet}|}$  sont des unions de $L^+G$-orbites dans $\widetilde{\Gr}_{\mu_{\bullet}}$. Si la preuve de la proposition \ref{StratGrConv} semble être \og à la main\fg et si elle ne se généralise pas à d'autres situations (voir \cite{BijHer2} Proposition 3.9) c'est parce qu'elle n'est pas naturelle : la stratification la plus naturelle est celle induite par la décomposition en orbites.

\begin{Conjecture}
Il existe un ensemble fini partiellement ordonné $(X,<)$ décrivant les $L^+G$-orbites dans $\widetilde{\Gr}_{\mu_{\bullet}}$. En d'autres termes on dispose d'un homéomorphisme
\begin{equation*}
|\mathrm{Hecke}_{\mu_{\bullet}}|\simeq X
\end{equation*}
où $X$ est muni de la topologie induite par $<$.
\end{Conjecture}
On consultera \cite{Wedhorn} (\S 2.1, \S 2.2) pour plus de détails concernant la stratification d'un espace en orbites sous l'action d'un groupe.

\subsection{Le cas de petites dimensions}

Dans \cite{BijHodge}, S.Bijakowski a répondu positivement à la conjecture ci-dessus dans le cas Unitaire et dans le Hilbert-Siegel lorsque $e\leq 3$ . Énonçons le résultat dans le cas Hilbert totalement ramifié d'indice de ramification $e=3$ 
\begin{theorem}[\cite{BijHodge}]
La classe d'isomorphisme d'une filtration de Pappas-Rapoport $(0\subset\omega^{(1)}\subset\omega^{(2)}\subset\omega^{(3)})$ est entièrement déterminée par les invariants $\mathrm{Hodge}(\omega^{(3)})$, $\mathrm{Hodge}(\omega^{(2)})$ et $\mathrm{Hodge}(\omega^{(3)}/\omega^{(1)})$. De plus ces invariants définissent une bonne stratification de $\SSplit$ où la relation d'ordre est la relation naïve : $(\lambda_1,\lambda_2,\lambda_3)\leq (\lambda_1',\lambda_2',\lambda_3')$ si $\lambda_i\leq\lambda_i'$ pour $i=1,2 ,3$.
\end{theorem}

Faisons quelques commentaires :
\begin{enumerate}
\item Dans le théorème ci-dessous, le fait que les classes d'isomorphismes induisent une bonne stratification est automatique : les classes d'isomorphismes de filtration de Pappas-Rapoport correspondent aux $L^+G$-orbites dans $\widetilde{\Gr}_{\mu_{\bullet}}$. Le résultat non trivial du théorème  ci-dessus concernant la stratification réside dans le calcul explicite des relations d'adhérences entres orbites.
\item Dans le cas Hilbert, c'est-à-dire dans le cas $G=\mathrm{GL}_2$ et $\mu_{\bullet}=(1,0)^e$ les gradués $\Lambda^{(i)}/\Lambda^{(i-1)}$ sont de dimension $1$ et on obtient donc les équivalences \footnote{On utilise la notation abusive $\mathrm{Hodge}(\Lambda^{(i)}/\Lambda^{(i-2)})=\mathrm{Inv}(\Lambda^{(i)},\Lambda^{(i-2)})$} :
\begin{equation*}
\mathrm{Hodge}(\Lambda^{(2)})=(3,3)\ \Leftrightarrow\ \mathrm{Hodge}(\Lambda^{(2)}/\Lambda^{(0)})=(2,0)\ \Leftrightarrow\ m_2\neq 0
\end{equation*}
\begin{equation*}
\mathrm{Hodge}(\Lambda^{(3)}/\Lambda^{(1)})=(2,0)\ \Leftrightarrow\ m_3\neq 0
\end{equation*}
En d'autres termes dans le cas $e=3$ la donnée $(\mathrm{Hodge}(\omega),m_3,m_2)$ détermine entièrement la classe d'isomorphisme. En fait on peut être plus précis : $(m_3,m_2)$ détermine $\mathrm{Hodge}(\omega)$ car dans le cas $e=3$ il y a que deux possibilités $(2,1)$ et $(3,0)$. 
\item Pour le cas $e=4$ nous avons vu que les strates $X_{(2,2)}^{(m_3)}$ et $X_{(3,1)}^{(m_3)}$ étaient non vides. Dans cette situation les invariants $(m_i)_{i}$ ne déterminent donc pas l'invariant de Hodge, et donc en particulier ils ne détectent pas la classe d'isomorphisme non plus.
\item Dans le cas Hilbert-Siegel, c'est-à-dire pour $G=\mathrm{GSp}_{2g}$, les invariants ci-dessous ne sont pas assez fins pour détecter la classe d'isomorphisme. En fait dans le cas général on a l'équivalence suivante :
\begin{equation*}
\mathrm{Hodge}(\Lambda^{(i)}/\Lambda^{(i-2)})=(2,0)^g\Leftrightarrow\ m_i\neq 0
\end{equation*}
En d'autres termes le lieu de non annulation de l'invariant $m_i$ coïncide avec la strate maximale de la stratification par l'invariant $\mathrm{Hodge}(\Lambda^{(i)}/\Lambda^{(i-2)})$. Ce dernier est donc le bon invariant à considérer en dimension supérieure. 
\end{enumerate}

\newpage

\bibliography{bibliographie}
\bibliographystyle{alpha}

\end{document}